\documentclass[a4paper, 11 pt , english]{article}

\usepackage[english]{babel}
\usepackage{amsthm,enumitem}
\usepackage[utf8]{inputenc}
\usepackage[T1]{fontenc}
\usepackage{fancybox}%
\usepackage{graphicx}
\usepackage{lmodern}
\usepackage[top = 3cm, bottom = 3 cm, right = 3cm, left = 3 cm]{geometry}
\usepackage{float}

\usepackage{subcaption}

\newcommand\smallO{
  \mathchoice
    {{\scriptstyle\mathcal{O}}}%
    {{\scriptstyle\mathcal{O}}}%
    {{\scriptscriptstyle\mathcal{O}}}%
    {\scalebox{.7}{$\scriptscriptstyle\mathcal{O}$}}%
  }

\usepackage[ruled,linesnumbered]{algorithm2e}

\usepackage{amsmath,amssymb}
\usepackage[hidelinks]{hyperref}
\usepackage{dsfont}
\usepackage{mathtools}
\usetagform{default}%
\newtagform{noparen}{}{}
\newcommand{\independent}{\perp\!\!\!\!\perp}
\newcommand{\eqreftag}[1]{\ref{#1}}
\def \E{\mathbb{E}}
\def \P{\mathbb{P}}
\def \Q{\mathbb{Q}}
\def \F{\mathbb{F}}
\def \R{\mathbb{R}}
\def \X{\mathbb{X}}
\def \L{\mathbb{L}}
\def\Cc{\mathcal{C}}
\def\Fc{\mathcal{F}}
\def\Lc{\mathcal{L}}
\def\Pc{\mathcal{P}}
\def\Kc{\mathcal{K}}
\def\Bc{\mathcal{B}}
\def\Xc{\mathcal{X}}
\def\Wc{\mathcal{W}}
\def\Yc{\mathcal{Y}}
\def\Zc{\mathcal{Z}}

\DeclareMathOperator{\Ker}{Ker}
\DeclareMathOperator{\Tan}{Tan}
\DeclareMathOperator*{\argmax}{arg\,max}

\def \AW{\mathcal{AW}}

\numberwithin{equation}{section}

\newtheorem{Theorem}{Theorem}[section]
\newtheorem{Definition}[Theorem]{Definition}
\newtheorem{Proposition}[Theorem]{Proposition}
\newtheorem{Assumption}{Assumption}[section]
\newtheorem{Lemma}[Theorem]{Lemma}
\newtheorem{Remark}[Theorem]{Remark}

\title{Projected McKean--Vlasov Dynamics for Entropic Weak Optimal Transport}
\author{Nathan Sauldubois
\footnote{Department of Finance and Risk Engineering, New York University. ns6982@nyu.edu }
        \and Xin Zhang
        \footnote{Department of Finance and Risk Engineering, New York University. xz1662@nyu.edu.}
        \thanks{X. Zhang is partially supported by the NSF Grant DMS-2508556.}
        }

\begin{document}

\maketitle

\begin{abstract}
Unlike classical optimal transport, weak transport costs depend nonlinearly on
the conditional law of couplings. This feature is essential in problems involving barycenter, conditional moments, and martingale-type constraints. Meanwhile, such conditional dependence makes ordinary Wasserstein geometry insufficient and calls instead for an adapted Wasserstein viewpoint. In this paper, we investigate the entropy-regularized weak optimal transport via gradient flows in adapted Wasserstein space.

We derive, from the formal tangent structure of adapted Wasserstein
space and the projection onto the set of couplings with prescribed marginals,
a coupled McKean--Vlasov SDE. A novel and subtle term is a projection that, at each $Y$-location, averages a weak-transport force that already depends on the conditional law of $Y$ given $X$, thereby preserving marginals while retaining the nonlinear weak-transport structure.

Under mild integrability and regularity assumptions, we prove weak existence and uniqueness in law for this
projected McKean--Vlasov equation. We then prove that the flow
converges, in the adapted Wasserstein topology, to the unique minimizer of the
entropic weak optimal transport problem. We also describe a particle approximation and
illustrate the dynamics on optimal transport and martingale optimal transport examples.
\end{abstract}

\section{Introduction}

Optimal transport provides a powerful framework for comparing probability measures. Let $\mu$, $\nu$ be two probability measures on a complete metric space $\Xc$ and $\R^d$ respectively, and let $\bar c: \Xc \times \R^d \to \R$ be a transport cost. Classical optimal transport minimizes
\[
    \int_{\Xc\times\R^d} \bar c(x,y)\,\pi(dx,dy)
\]
over all couplings \(\pi\in\Pi(\mu,\nu)\). Weak optimal transport replaces this
linear dependence by a dependence on conditional laws. If
\(\pi(dx,dy)=\mu(dx)\pi_x(dy)\), then a weak transport cost has the form
\[
    \int_{\Xc} c(x,\pi_x)\,\mu(dx),
\]
where \(c(x,\cdot)\) is a functional on probability measures on \(\R^d\). Introduced by \cite{gozlan2017kantorovich},
this
extension is useful whenever the value of a transport plan is determined not
only by pointwise pairs \((x,y)\), but also by local distributional information
such as barycenters, conditional moments, or martingale-type constraints.
Classical optimal transport is recovered from the special case
\(c(x,\rho)=\int \bar c(x,y)\rho(dy)\), whereas genuinely weak costs lead to a
nonlinear dependence on the transport kernel \(x\mapsto\pi_x\).

The object of this paper is the entropically regularized weak optimal transport
problem
\begin{align}\label{eq:WOT}
    \inf_{\pi\in\Pi(\mu,\nu)}
    \left\{
    \int_{\Xc} c(x,\pi_x)\,\mu(dx)
    +
    \varepsilon H(\pi\,\|\,\mu\otimes\nu)
    \right\},
\end{align}
where \(\varepsilon>0\) and \(H(\cdot\|\cdot)\) denotes relative entropy. We write
\begin{align}\label{eq:objective}
    J(\pi)
    :=
    \int_{\Xc} c(x,\pi_x)\,\mu(dx)
    +
    \varepsilon H(\pi\,\|\,\mu\otimes\nu).
\end{align}
The entropy term plays the same regularizing role as in entropic optimal
transport: it selects absolutely continuous couplings, improves compactness and
strict convexity properties, and introduces a diffusion mechanism in the
dynamics. Since every \(\pi\in\Pi(\mu,\nu)\) has first marginal \(\mu\), the
entropy also admits the disintegrated form
\[
    H(\pi\,\|\,\mu\otimes\nu)
    =
    \int_{\Xc} H(\pi_x\,\|\,\nu)\,\mu(dx),
\]
whenever the left-hand side is finite. Thus \eqref{eq:WOT} can be viewed as a weak optimal transport problem.

Inspired by \cite{conforti2023projected}, the purpose of this paper is to construct and analyze a continuous-time dynamic that converges to a minimizer of \eqref{eq:WOT} from the perspective of gradient flows. To analyze the gradient of $J$ on $\Pi(\mu,\nu)$, the appropriate geometric framework is the adapted Wasserstein geometry, rather than the ordinary Wasserstein geometry on \(\Xc \times \R^d\). This choice is structural: the functional \(J\) depends on the kernel \(x \mapsto \pi_x\), and is therefore not continuous with respect to the classical Wasserstein topology on \(\Pc_2(\Xc \times \R^d)\). The adapted Wasserstein topology is thus the natural setting on which one can analyze stability and gradient flows for weak transport costs.

This paper gives a
formal derivation of a gradient flow on adapted Wasserstein space, showing
that the steepest-descent dynamics of the entropic weak transport functional is
a new projected coupled McKean--Vlasov SDE. The intriguing term is a projection that, at each $Y$-location, averages a weak-transport force that already depends on the conditional law of $Y$ given $X$, thereby preserving the second marginal while retaining the nonlinear weak-transport structure. By proving the well-posedness of this dynamic
and convergence to the unique entropic weak transport optimizer, the paper
connects the static theory of weak optimal transport with a continuous-time
gradient-flow perspective and provides a basis for particle algorithms for
weak optimal transport.
\subsection{Main results}

We first state the main dynamical object and the principal results. Assume that
\(c\) admits a linear functional derivative in its measure argument, and set
\begin{equation}\label{eq:def-chat}
    \hat c(x,y,\rho)
    :=
    \partial_y\delta_m c(x,\rho)(y).
\end{equation}
For \(\pi\in\Pi(\mu,\nu)\), define the conditional projection term
\begin{equation*}
    \tilde c^{\pi}(y)
    :=
    \E^{\pi}\!\left[
        \hat c(X,Y,\pi_X)\,\middle|\,Y=y
    \right]
    =
    \int_{\Xc}\hat c(x,y,\pi_x)\,\pi^y(dx),
\end{equation*}
where \(\pi^y\) is a regular conditional law of \(X\) given \(Y=y\). Let \(\nu(dy)=e^{-V(y)}dy\).

The
projected adapted-Wasserstein gradient flow of $J$, derived in Section~\ref{sec:gradient-flow},
is the McKean--Vlasov stochastic differential equation
\begin{equation}\label{flow:grad}
\begin{cases}
    dY_t
    =
    -\left(
        \hat c(X,Y_t,(\pi_t)_X)
        -
        \tilde c^{\pi_t}(Y_t)
        +
        \varepsilon\nabla V(Y_t)
    \right)dt
    +
    \sqrt{2\varepsilon}\,dB_t, \\
    \pi_t=\Lc(X,Y_t), \quad (X,Y_0)\sim\pi_0.
    \end{cases}
\end{equation}
Here \(X\) is frozen in time, and \(B\) is a
Brownian motion independent of the initial condition. The term
\(\hat c(X,Y_t,(\pi_t)_X)\) is the weak-transport force, and the conditional average
\(\tilde c^{\pi_t}(Y_t)\) is the projection term which enforces preservation of
the second marginal. The drift \(\varepsilon\nabla V\) and the noise
\(\sqrt{2\varepsilon}\,dB_t\) represent the entropy relative to the reference
measure \(\nu\).

The cost \(c\), the marginals \(\mu\) and \(\nu\), and the reference measure are
fixed throughout the paper. We work under the following standing assumptions.

\begin{Assumption}[Marginals and initial condition]\label{ass:marginals}
Let \((\Xc,d_{\Xc})\) be a complete metric space and let \(\mu\in\Pc(\Xc)\). Let $\nu \in \Pc(\R^d)$ with
\(
    \nu(dy)=e^{-V(y)}\,dy
\)
for some \(V\in C^2(\R^d; \R)\). We assume that
\begin{enumerate}[label=(\roman*), ref=(\roman*)]
    \item there exists \(\lambda\in\R\) such that
    \[
        D^2V(y)\succeq \lambda I_d,
        \qquad y\in\R^d;
    \]
    \item $\mu \in \Pc_p(\mathcal{X})$, \(\nabla V\in L^p(\nu)\) for some \(p>d+2\).

    \item \label{ass:initial}
    The initial law is a coupling
    $\pi_0\in\Pi(\mu,\nu)$. Writing
    $\pi_0(dx,dy)=\mu(dx)\,\pi^0_x(dy)$, its conditional laws satisfy
    \[
        \lim_{\delta\downarrow0}
        \sup_{\substack{
            \gamma\in\Pi(\mu,\mu)\\
            \E^\gamma[d_\Xc(X,X')]\le\delta
        }}
        \E^\gamma\!\left[
            H\bigl(\pi^0_X\mid\pi^0_{X'}\bigr)
        \right]=0.
    \]
\end{enumerate}
\end{Assumption}

\begin{Assumption}[Cost]\label{ass:cost}
The cost
\(
    c:\Xc\times\Pc_2(\R^d)\to\R
\)
is measurable, convex in its second argument, and admits a linear functional
derivative \(\delta_m c\). With \(\hat c\) defined by \eqref{eq:def-chat}, we assume that
\begin{enumerate}[label=(\roman*), ref=(\roman*)]
    \item \emph{Integrability bound.} There exist measurable functions
    \(a:\Xc\to\R\), \(b:\R^d\to\R\), and \(h:\R_+\to\R_+\) such that
    \(a\in L^1(\mu)\), \(b\in L^1(\nu)\), and, whenever \(\rho\ll\nu\),
    \[
        c(x,\rho)
        \leq
        a(x)
        +
        \int_{\R^d} b(y)\,\rho(dy)
        +
        \int_{\R^d}
        h\!\left(\frac{d\rho}{d\nu}(y)\right)\nu(dy).
    \]

    \item \label{ass:cost:lipschitz}
    \emph{Lipschitz continuity in $(x,\rho)$.}
    The map $\hat c$ admits a measurable extension to
    $\Xc\times\R^d\times\Pc_1(\R^d)$, still denoted by $\hat c$.
    There exists $L_c>0$ such that, for all $x,x'\in\Xc$, $y\in\R^d$,
    and $\rho,\rho'\in\Pc_1(\R^d)$,
    \[
        |\hat c(x,y,\rho)-\hat c(x',y,\rho')|
        \le L_c\bigl(d_\Xc(x,x')+\Wc_1(\rho,\rho')\bigr).
    \]

    \item \label{ass:cost:exp} \emph{Sub-Gaussian condition.} There exists \(r^*>0\) such
    that
    \[
        \sup_{\pi\in\Pi(\mu,\nu)}
        \E^{\pi}
        \left[
            \exp\left(
                r^*
                \left|
                    \hat c(X,Y,\pi_X)
                \right|^2
            \right)
        \right]
        <\infty.
    \]

    \item The divergence \(\nabla_y\!\cdot\hat c(x,y,\rho)\) is bounded uniformly in
    \((x,y,\rho)\).
\end{enumerate}
\end{Assumption}

The assumptions are close in spirit to those of
\cite{conforti2023projected}. In particular, the exponential condition in
Assumption~\ref{ass:cost}\ref{ass:cost:exp} follows from an \emph{at most linear} growth
bound: for some $C>0$ and $x_0\in\Xc$,
\[
    |\hat c(x,y,\rho)|
    \le C\left(1+d_\Xc(x,x_0)+|y|
        +\int_{\R^d}|z|\,\rho(dz)\right),
\]
provided that, for some $\alpha>0$,
\[
    \int_\Xc e^{\alpha d_\Xc(x,x_0)^2}\,\mu(dx)
    +\int_{\R^d}e^{\alpha|y|^2}\,\nu(dy)<\infty.
\]
The Lipschitz estimate in
Assumption~\ref{ass:cost}\ref{ass:cost:lipschitz} is imposed separately.

The initial-kernel condition in Assumption~\ref{ass:marginals}\ref{ass:initial} holds
automatically for $\pi_0=\mu\otimes\nu$, since then $\pi^0_x=\nu$.
More generally, it follows from a bound
\[
    H(\pi^0_x\mid\pi^0_{x'})\le L_0\,d_\Xc(x,x')^2,
    \qquad x,x'\in\Xc,
\]
for a version of the initial kernel and some $L_0<\infty$:
the $p$-moment assumption, with $p>2$, makes
$d_\Xc(X,X')^2$ uniformly integrable over $\gamma\in\Pi(\mu,\mu)$.

\subsubsection{Well-posedness}
Our first main result establishes the well-posedness of the McKean--Vlasov SDE \eqref{flow:grad}. More precisely, we prove the well-posedness for a general equation \eqref{eq:MV-a} in Theorem~\ref{prop:existence-general-unbounded} and Theorem~\ref{prop:uniqueness-general}.
The additional requirements on $\hat c$ and $\pi_0$
are precisely those of Assumption~\ref{ass:a-existence-solution} for the
coefficient $a=\hat c$ when $\varepsilon=1$. For general $\varepsilon>0$,
the time change $s=\varepsilon t$ gives the same equation with
$a=\varepsilon^{-1}\hat c$.

\begin{Theorem}\label{thm:well-posedness}
Under Assumptions~\ref{ass:marginals}
and~\ref{ass:cost}, the McKean--Vlasov SDE~\eqref{flow:grad}, with
initial law $\pi_0$, admits a weak solution that is unique in law.
Moreover, $\pi_t=\Lc(X,Y_t)\in\Pi(\mu,\nu)$ for every $t\ge0$.
\end{Theorem}

\paragraph{Main challenges in proof.}
The proof of the well-posedness of \eqref{flow:grad} is more challenging than its counterpart in \cite{conforti2023projected}.
There are two main reasons for this. The first, and most obvious one, is the additional dependence of the cost $c$ on the conditional law with respect to $X$.
The second reason is that, since there is no diffusion in the $x$-variable, elliptic regularity is harder to exploit than in the existence proof of \cite{conforti2023projected}.
In their paper, elliptic regularity yields compactness in total variation, which in turn ensures that conditioning with respect to $X_t$ or $Y_t$ is a continuous operation at the level of laws.
This is not available here, at least not directly, so we need to treat conditioning with respect to $X$ and conditioning with respect to $Y_t$ differently.

For the conditioning with respect to $Y_t$, we observe in Proposition \ref{prop:estim} that, if \eqref{flow:grad} admits a weak solution, then, for every function $\phi:\Xc\to\R$, the conditional expectation
\[
    \E[\phi(X)\mid Y_t]
\]
seen as a Borel-measurable function of $(t,y)$, enjoys some elliptic regularity.
From this regularity, we derive Sobolev estimates which, handled carefully, are sufficient for our purposes, since only linear functionals of $\Lc(X\mid Y_t)$ are involved.

The conditioning with respect to $X$ is more delicate, because the drift has a nonlinear dependence on $\Lc(Y_t\mid X)$.
The treatment of this nonlinearity relies on compactness results on the space
\[
    \bigl(\Pc_2(\Xc\times\R^d),\AW_2\bigr),
\]
obtained in \cite{eder2019compactness}.
It is worth noting that, for these arguments to work, it is more convenient to consider the Fokker--Planck equation solved by the flow of marginal laws associated with the McKean--Vlasov equation.
This causes no difficulty thanks to the mimicking theorem, Theorem~1.3 in \cite{lacker2022superposition}, whose assumptions are satisfied in our setting, as shown in Proposition \ref{Prop:mimick}.

\subsubsection{Long-time convergence}

The second main result concerns convergence to the minimizer of \eqref{eq:WOT}.

\begin{Theorem}\label{thm:convergence}
Assume $\pi_0=\mu\otimes\nu$.
Let \(\pi^*\) be the unique minimizer of \eqref{eq:WOT}. Under Assumption~\ref{ass:marginals} and Assumption~\ref{ass:cost}, the solution \((\pi_t)_{t\ge0}\) of
\eqref{flow:grad} converges to $\pi^*$  in the adapted Wasserstein topology as $t \to \infty$.
\end{Theorem}

\paragraph{Main challenges in proof.} Along solutions to \eqref{flow:grad}, the projected gradient-flow structure
formally yields the dissipation identity
\begin{align}\label{eq:intro-dissipation}
    \frac{d}{dt}J(\pi_t)
    =
    -
    \E^{\pi_t}\!\left[
        \left|
        \hat c(X,Y,(\pi_t)_X)
        -
        \tilde c^{\pi_t}(Y)
        +
        \varepsilon\nabla_y
        \log\frac{d\pi_t}{d(\mu\otimes\nu)}(X,Y)
        \right|^2
    \right].
\end{align}
The squared norm in \eqref{eq:intro-dissipation} is the weak analogue of a
relative Fisher information. To prove convergence, we adopt the usual compactness argument. The first difficulty is that $\Pc_2(\Xc \times \R^d)$ is not complete in the adapted Wasserstein space. To address this issue, in \eqref{eq:WOTfisher} we extend the right-hand side of \eqref{eq:intro-dissipation} to the set of filter processes which is the completion of $\Pc_2(\Xc \times \R^d)$, and hence there exists a limit point by a Prokhorov-type result. The second challenge is the lower semicontinuity of \eqref{eq:intro-dissipation}, in particular for the term $\lVert \tilde c^{\pi}(Y)\rVert_{L^2(\nu)}$ because $\tilde c$ involves the conditional expectation with respect to $Y$. To overcome the difficulty, we prove a uniform bound for $\lVert \tilde c^{\pi}(Y)\rVert_{L^2(\nu)}$ using a uniform estimate on the Fisher information. The third difficulty is to characterize the
zero-dissipation set: from \eqref{eq:intro-dissipation} being $0$, we recover the Euler--Lagrange
condition for the entropic weak optimal transport problem. We show that the
limiting conditional densities are sufficiently regular and positive, and then
use convexity together with the marginal constraints to conclude that the limit is the unique optimizer.

\begin{Remark}
    For the quadratic cost, i.e. $c(x,\pi_x)=\int |x-y|^2 \, \pi_x(dy)$, exponential convergence of $\pi_t \to \pi^*$ {in terms of relative entropy}
  can be proved under the assumptions of \cite[Theorem 1.4]{conforti2023projected}. The derivation of convergence rates for general cost functions is left for future research.
\end{Remark}

\subsection{Gradient-flow derivation of the McKean--Vlasov SDE}

We now explain the gradient-flow derivation of \eqref{flow:grad}, following the
heuristic calculation in Section~\ref{sec:gradient-flow} with $\Xc=\R^d$. The starting point is
the formal tangent structure of the adapted Wasserstein space. Section~\ref{sec:gradient-flow}
motivates this tangent structure by a small-radius sensitivity argument: in
finite dimensions, optimizing a smooth function over a ball of radius \(r\)
recovers the normalized gradient direction at first order; in the adapted
Wasserstein space, the analogous distributionally robust problem over an
\(\AW_2\)-ball identifies the adapted gradient of the first variation. For a
smooth test function \(\varphi\) and a law \(\pi\) on \(\R^d\times\R^d\), this adapted gradient is
\begin{equation}\label{eq:intro-adapted-gradient}
    \nabla^{\rm ad}_{\pi}\varphi(x,y)
    =
    \left(
        \E^{\pi}[\partial_x\varphi(X,Y)\,|\, X=x],
        \partial_y\varphi(x,y)
    \right).
\end{equation}
The first component is averaged conditionally on \(X\), as first-stage
motions may use only first-stage information. The second component is the usual
\(y\)-derivative, as the second-stage motion may depend on the full pair
\((x,y)\). This is the infinitesimal version of the adapted transport geometry:
ordinary Wasserstein gradients are replaced by adapted gradients.

The projection onto the tangent space of $\Pi(\mu, \nu)$ has been derived in \cite{conforti2023projected}.
We recover it from Distributionally Robust Optimization sensitivity argument.
Section~\ref{sec:gradient-flow}
identifies \(\Pi(\mu,\nu)\) as the orthogonal complement of the linear space
\[
    \mathcal V
    =
    \left\{h_1\oplus h_2:\ \mu(h_1)=0,\ \nu(h_2)=0, \ h_1, h_2 \in C_c^{\infty}(\R^d; \R) \right\},
\]
and interprets the tangent space to \(\Pi(\mu,\nu)\) through first-order
sensitivity of small adapted-Wasserstein balls. In one dimension, this gives the
explicit formula for the projection to the tangent space of $\Pi(\mu,\nu)$, denoted by $P^{\pi}$.
\begin{equation}\label{eq:intro-projection}
    P^{\pi}(u_1,u_2)(x,y)
    =
    \left(
        0,
        u_2(x,y)-\E^{\pi}[u_2(X,Y)\mid Y=y]
    \right),
\end{equation}
where $(u_1,u_2)$ is a tangent vector in adapted Wasserstein space.
The same projection is then used for the formal derivation in general dimension.
Formula \eqref{eq:intro-projection} has two effects. It freezes the first
coordinate, and it removes from the second velocity the component depending only
on \(Y\). This is precisely the centering mechanism that keeps the second
marginal fixed.

With this projection, the adapted descent of \(J\) on \(\Pi(\mu,\nu)\)
is written formally as
\begin{equation}\label{eq:intro-section3-gradient-flow}
    \partial_t\pi_t
    =
    \nabla\!\cdot\left(
        \pi_t\,
        P^{\pi_t}\!\big(
            \nabla^{\rm ad}_{\pi_t}\delta_m J(\pi_t)
        \big)
    \right).
\end{equation}
After some computations, we obtain
\[
\begin{aligned}
    P^{\pi}\big(\nabla^{\rm ad}_{\pi}\delta_m J(\pi)\big)(x,y)
    =
    \Big(&0,
    \hat c(x,y,\pi_x)
    +
    \varepsilon\partial_y\log\rho(x,y)
    \\
    &-
    \int_{\Xc}\hat c(z,y,\pi_z)\,\pi^y(dz)
    -
    \varepsilon\int_{\Xc}\partial_y\log\rho(z,y)\,\pi^y(dz)
    \Big).
\end{aligned}
\]
As the second marginal of \(\pi\) is \(\nu\), by a computation done in Section 2.3 of \cite{conforti2023projected} and recalled in Section~\ref{sec:gradient-flow}, we obtain
\[
    \int_{\Xc}\partial_y\log\rho(z,y)\,\pi^y(dz)
    =
    \partial_y\log\int_{\Xc}\rho(z,y)\,\mu(dz)
    =
    -\nabla V(y).
\]
Therefore the projected adapted gradient has second component
\[
    \hat c(x,y,\pi_x)
    -
    \tilde c^{\pi}(y)
    +
    \varepsilon\partial_y\log\rho(x,y)
    +
    \varepsilon\nabla V(y)
    =
    \hat c(x,y,\pi_x)
    -
    \tilde c^{\pi}(y)
    +
    \varepsilon\nabla_y\log\frac{d\pi}{d(\mu\otimes\nu)}(x,y).
\]
Thus \eqref{eq:intro-section3-gradient-flow} reduces to the
Fokker--Planck equation
\begin{equation}\label{eq:intro-pde-gradient}
    \partial_t\pi_t
    =
    \nabla_y\!\cdot\left(
        \pi_t
        \left[
            \hat c(x,y,(\pi_t)_x)
            -
            \tilde c^{\pi_t}(y)
            +
            \varepsilon\nabla_y\log\frac{d\pi_t}{d(\mu\otimes\nu)}(x,y)
        \right]
    \right).
\end{equation}
Finally, writing \(\rho_t\) for the density of \(\pi_t\) with respect to
\(\mu(dx)dy\), the entropy term has the density-level form
\[
    \nabla_y\!\cdot\left(
        \varepsilon\rho_t
        \nabla_y\log\frac{d\pi_t}{d(\mu\otimes\nu)}
    \right)
    =
    \varepsilon\Delta_y\rho_t
    +
    \nabla_y\!\cdot(\varepsilon\rho_t\nabla V),
\]
which is the Fokker--Planck operator associated with the drift
\(-\varepsilon\nabla V\) and diffusion coefficient \(\sqrt{2\varepsilon}\). Hence
\eqref{eq:intro-pde-gradient} is exactly the equation of \eqref{flow:grad}.
This calculation also explains the two conditional directions in the SDE: the
weak-transport force $\hat c$ depends on the forward kernel \(x\mapsto(\pi_t)_x=\Lc(Y_t\mid X=x)\),
while the projection depends on the kernel \(y\mapsto\pi_t^y=\Lc(X\mid
Y_t=y)\).

\subsection{Related literature}

This work lies at the intersection of weak optimal transport, entropy
regularization, adapted transport, and Wasserstein gradient flows. Classical
optimal transport and its Wasserstein geometry are standard tools in analysis,
probability, and applied mathematics; see, for instance,
\cite{villani2003topics,villani2009optimal,santambrogio20151}. Weak optimal
transport enlarges the Kantorovich problem by allowing nonlinear costs of
conditional laws. The general weak-cost duality theory was initiated in
\cite{gozlan2017kantorovich}, and existence, duality, and cyclical monotonicity
for weak transport costs were developed in \cite{backhoff2019existence}.
Related structural and application-oriented developments include the
Brenier--Strassen theory of \cite{gozlan2020mixture}, the Hopf--Lax and weak
transfer results of \cite{shu2020hopf}, and the survey of applications in
\cite{backhoff2022applications}. The static existence, duality, and optimizer
representation used in the present paper are based on the fundamental theorem of
weak optimal transport \cite{BPRS25}.

Weak transport is also closely connected with martingale optimal transport and,
more generally, with transport problems under conditional moment constraints.
Martingale optimal transport originated in model-independent finance
\cite{beiglbock2013model,henrylabordere2017model}.
Complete duality on the line and the multiperiod
theory, including the structure of left-monotone transports, are developed in
\cite{beiglbock2017complete,nutz2020multiperiod}.
Benamou--Brenier-type and continuous-time
formulations of martingale transport are studied in
\cite{huesmann2019benamou,martingaleBenamouBrenierProb2020,backhoff2026bridging},
the latter relating them to weak optimal transport duality.
Stability with respect to the marginals is established for
martingale and weak transport in \cite{backhoff2022stability} and for weak
martingale optimal transport in \cite{Beiglbock2021StabilityOT}.
Martingale Schr\"odinger bridges and entropic
martingale transport are analyzed in
\cite{nutz2024martingale,chen2024convergence,backhoff2026bridging}, with a link
to exponential utility maximization over semistatic portfolios in
\cite{nutz2023semistatic}; weak transport with moment constraints and entropic
regularization is considered in \cite{carlier2025weak}.

\begingroup
\emergencystretch=1em

The entropy term in \eqref{eq:WOT} connects the paper with the
Schr\"odinger/Sinkhorn viewpoint on transport, going back to Schr\"odinger's
problem \cite{schrodinger1931umkehrung}; modern accounts of the Schr\"odinger
problem and its relation with optimal transport include
\cite{leonard2014survey,mikami2008optimal,chen2016relation}.
Existence, duality and the theory of Sinkhorn's algorithm
for classical entropic optimal transport are developed in
\cite{nutz2021introduction}.
Stability of entropic optimal couplings and
static Schr\"odinger bridges with respect to the marginals and the cost is
proved in \cite{ghosal2022stability}; the behavior of Schr\"odinger potentials
as the regularization vanishes is studied in \cite{nutz2022potentials}, and
their stability under varying marginals, with applications to the convergence
of Sinkhorn's algorithm, in \cite{nutz2023sinkhorn}.
Divergence-regularized transport on path space
with weak terminal constraints is considered in \cite{hernandez2025marginal}.
On the computational side, entropic regularization is strongly associated with Sinkhorn
scaling \cite{cuturi2013sinkhorn}; see also the monograph
\cite{peyre2019computational}. In the linear-cost case, \eqref{eq:WOT} reduces
to the usual entropic optimal transport problem. In the weak case, the entropy
still gives a Gibbs-type representation, but it is implicit because
\(\delta_m c(x,\pi_x)\) depends on the unknown conditional law. Our contribution
is not a Sinkhorn-type discrete iteration, but a continuous-time projected
diffusion whose invariant limit is the entropic weak transport optimizer.
\par\endgroup

The geometric framework of our paper is the adapted Wasserstein distance. This distance is
closely related to the nested distance introduced in multistage stochastic
optimization \cite{pflug2012distance,pflug2014multistage}, and to causal and
bicausal transport
\cite{lassalle2013causal,backhoff2017causal,acciaio2020causal}. We use the
formulation of the Wasserstein space of stochastic processes developed in
\cite{BBP26}, where convergence is expressed through information maps and hence
controls conditional laws.
Compactness and topological properties of adapted weak
convergence are studied in
\cite{eder2019compactness,pammer2024note,BBPSZ25,BPSZ23}.
Estimation and stability of adapted Wasserstein distances,
including convergence rates and quantitative error bounds for adapted empirical
measures, are developed in
\cite{backhoff2020adapted,backhoff2022estimating,acciaio2024empirical}.
Explicit adapted transport formulas for
Gaussian processes, also with entropic regularization, are obtained in
\cite{acciaio2025entropicgaussians}, and the corresponding adapted
Bures--Wasserstein geometry, including tangent cones and exponential maps, is
studied in \cite{acciaio2026adaptedbures}; absolutely continuous curves of
stochastic processes are considered in \cite{acciaio2025absolutely}.
Computational and entropic methods for adapted
optimal transport are treated in \cite{eckstein2024computational}, and duality
and dual attainment results appear in \cite{krsek2025dual}.
Finally, adapted and causal distances underlie
sensitivity analysis for distributionally robust optimization
\cite{bartl2021sensitivity,bartl2022sensitivity,jiang2024sensitivity,sauldubois2024first,sauldubois_modelrisk_statichedging_cdro}.

A related geometric viewpoint treats conditional distributions as random
probability measures. Lisini's superposition principle
\cite{lisini2007characterization} represents Wasserstein curves with finite
$p$-energy, $p>1$, by measures on trajectories; Stepanov and Trevisan
\cite{stepanov2017superposition} develop this connection for continuity
equations on metric spaces. Pinzi and Savar\'e \cite{pinzi2025nested} establish
a nested superposition principle and a differential calculus on the space of
laws of random probability measures, including a description of its tangent
bundle. These results provide a broader setting for the conditional-law
geometry considered here. Our focus is the projected diffusive evolution
associated with a nonlinear weak transport cost and prescribed marginals.

The gradient-flow perspective builds on the classical interpretation of
Fokker--Planck equations as Wasserstein gradient flows, beginning with the
Jordan--Kinderlehrer--Otto scheme \cite{jordan1998variational} and developed in
the metric-space theory of Ambrosio, Gigli, and Savar\'e
\cite{ambrosio2008gradient}; see also Otto's formal Riemannian calculus
\cite{otto2001geometry} and the overview \cite{santambrogio2017euclidean}. The
closest antecedent for the dynamical part is the projected Langevin dynamics of
\cite{conforti2023projected}. In that work, the authors construct a diffusion on
the set \(\Pi(\mu,\nu)\) for the classical entropic optimal transport problem;
conditional expectation terms in the drift keep the marginals fixed, and the
long-time limit is the entropic optimal transport optimizer. The present paper
considers this projected-flow philosophy in a weak-transport direction. First, the
cost is nonlinear in the conditional law, so the drift has a McKean--Vlasov
dependence on \(\Lc(Y_t | X)\). Second, only the \(Y\)-coordinate diffuses,
while \(X\) remains fixed as the conditioning variable. Third, the convergence
analysis is carried out in the adapted topology, because the limiting object is
determined by conditional kernels and not merely by the joint law in the ordinary
Wasserstein sense. The mimicking and superposition ideas used in the
well-posedness proof are in the spirit of
\cite{brunick2013mimicking,lacker2022superposition}. For the relation between
Fokker--Planck equations and martingale problems under low regularity, see also
\cite{trevisan2016wellposed}.

\subsection{Organization of paper}

The rest of the paper is organized as follows. Section~\ref{sec:notation}
collects notation and background on functional derivatives, entropy and Fisher
information, weak optimal transport duality, and the adapted Wasserstein
distance. Section~\ref{sec:gradient-flow} formally derives the projected
adapted-Wasserstein gradient flow and obtains \eqref{flow:grad}. Section~\ref{sec:numerics} describes the particle approximation and its applications to optimal
transport and martingale optimal transport examples. Section~\ref{sec:wellposed} and Section~\ref{sec:convergence}
prove Theorem~\ref{thm:well-posedness} and
Theorem~\ref{thm:convergence} respectively.

\section{Notation and Preliminaries}
\label{sec:notation}

\subsection{General notation}

We write $|\cdot|$ for the Euclidean norm on $\mathbb{R}^d$, $B(x,R)$ for the open
ball of radius $R$ centered at $x$, and $\mathrm{L}$ for the Lebesgue measure.
Throughout the paper, $\Xc$ denotes a complete metric space.

For a complete metric space $\mathcal{Z}$, $\mathcal{P}(\mathcal{Z})$ denotes the set of
Borel probability measures on $\mathcal{Z}$, and for $p\geq 1$,
$\mathcal{P}_p(\mathcal{Z})$ denotes the subset of measures with finite $p$-th
moment, equipped with the $p$-Wasserstein distance $\Wc_p$.

For $\mu\in\mathcal{P}(\Xc)$ and $\nu\in\mathcal{P}(\mathbb{R}^d)$,
$\Pi(\mu,\nu)$ denotes the set of couplings, i.e.\ measures
$\pi\in\mathcal{P}(\Xc\times\mathbb{R}^d)$ with first marginal $\mu$ and
second marginal $\nu$. For $\pi\in\mathcal{P}(\Xc\times\mathbb{R}^d)$, we
denote by $\pi_1$, $\pi_2$ the first and second marginals of $\pi$ respectively, by $\{\pi_x\}_{x\in\Xc}$ the disintegration of $\pi$ with respect
to its first marginal, and by $\{\pi^y\}_{y\in\mathbb{R}^d}$ the disintegration
with respect to its second marginal.
Throughout this paper, once existence has been proved, we will not distinguish between a measure and its density with respect to Lebesgue measure.

On a probability space $(\Omega,\mathcal{F},\P)$ with sub-$\sigma$-algebra
$\mathcal{G}\subseteq\mathcal{F}$, $\mathbb{E}^{\P}[\,\cdot\,|\mathcal{G}]$ denotes
the conditional expectation. When $\mathcal{G}=\sigma(Y)$ for a random variable
$Y$, we write $\mathbb{E}^{\P}_{Y}[\,\cdot\,]$, and identify it with a Borel-measurable function of $Y$.

For a complete metric space $\mathcal{Z}$, we write $C_b(\mathcal{Z}; \R)$ for bounded continuous functions on $\mathcal{Z}$.
$C_c^\infty(\mathbb{R}^d; \R)$ denotes the set of smooth, compactly supported functions on
$\mathbb{R}^d$. Take
\[
C^{1}_{\mathrm{pol},2}(\mathbb{R}^{2d}; \R)
:=
\Big\{ f\in C^{1}(\mathbb{R}^{2d}; \R )\;:\;\exists\,K>0\ \text{s.t.}\
\|\nabla f(x)\|\le K\,(1+\|x\|),\ \forall x\in\mathbb{R}^{2d}
\Big\}.
\]
The $2$ being here since having a linear growth in the gradient means a quadratic growth for the function. For any vector space $\mathcal{V} \subset C^{1}_{\mathrm{pol},2}(\mathbb{R}^{2d}; \R)$, define
\begin{equation}\label{eqdef:orthodual}
\mathcal{V}^{\perp} := \left\{ \pi \in \Pc_2(\R^{2d}) \,\, \Big| \,\,  \int_{\R^{2d}} f(x, y) \, \pi(dx, dy) =0 \,\, \text{for all } f \in \mathcal{V} \right\} \subset \Pc_2(\R^{2d}) .
\end{equation}

\subsection{Linear functional derivatives}
\begin{Definition}
A map $F:\mathcal{P}_2(\mathbb{R}^d)\to\mathbb{R}$ admits a
\emph{linear functional derivative} $\delta_m F$ if
\begin{enumerate}[label=(\roman*), ref=(\roman*)]
    \item $\delta_m F:\mathcal{P}_2(\mathbb{R}^d)\times\mathbb{R}^d\to\mathbb{R}$ is continuous, and for every $R>0$ there exists $C_R>0$ such that whenever $\int |x|^2 \, \rho(dx) \leq R$
    \begin{align*}
       | \delta_m F(\rho, y)| \leq C_R(1+|y|^2).
    \end{align*}
\item  It satisfies the normalization $\int\delta_m F(\rho)(y)\rho(dy)=0$, \, $\forall \, \rho \in \Pc_2(\R^d)$.
\item For any $\rho,\rho' \in \Pc(\R^d)$, it holds that
\[
F(\rho')-F(\rho)
\;=\;
\int_0^1\!\!\int \delta_m F\big((1-t)\rho+t\rho'\big)(y)\,(\rho'-\rho)(dy)\,dt.
\]
\end{enumerate}
\end{Definition}

\subsection{Relative entropy and Fisher information}
\begin{Definition}
Suppose $\mathcal{Y}$ is a Polish space. For $\rho,\sigma\in\mathcal{P}(\mathcal{Y})$, the \emph{relative entropy} is defined as
\[
H(\rho\,\|\,\sigma)
\;:=\;
\begin{cases}
\displaystyle\int \log\frac{d\rho}{d\sigma}\,d\rho & \text{if } \rho\ll\sigma,\\[4pt]
+\infty & \text{otherwise.}
\end{cases}
\]
\end{Definition}
\begin{Definition}\label{def:fisher}
For $\mu\in\mathcal{P}(\mathbb{R}^d)$ with Lebesgue density
$\rho:=\frac{d\mu}{d\mathrm{L}}$,  the \emph{Fisher information} is
\[
I(\mu)\;:=\;\int_{\mathbb{R}^d}\frac{|\nabla\rho(x)|^2}{\rho(x)}\,dx,
\]
with the convention $I(\mu):=+\infty$ if $\mu$ is not absolutely continuous with respect to Lebesgue measure.
For a reference measure $\nu \in \Pc(\R^d)$, the
\emph{relative Fisher information} of $\mu$ with respect to $\nu$ is
\[
I(\mu\,\|\,\nu)\;:=\;\int\Big|\nabla\log\tfrac{d\mu}{d\nu}(y)\Big|^2\,\mu(dy).
\]
\end{Definition}

\subsection{Weak optimal transport}

The following result, taken from~\cite{BPRS25}, gives the existence, uniqueness,
duality, and an explicit representation of the optimal coupling
for~\eqref{eq:WOT}.

\begin{Lemma}\label{lem:wot-duality}
Under Assumptions~\ref{ass:cost}, problem~\eqref{eq:WOT}
admits a unique minimizer $\pi^*\in\Pi(\mu,\nu)$, and there exists a pair of
potentials $(\phi,\psi)\in L^1(\mu)\times L^1(\nu)$ such that
\begin{align*}
c(x,\rho)+\varepsilon H(\rho\,\|\,\nu)
&\;\geq\;\phi(x)+\rho(\psi),
&&\forall\,(x,\rho)\in\Xc\times\mathcal{P}(\mathbb{R}^d),
\\
c(x,\pi^*_x)+\varepsilon H(\pi^*_x\,\|\,\nu)
&\;=\;\phi(x)+\pi^*_x(\psi),
&&\mu\text{-a.e.\ }x.
\end{align*}
In particular, $\inf_{\pi\in\Pi(\mu,\nu)} J(\pi) = \mu(\phi)+\nu(\psi)$.

Moreover, the optimal coupling is given explicitly in
terms of the dual potential $\psi$ and the cost gradient $\hat{c}$
of~\eqref{eq:def-chat} by
\begin{equation}\label{eq:wot-optimal-density}
\frac{d\pi^*}{d(\mu\otimes \nu)}(x,y)
\;=\;
\exp\!\left(\tfrac{1}{\varepsilon}\big(\phi(x)+\psi(y)
-\delta_m c(x,\pi^*_x)(y)\big)\right).
\end{equation}
Equivalently, the disintegration of the optimal coupling is
\begin{equation}\label{eq:wot-optimal-disintegration}
\pi^*_x(dy)
\;=\;
\frac{1}{Z(x)}\,
\exp\!\left(-\frac{1}{\varepsilon}\,\delta_m c(x,\pi^*_x)(y)
\;+\;\frac{1}{\varepsilon}\,\psi(y)\right)\nu(dy),
\qquad \mu\text{-a.e.\ }x,
\end{equation}
where $Z(x)$ is the normalizing constant
\[
Z(x)\;=\;\int\exp\!\left(-\tfrac{1}{\varepsilon}\delta_m c(x,\pi^*_x)(y)
+\tfrac{1}{\varepsilon}\psi(y)\right)\nu(dy)
\;=\;\exp\!\big(-\tfrac{1}{\varepsilon}\phi(x)\big).
\]
\end{Lemma}

\begin{Remark}\label{rem:wot-fixed-point}
Formula~\eqref{eq:wot-optimal-density} is implicit, since the right-hand
side depends on $\pi^*_x$ through $\delta_m c(x,\pi^*_x)$. It can be read as a
fixed-point characterization of $\pi^*$, and reduces to the standard Sinkhorn
form when $c(x,\rho)=\int \bar c(x,y)\rho(dy)$ is linear in $\rho$: in that case
$\delta_m c(x,\rho)(y)=\bar c(x,y)$ is independent of $\rho$, and
\eqref{eq:wot-optimal-density} becomes the classical entropic optimal transport
density
\[
\frac{d\pi^*}{d(\mu\otimes\nu)}(x,y)
\;=\;\exp\!\Big(\tfrac{1}{\varepsilon}\big(\phi(x)+\psi(y)-\bar c(x,y)\big)\Big).
\]
Differentiating $\log\pi^*_x(y)$ in $y$ recovers the first-order condition
\[
\hat{c}(x,y,\pi^*_x)\;+\;\varepsilon\,\partial_y\log\tfrac{d\pi^*_x}{d\nu}(y)
\;=\;-\partial_y\psi(y),
\]
which is the pointwise stationarity condition that drives the gradient flow
analysis in Section~\ref{sec:gradient-flow}.
\end{Remark}
\subsection{Adapted Wasserstein distance and filtered processes}\label{sec:adapted-preliminaries}

We adopt the framework of~\cite{BBP26}. A \emph{two-step filtered process} is a
tuple
\[
{\X}\;=\;\big(\Omega^{\X},\,(\mathcal{F}^{\X}_i)_{i=1,2},\,\P^{\X},\,X,\,Y\big),
\]
where $(\Omega^{\X},\mathcal{F}^{\X}_2,\P^{{\X}})$ is a
probability space, $\mathcal{F}^{\X}_1\subseteq\mathcal{F}^{\X}_2$
is a filtration, and $X$, $Y$ are random variables with values in $\Xc$, $\mathbb R^d$ that are measurable with respect to $\mathcal{F}^{\X}_1$, $\mathcal{F}^{\X}_2$ respectively. We write $\mathrm{FP}_p$ for the set of
(equivalence classes of) two-step filtered processes with
$\mathbb{E}[d_{\Xc}(X,x_0)^p+|Y|^p]<\infty$ for arbitrary $x_0 \in \Xc$, and
$\mathrm{FP}_p(\mu,\nu)\subset\mathrm{FP}_p$ for the subset with marginal laws
$\mathcal{L}(X)=\mu$ and $\mathcal{L}(Y)=\nu$.

\begin{Definition}[Information map]\label{def:info-map}
For $\X\in\mathrm{FP}_p$, the \emph{information maps} are
\begin{align}\label{eq:info-map}
\mathcal{E}(\X)\;:=&\;\mathcal{L}\!\left(X,\,Y,\,\mathcal{L}(Y\,|\,\mathcal{F}^{\X}_1)\right)
\;\in\;\mathcal{P}\!\left(\Xc\times\mathbb{R}^d\times\mathcal{P}(\mathbb{R}^d)\right),\notag  \\
\mathcal{R}(\X)\;:=&\;\mathcal{L}\!\left(X,\,\mathcal{L}(Y\,|\,\mathcal{F}^{\X}_1)\right)
\;\in\;\mathcal{P}\!\left(\Xc\times\mathcal{P}(\mathbb{R}^d)\right).
\end{align}
\end{Definition}
The \emph{adapted $p$-Wasserstein distance} on $\mathrm{FP}_p$ admits the
following equivalent characterizations, see~\cite{BBP26}: for
$\X,\X'\in\mathrm{FP}_p$ with $\pi:=\mathcal{L}(X,Y)$,
$\pi':=\mathcal{L}(X',Y')$,
\[
\AW_p(\X,\X')^p
\;=\;
\inf_{\chi\in\Pi(\mathcal{L}(X),\mathcal{L}(X'))}
\int\!\Big[d_{\Xc}(x,x')^p+\Wc_p(\pi_{x},\pi'_{x'})^p\Big]\chi(dx,dx').
\]
The space $(\mathrm{FP}_p,\AW_p)$ is a complete separable metric space,
and a sequence $\X_n$ converges to $\X$ in $\AW_p$ iff
$\mathcal{E}(\X_n)\to\mathcal{E}(\X)$ in $\Wc_p$ iff $\mathcal{R}(\X_n) \to \mathcal{R}(\X)$ in $\Wc_p$.

By a slight abuse of notation, we identify $\pi\in\mathcal{P}_p(\Xc\times\mathbb{R}^d)$
with its canonical filtered process, called \emph{self-aware processes}, and write $\AW_p(\pi,\pi')$
accordingly. Under this identification,
\[
\AW_p(\pi,\pi')^p
\;=\;
\inf_{\chi\in\Pi(\pi_1,\pi'_1)}\int\!\big[d_{\Xc}(x,x')^p+\Wc_p(\pi_x,\pi'_{x'})^p\big]\,\chi(dx,dx').
\]

\begin{Lemma}[Lower semicontinuity in the adapted topology]\label{lem:lsc-aw}
Let $H:\Xc\times\mathbb{R}^d\times\mathcal{P}(\mathbb{R}^d)\to\mathbb{R}$
be lower semicontinuous and bounded from below. Then the map
\[
\mathrm{FP}_p\ni\X\;\longmapsto\;
\mathbb{E}^{\P^{\X}}\!\left[H\!\left(X_1,X_2,\mathcal{L}(X_2\,|\,\mathcal{F}^{\X}_1)\right)\right]
\]
is lower semicontinuous with respect to $\AW_p$.
\end{Lemma}

\begin{proof}
By~\eqref{eq:info-map},
$\mathbb{E}^{\P^{\X}}[H(X_1,X_2,\mathcal{L}(X_2|\mathcal{F}^{\X}_1))]
=\int H(x,y,\rho)\,\mathcal{E}(\X)(dx,dy,d\rho)$. Since
$\AW_p$-convergence implies weak convergence of the information maps,
the result follows from the Portmanteau theorem.
\end{proof}

\section{Derivation of projected gradient flow}\label{sec:gradient-flow}

In this section, we heuristically derive the projected gradient flow of \eqref{eq:WOT} on the adapted Wasserstein space with $\Xc=\mathbb R^d$.

The information-map representation in Section~\ref{sec:adapted-preliminaries} links this calculation to
the geometry of laws of random measures studied in \cite{pinzi2025nested}.
Here we use the adapted structure to identify the transport directions and
then impose the marginal constraints through an orthogonal projection. The
derivation is formal; the resulting stochastic equation is analyzed in
Sections~\ref{sec:wellposed} and~\ref{sec:convergence}.

\subsection{Tangent space of adapted Wasserstein space}

Let us heuristically derive the tangent space of the adapted Wasserstein space
$(\mathcal P_2(\mathbb R^{2d}),\AW_2)$ from the perspective of
distributionally robust optimization (DRO); see e.g. \cite{bartl2022sensitivity,jiang2024sensitivity,
sauldubois2024first}.

First, let us consider a finite-dimensional analogue. Suppose
$f:\mathbb R^d\to \mathbb R$ is continuously differentiable. Fix
$x_0\in \mathbb R^d$ with $|\nabla f(x_0)|>0$, and let
$
x_r\in \argmax\{f(x): |x-x_0|\le r\}.
$
It can be shown that
$
x_r
=
x_0
+
r\frac{\nabla f(x_0)}{|\nabla f(x_0)|}
+
\smallO(r).
$
In other words, the first-order displacement
$
\frac{x_r-x_0}{r}
$
recovers the normalized steepest ascent direction of $f$ at $x_0$. This is the
basic geometric idea behind identifying tangent vectors through small-radius
DRO problems.

We now apply the same idea to the adapted Wasserstein space. Let
$\Fc:\mathcal P_2(\mathbb R^{2d})\to \mathbb R$ be linearly
functional differentiable, and fix $\pi_0\in \mathcal P_2(\mathbb R^{2d})$.
Consider the local robust optimization problem
\[
\pi_r
\in
\argmax
\left\{
\Fc(\pi):
\mathcal{AW}_2(\pi,\pi_0)\le r
\right\}.
\]
The analogy with the finite-dimensional case suggests that, for small $r$, the
optimizer $\pi_r$ should be obtained by pushing $\pi_0$ along the
steepest ascent direction of $\Fc$ with respect to the
$\mathcal{AW}_2$ geometry.

To make this precise at a formal level, recall that the first variation of
$\Fc$ at $\pi_0$ is denoted by $\delta_m\Fc(\pi_0)$. In the classical Wasserstein geometry, the relevant gradient would be the full
spatial gradient of $\delta_m\Fc(\pi_0)$. In the adapted Wasserstein
geometry, however, perturbations are constrained by the causal structure: the
first component can only be transported using information available at the
first step, while the second component may depend on both coordinates. This is
why the usual gradient is replaced by the adapted gradient.

For $\pi\in \mathcal P_2(\mathbb R^{2d})$ and
$\varphi\in C_c^2(\mathbb R^{2d}; \R)$, define
\begin{align}\label{eq:adaptedgradient}
\nabla_\pi^{\mathrm{ad}}\varphi(x,y)
:=
\left(
\mathbb E^\pi[\partial_x\varphi(X,Y)\mid X=x],
\partial_y\varphi(x,y)
\right).
\end{align}
The first component is the conditional expectation of $\partial_x\varphi$ given
$X=x$, reflecting the fact that the first-stage perturbation must be adapted to
the first coordinate. The second component remains the usual derivative in the
$y$-variable, since the second-stage perturbation may depend on the full state
$(x,y)$.

According to \cite{sauldubois2024first}, under suitable regularity assumptions, the small-radius DRO expansion in
$\mathcal{AW}_2$ takes the form
\[
\sup_{\mathcal{AW}_2(\pi,\pi_0)\le r}\Fc(\pi)
=
\Fc(\pi_0)
+
r
\left\|
\nabla_{\pi_0}^{\mathrm{ad}}\delta_m\Fc(\pi_0)
\right\|_{L^2(\pi_0)}
+
\smallO(r).
\]
Moreover, the corresponding optimizer satisfies the formal expansion
\[
\pi_r
=
\left(
id
+
r
\frac{
\nabla_{\pi_0}^{\mathrm{ad}}\delta_m\Fc(\pi_0)
}{
\left\|
\nabla_{\pi_0}^{\mathrm{ad}}\delta_m\Fc(\pi_0)
\right\|_{L^2(\pi_0)}
}
+
\smallO(r)
\right)_\#\pi_0 .
\]
Thus the adapted Wasserstein derivative of $\Fc$ at $\pi_0$ is naturally
identified with
\[
\nabla_{\pi_0}^{\mathrm{ad}}\delta_m\Fc(\pi_0)
:\mathbb R^{2d}\to \mathbb R^{2d}.
\]

This calculation suggests the following formal definition of the tangent space.

\begin{Definition}[Formal tangent space]
The tangent space of the adapted Wasserstein space
$(\mathcal P_2(\mathbb R^{2d}),\mathcal{AW}_2)$ at
$\pi\in\mathcal P_2(\mathbb R^{2d})$ is given by
\[
T_\pi^{\mathrm{ad}}\mathcal P_2(\mathbb R^{2d})
\simeq
\overline{
\left\{
\nabla_\pi^{\mathrm{ad}}\varphi:
\varphi\in C_c^\infty(\mathbb R^{2d}; \R)
\right\}
}^{\,L^2(\pi)}.
\]
It is equipped with the Riemannian-like metric
\[
\|v\|_{T_\pi^{\mathrm{ad}}}^2
:=
\int_{\mathbb R^{2d}} |v(x,y)|^2\,\pi(dx,dy).
\]
\end{Definition}

\subsection{Projection on the set of couplings $\Pi(\mu,\nu)$}

In \cite{conforti2023projected}, the authors characterize the tangent cone
$\Tan_{\pi}\big(\Pi(\mu,\nu)\big)$ in the  Wasserstein space $(\Pc_2(\R^{2d}),\Wc_2)$ and, more generally,
for an arbitrary vector subspace $\mathcal{V} \subset C^{1}_{\mathrm{pol},2}(\R^{2d}; \R)$, they identify $\Tan_{\pi}(\mathcal{V}^{\perp})$
(where $\mathcal{V}^{\perp}$ is defined in \eqref{eqdef:orthodual}) by geometric and flow-based arguments.
Here we derive this identification from a complementary viewpoint,
namely through first-order sensitivity analysis of constrained Wasserstein DRO.
This yields, at a formal level, the same characterization as in \cite{conforti2023projected}.

First, let us consider a finite-dimensional analogue. Let $M:=\{x\in\R^d:\ g(x)=0\}$ be a smooth submanifold in $\R^d$, with $g:\R^d\to\R^k$ of class $C^1$ and $\mathrm{rank}\,Dg(x_0)=k$.
Given a smooth function $f:\R^d\to\R$ and a point $x_0\in M$, consider
\[
G(r):=\sup\Big\{f(x):\ x\in M,\ |x-x_0|\le r\Big\}.
\]
Under sufficient regularity, the tangent space of $M$ at $x_0$ is given by $T:=\Ker Dg(x_0)$, and it can be shown that the first-order increase of $G$ at $0$ is
\begin{align}\label{eq:finitetagent}
G'(0)
&= \left|  \mathrm{Proj}_{T} \nabla f(x_0)  \right|
=
\mathrm{dist}\big(\nabla f(x_0), T^\perp\big),
\end{align}
where $\mathrm{Proj}_T \nabla f(x_0)$ denotes the projection of $\nabla f(x_0)$ onto the tangent space $T$.

Let us turn to the adapted Wasserstein space. Take a vector space
\begin{align*}
\mathcal{V}:= \left\{ h_1 \oplus h_2 : \, h_1,h_2 \in C^{1}_{\mathrm{pol},2}(\mathbb{R}^{d}; \R), \, \mu(h_1)= \nu(h_2)=0 \right\} \subset C^{1}_{\mathrm{pol},2}(\mathbb{R}^{2d}; \R).
\end{align*}
It can easily be that $\Pi(\mu,\nu)=\mathcal{V}^{\perp}$ as defined in \eqref{eqdef:orthodual} (see section 2.7 of \cite{conforti2023projected}).
Consider
\begin{align*}
    \mathcal{G}(r):=\sup\left\{\mathcal{F}(\pi): \, \pi \in \Pi(\mu,\nu), \, \AW_2(\pi,\pi_0) \leq r \right\},
\end{align*}
where $\mathcal{F}: \Pc_2(\R^{2d}) \to \R$ is sufficiently regular at $\pi_0 \in \Pi(\mu,\nu)$. According to \cite{sauldubois2024first,sauldubois_modelrisk_statichedging_cdro}, we have the first-order expansion of $\mathcal{G}$
\begin{equation}\label{eq:Gprime_ad}
\mathcal{G}'(0)
=
\inf_{h\in\mathcal{V}}\big\|\nabla^{\rm ad}_{\pi_0}\delta_m\Fc(\pi_0,\cdot)-\nabla^{\rm ad}_{\pi} h \big\|_{L^2(\pi_0)}.
\end{equation}
Setting $\nabla^{\rm ad}_{\pi}(\mathcal{V})=\{\nabla^{\rm ad}_{\pi} h : \, h \in \mathcal{V} \}$, we rewrite \eqref{eq:Gprime_ad} as
\begin{align*}
    \mathcal{G}'(0)= \mathrm{dist}_{L^2(\pi_0)}\left(\nabla^{\rm ad}_{\pi_0}\delta_m\Fc(\pi_0,\cdot), \overline{\nabla^{\rm ad}_{\pi}(\mathcal{V})}^{L^2(\pi_0)} \right)
\end{align*}
In analogy to \eqref{eq:finitetagent}, we heuristically identify the tangent space of $\Pi(\mu,\nu) \subset (  \Pc_2(\R^{2d}) , \AW_2 )$.
\begin{Definition}
    For any $\pi \in \Pi(\mu,\nu)$, we formally identify the tangent space as
\begin{align}\label{eq:adaptedtangentcoupling}
T^{\rm ad}_{\pi}\left( \Pi(\mu,\nu)\right)
&\;\simeq\;
\Big(\overline{\nabla^{\rm ad}_{\pi}(\mathcal{V})}^{\,L^2(\pi)}\Big)^{\perp} \\
& \; = \;
\Big\{
g\in T^{\rm ad}_{\pi}\Pc_2(\R^{2d}):\
\int g(x,y)\cdot \nabla^{\rm ad}_{\pi} h(x,y) \, \pi(dx,dy)=0,
\ \ \forall \, h\in \mathcal{V}
\Big\}. \notag
\end{align}
\end{Definition}

\begin{Proposition}
When $d=1$ and $\mu,\nu \in \Pc_2(\R)$, for any $\pi\in\Pi(\mu,\nu)$,
\begin{equation}\label{eqdef:Tangent space Couplings}
\begin{aligned}
T^{\rm ad}_{\pi}\big(\Pi(\mu,\nu)\big)
& =
\big(L^2(\mu)\oplus L^2(\nu)\big)^{\perp},
\end{aligned}
\end{equation}
where orthogonality is taken in $L^2(\pi;\R^{2})$.
The orthogonal projection onto $T^{\rm ad}_{\pi}\big(\Pi(\mu,\nu)\big)$ is given by
\begin{equation}\label{eqdef:orthogonal projectio space Couplings}
\begin{aligned}
P^{\pi}:\ T^{\rm ad}_{\pi}\!\big(\Pc_2(\R^{2})\big) &\longrightarrow T^{\rm ad}_{\pi}\big(\Pi(\mu,\nu)\big),\\
u=(u_1,u_2) &\longmapsto \Big(0,\ u_2-\E_{Y}^{\pi}\!\big[u_2(X,Y)\big]\Big).
\end{aligned}
\end{equation}
\end{Proposition}
\begin{proof}
For the coupling constraint, combining \eqref{eq:adaptedgradient} with \eqref{eq:adaptedtangentcoupling} yields
\[
\big(T^{\rm ad}_{\pi}(\Pi(\mu,\nu))\big)^{\perp}
=
\overline{\Big\{\, (x, y) \mapsto (h_1'(x), h_2'(y)):\ h_1,h_2\in \, C^1_{\rm pol, 2}(\R, \R)\Big\}}^{\,L^2(\pi)}.
\]
Since $h_1'$ (resp.\ $h_2'$) ranges over $L^2(\mu)$ (resp.\ $L^2(\nu)$) up to $L^2$-closure, we obtain
\begin{align}\label{eq:orthogonaltangent}
\big(T^{\rm ad}_{\pi}(\Pi(\mu,\nu))\big)^{\perp}
=
L^2(\mu)\oplus L^2(\nu),
\end{align}
and therefore \eqref{eqdef:Tangent space Couplings}. This agrees with \cite[Eq (2.8)]{conforti2023projected}.

To compute the projection, it is convenient to first determine the orthogonal projection onto the normal space
$\big(T^{\rm ad}_{\pi}(\Pi(\mu,\nu))\big)^{\perp}=L^2(\mu)\oplus L^2(\nu)$.
Let $u=(u_1,u_2)\in T^{\rm ad}_{\pi}(\Pc_2(\R^2))$. As
$u_1$ is $\sigma(X)$-measurable,  the $L^2(\pi)$-projection of $(u_1,u_2)$ onto $L^2(\mu)\oplus L^2(\nu)$ is
\[
\Big(u_1,\ \E^{\pi}[u_2(X,Y)\mid Y]\Big)
\in L^2(\mu)\times L^2(\nu),
\]
which yields \eqref{eqdef:orthogonal projectio space Couplings}.
\end{proof}

\begin{Remark}
For $d \geq 2$, \eqref{eq:orthogonaltangent} doesn't hold anymore. Actually, it is only true that
\begin{align*}
\big(T^{\rm ad}_{\pi}(\Pi(\mu,\nu))\big)^{\perp}
\subset
L^2(\mu)\oplus L^2(\nu).
\end{align*}
However, in the rest of the paper, we still take \eqref{eqdef:orthogonal projectio space Couplings} in the derivation of gradient flow.
\end{Remark}

\subsection{Gradient flow}
Now we are ready to derive the gradient flow of $J$ in \eqref{eq:objective} in $(\Pi(\mu,\nu), \AW_2)$,
\begin{align}\label{eq:gradient_flow}
    \partial_t \pi_t=\nabla \cdot \left( \pi_t \;  P^{\pi_t} \left(\nabla^{\rm ad}_{\pi_t}\delta_m J(\pi_t) \right)\right).
\end{align}
Restricted to $\Pi(\mu,\cdot)$, we have its linear derivative
\begin{align*}
    \delta_m J(\pi)(x,y)&=\delta_m c(x,\pi_x)(y)+ \varepsilon \log \left( \frac{d \pi}{d \mu \otimes \nu} \right),\\
    \partial_y \delta_m J(\pi)(x,y) & =\partial_y \delta_m c(x,\pi_x)(y)+\varepsilon \partial_y \log \rho(x,y)+\varepsilon \nabla V(y),
\end{align*}
and its projection $ P^{\pi} \left(\nabla^{\rm ad}_{\pi}\delta_m J(\pi) \right)$ equals
\begin{align*}
   \left(0, \hat{c}(x,y,\pi_x)+\varepsilon \partial_y \log \rho(x,y)- \int \hat c(z,y,\pi_z) \, \pi^y(dz)-\varepsilon \int \partial_y \log \rho(z,y) \, \pi^y(dz) \right),
\end{align*}
where $\rho$ is the Lebesgue density of $\pi$, and $\hat c(x,y,\pi_x)=\partial_y \delta_m c (x,\pi_x)(y)$. By direct computation
\vspace{-1.5mm}
\begin{equation}\label{eq:conditiondensity}
    \int \partial_y \log \rho(z,y) \pi^y(dz)= \frac{\partial_y \int \rho(z,y) \, dz}{\int \rho(z,y) \, dz}=-\nabla V(y).
    \vspace{-1.5mm}
\end{equation}
Therefore, we get the SDE corresponding to \eqref{eq:gradient_flow},
\begin{align}\label{eq:SDE}
    dY_t=-\left(\hat c(X,Y_t,(\pi_t)_{X})-\tilde c^{\pi_t}(Y_t) +\varepsilon \nabla V(Y_t)\right)dt+\sqrt{2\varepsilon}\,dB_t, \end{align}
where $\pi_t= \mathcal{L}(X,Y_t)$, and $\tilde c^{\pi}(Y):= \mathbb E^{\pi}_{Y}[\hat c (X, Y,(\pi)_{X})]$.

\section{Numerical Application}\label{sec:numerics}
We illustrate here some applications of the gradient-flow approach to well-known problems related to Martingale Optimal Transport, see \cite{beiglbock2013model} for the founding article and \cite{henrylabordere2017model} for an overview, Weak Optimal Transport, and more generally Weak Martingale Optimal Transport, see \cite{Beiglbock2021StabilityOT}. Throughout this section, we fix a function
$
c : \R \times \Pc(\R) \to \R,
$
which is convex in its measure argument, a probability measure $\mu \in \Pc(\R)$, and a Gibbs measure
$
\nu := e^{-V}\,dx
$
with potential $V$. Given $\varepsilon > 0$ and $\lambda_{\rm M} \geq 0$, we aim to solve
\begin{equation*}
WMOT(\lambda_{\rm M},\varepsilon)
:=
\inf_{\Pi \in \Pi(\mu,\nu)}
\left\{
\E^{\mu}\big[ c(X,\pi_X) \big]
+
\frac{\lambda_{\rm M}}{2} \E^{\mu}
\left[ (
X - \int y \,\pi_X(dy) )^2
\right]
+
\varepsilon H(\pi \mid \mu \otimes \nu)
\right\}.
\end{equation*}

We consider this family of problems because it interpolates, at least formally, between Weak Optimal Transport and Weak Martingale Optimal Transport. More precisely, for suitable choices of the penalization parameter $\lambda_{\rm M}$ and the entropic regularization parameter $\varepsilon$, one expects under irreducibility of $(\mu, \nu)$ (see Remark 2.2 of \cite{nutz2024martingale} for the definition) that
\vspace{-1.5mm}
\[
WMOT(\lambda_{\rm M},\varepsilon)
\underset{\substack{\lambda_{\rm M} \to \infty \\ \varepsilon \to 0}}{\longrightarrow}
\inf_{\pi \in \Pi^{\rm M}(\mu,\nu)}
\E^{\mu}\big[ c(X,\pi_X) \big],
\vspace{-1.5mm}
\]
where $\Pi^{\rm M}(\mu,\nu)$ denotes the set of martingale couplings between $\mu$ and $\nu$, while
\[
WMOT(\lambda_{\rm M},\varepsilon)
\underset{\substack{\lambda_{\rm M} \to 0 \\ \varepsilon \to 0}}{\longrightarrow}
\inf_{\pi \in \Pi(\mu,\nu)}
\E^{\mu}\big[ c(X,\pi_X) \big].
\]
This allows us to recover the two limiting problems of interest.
To solve these problems numerically, we use a particle method. This allows an efficient simulation of the dynamics, while projecting the particles onto a fixed grid in order to obtain a discrete approximation of the joint law and of the associated conditional laws.
More precisely, assume that $(X^i,Y^i)_{1 \leq i \leq N}$ are simulated particles. Let
\vspace{-1.5mm}
\[
-\infty=x_0 < x_1 < \cdots < x_M=+\infty,
\qquad
-\infty=y_0 < y_1 < \cdots < y_M=+\infty,
\vspace{-1.5mm}\]
be two grids of $\R$, and let
$
\Xc := \{x_0,\ldots,x_{M-1}\},
$ and $
\Yc := \{y_0,\ldots,y_{M-1}\}.
$
We define the discrete probability measure
$\pi_N \in \Pc(\Xc \times \Yc)$ by
\vspace{-1.5mm}
\[
\pi_N(x_k,y_\ell)
:=
\frac{1}{N}
\sum_{i=1}^N
\mathds{1}_{[x_k,x_{k+1})}(X^i)
\mathds{1}_{[y_\ell,y_{\ell+1})}(Y^i),
\qquad
0\leq k,\ell\leq M-1.
\vspace{-1.5mm}
\]
The advantage of this discretization is that the support of $\pi_N$ remains fixed throughout the numerical procedure, while formally
$\pi_N \approx \Lc(X,Y).$
Moreover, since $\pi_N$ is a discrete measure, the conditional laws of $Y$ given $X$ and of $X$ given $Y$ can be computed directly from the corresponding rows and columns of the matrix representation of $\pi_N$. Indeed, for any test function $\varphi : \R \to \R$, if
$
m_k := \sum_{\ell=0}^{M-1} \pi_N(x_k,y_\ell)>0,
$
then
\vspace{-2mm}
\[
\E_{\pi_N}\big[\varphi(Y)\mid X \in [x_k,x_{k+1})\big]
=
\sum_{\ell=0}^{M-1}
\varphi(y_\ell)
\frac{\pi_N(x_k,y_\ell)}{m_k}.
\vspace{-2mm}
\]
In particular,
$
\E_{\pi_N} [Y\mid X \in [x_k,x_{k+1}) ]
=
\sum_{\ell=0}^{M-1}
y_\ell
\frac{\pi_N(x_k,y_\ell)}{m_k}.
$
This provides a direct discrete approximation of the barycenter of the conditional law $\Lc(Y\mid X=x_k)$.

The grids $\Xc$ and $\Yc$ are chosen using quantile interpolation. More precisely, we construct bins adapted to the marginal distributions. For the $X$-grid, with the convention
$
x_0=-\infty,
$ and $
x_M=+\infty,
$
we choose the points $(x_k)_{0\leq k\leq M}$ so that
$
\mu ([x_k,x_{k+1}))
=
\frac{1}{M},
\qquad
0\leq k\leq M-1.
$
Similarly, the grid $\Yc$ is constructed from the quantiles of $\nu$. Since $\mu$ and $\nu$ do not charge $\pm\infty$, these bins form a partition of the real line up to sets of zero mass. Once the particle system has been projected onto the fixed grid, the McKean--Vlasov dynamics are simulated by a classical Euler scheme. The conditional quantities appearing in the drift are recomputed at each time step from the discrete matrix $\pi_N$.

\vspace{5mm}

\begin{algorithm}[H]
\caption{Particle approximation of the regularized WMOT problem}
\label{alg:particle-WMOT}

\KwIn{
Marginals $\mu,\nu = e^{-V} dx$; cost derivative $\nabla_y \delta_m c$;
regularization parameter $\varepsilon>0$; penalization $\lambda_{\rm M}$, number of particles $N$;
number of grid cells $M$; time step $\Delta t>0$; number of iterations
$n_{\mathrm{iter}}$.
}

\KwOut{
Final particle system $(X^i,Y_{n_{\mathrm{iter}}}^i)_{1\leq i\leq N}$,
discrete coupling $\pi_N^{n_{\mathrm{iter}}}$, and approximation of the
objective value.
}

Construct quantile grids $\Xc=\{x_1,\ldots,x_M\}$ and
$\Yc=\{y_1,\ldots,y_M\}$ associated with $\mu$ and $\nu$\;

Sample $(X^i)_{1\leq i\leq N}$ from $\mu$; initialize $(Y_0^i)_{1\leq i\leq N}$ from $\nu$\;

\For{$n=0,\ldots,n_{\mathrm{iter}}-1$}{

Construct the empirical discrete coupling $\pi_N^n$ by projecting
$(X^i,Y_n^i)_{1\leq i\leq N}$ onto $\Xc\times\Yc$\;

Compute the conditional laws $(\pi_N^n)_X$, namely the discrete laws of
$Y_n$ given $X$, from the rows of the matrix $\pi_N^n$\;

For each particle $i$, compute
\[
\hat{C} \big(X^i,Y_n^i,(\pi_N^n)_{X^i}\big)
:=
\nabla_y \delta_m c( X^i,Y_n^i,(\pi_N^n)_{X^i}\big) + \lambda_{\rm M} \big( X^i - \int z (\pi_N^n)_{X^i} (dz) \big)
,
\]
where $(\pi_N^n)_{X^i}$ denotes the conditional law of $Y_n$ given the bin
containing $X^i$\;

Compute the conditional average
\[
\tilde C_N^n
\simeq
\E_{\pi_N^n}
\left[
\hat{C} \big(X,Y_n^i,(\pi_N^n)_X\big)
\,\middle|\,
Y\in [y_\ell,y_{\ell+1})
\right],
\]
where $Y_n^i\in [y_\ell,y_{\ell+1})$\;

Draw independent standard Gaussian random variables
$(\xi_{n+1}^i)_{1\leq i\leq N}$; update the particles according to the Euler scheme
\[
Y_{n+1}^i
=
Y_n^i
+
\left(
\tilde C_N^n
-
\hat{C} \big(X^i,Y_n^i,(\pi_N^n)_{X^i}\big)
-
\varepsilon \nabla V(Y_n^i)
\right)\Delta t
+
\sqrt{2\varepsilon \Delta t}\,\xi_{n+1}^i.
\]

}

\Return $(X^i,Y_{n_{\mathrm{iter}}}^i)_{1\leq i\leq N}$ and
$\pi_N^{n_{\mathrm{iter}}}$\;

\end{algorithm}

\paragraph{Classical Optimal Transport.}
We first illustrate the outcome of our simulations in the regime $\varepsilon\ll 1$, $\lambda_{\rm M}=0$, and
\[
C(x,\rho):=\int c(x,y)\,\rho(dy),
\qquad
c(x,y)=|x-y|^2.
\]
This corresponds to the classical quadratic optimal transport problem. In dimension one, the optimizer is the monotone rearrangement, or equivalently the Fréchet--Hoeffding coupling. Hence, as $\varepsilon\to 0$, one expects the particles to concentrate on the graph of the increasing transport map. This is what happenss, as shown in the following histograms.

\begin{figure}[H]
   \centering
   \begin{subfigure}[t]{0.24\textwidth}
       \centering
       \includegraphics[width=\textwidth]{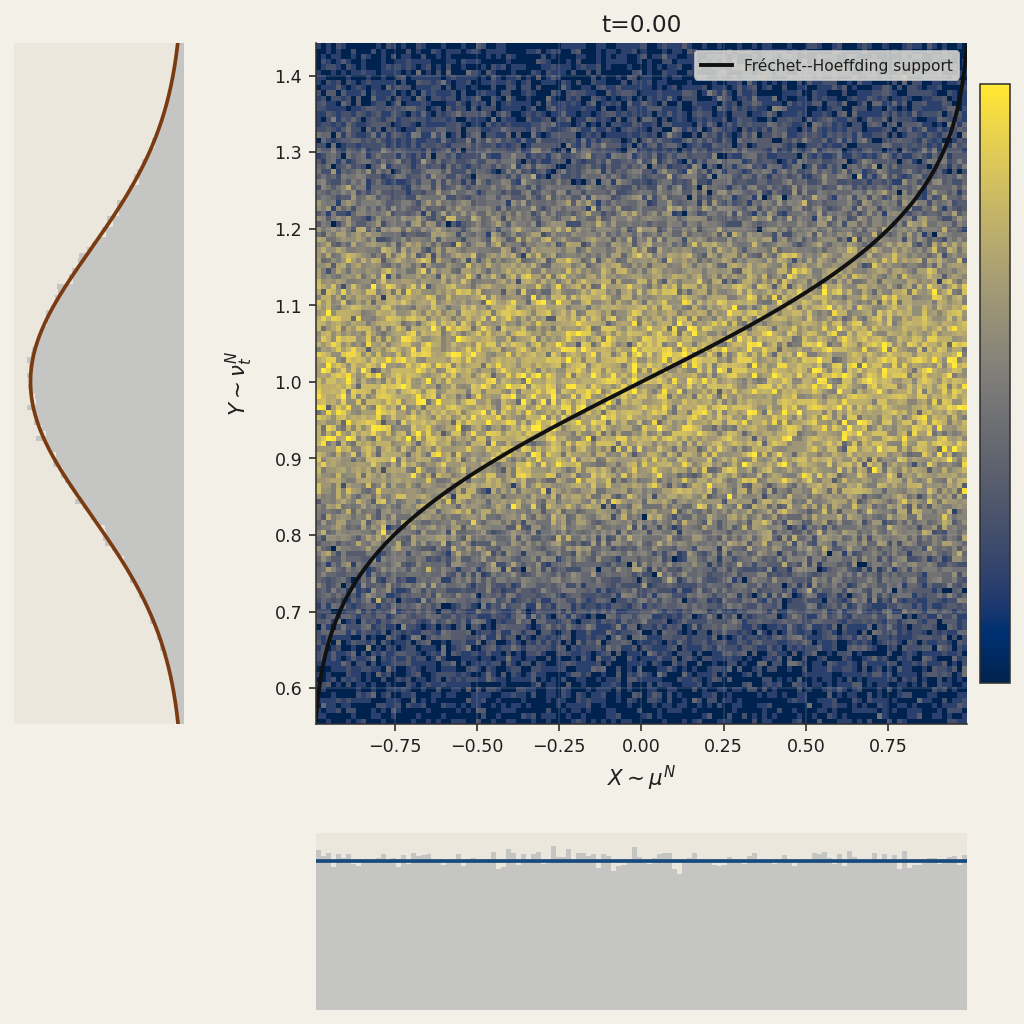}
       \caption{$t=0.00$}
   \end{subfigure}
   \hfill
   \begin{subfigure}[t]{0.24\textwidth}
       \centering
       \includegraphics[width=\textwidth]{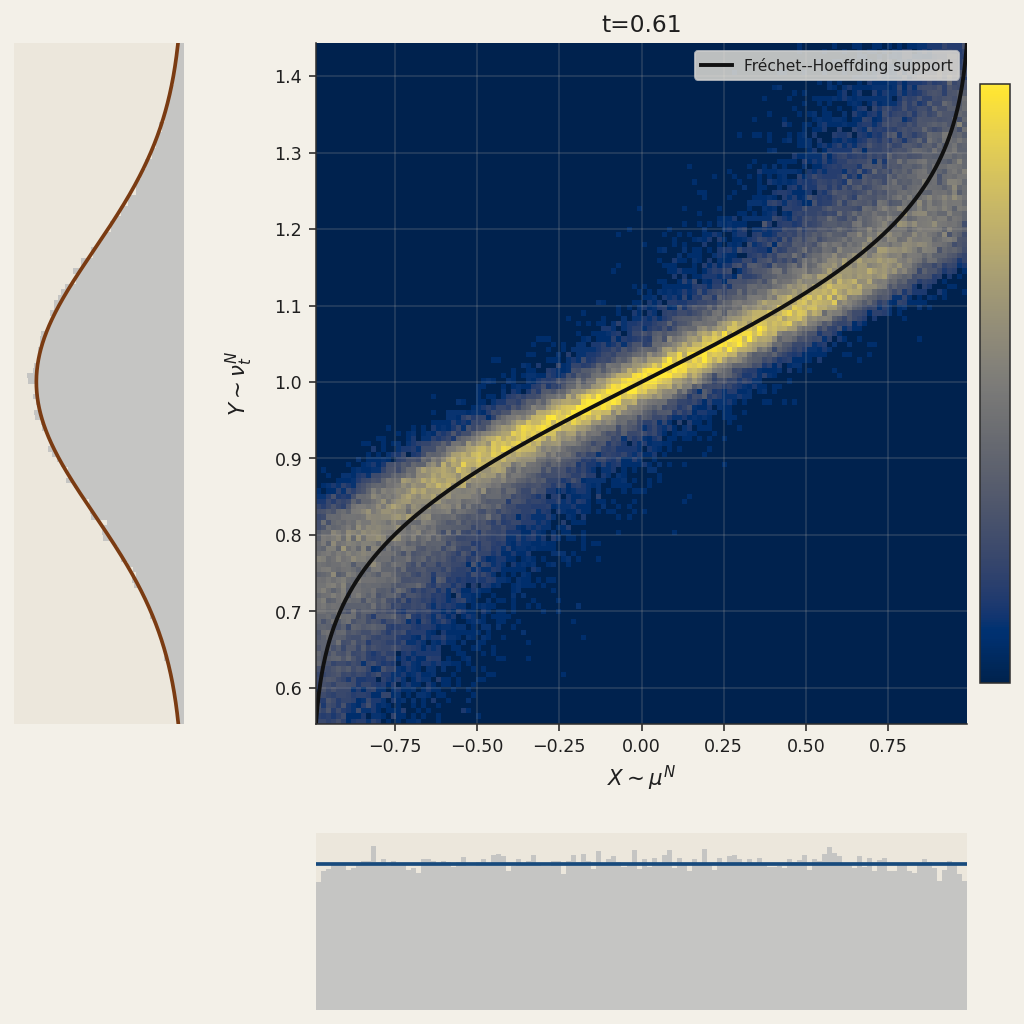}
       \caption{$t=0.61$}
   \end{subfigure}
   \hfill
   \begin{subfigure}[t]{0.24\textwidth}
       \centering
       \includegraphics[width=\textwidth]{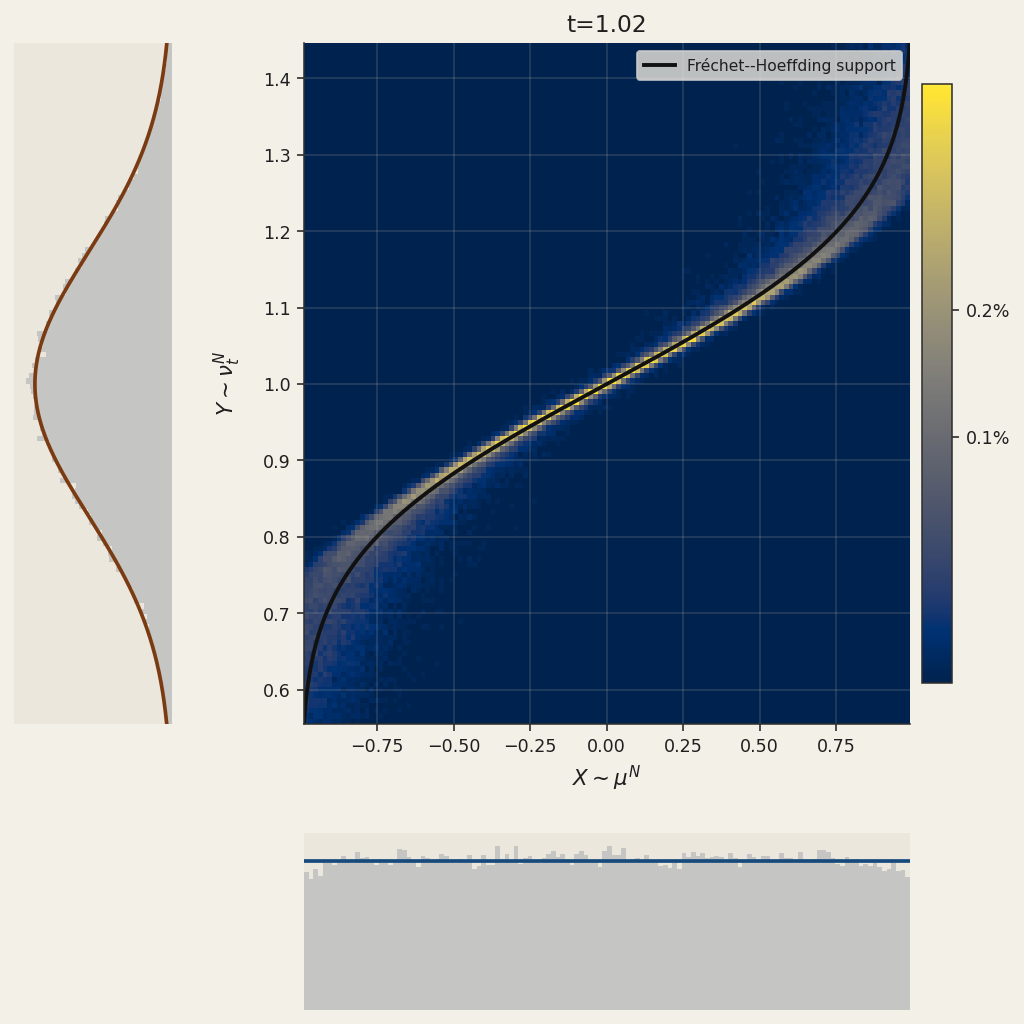}
       \caption{$t=1.02$}
   \end{subfigure}
   \hfill
   \begin{subfigure}[t]{0.24\textwidth}
       \centering
       \includegraphics[width=\textwidth]{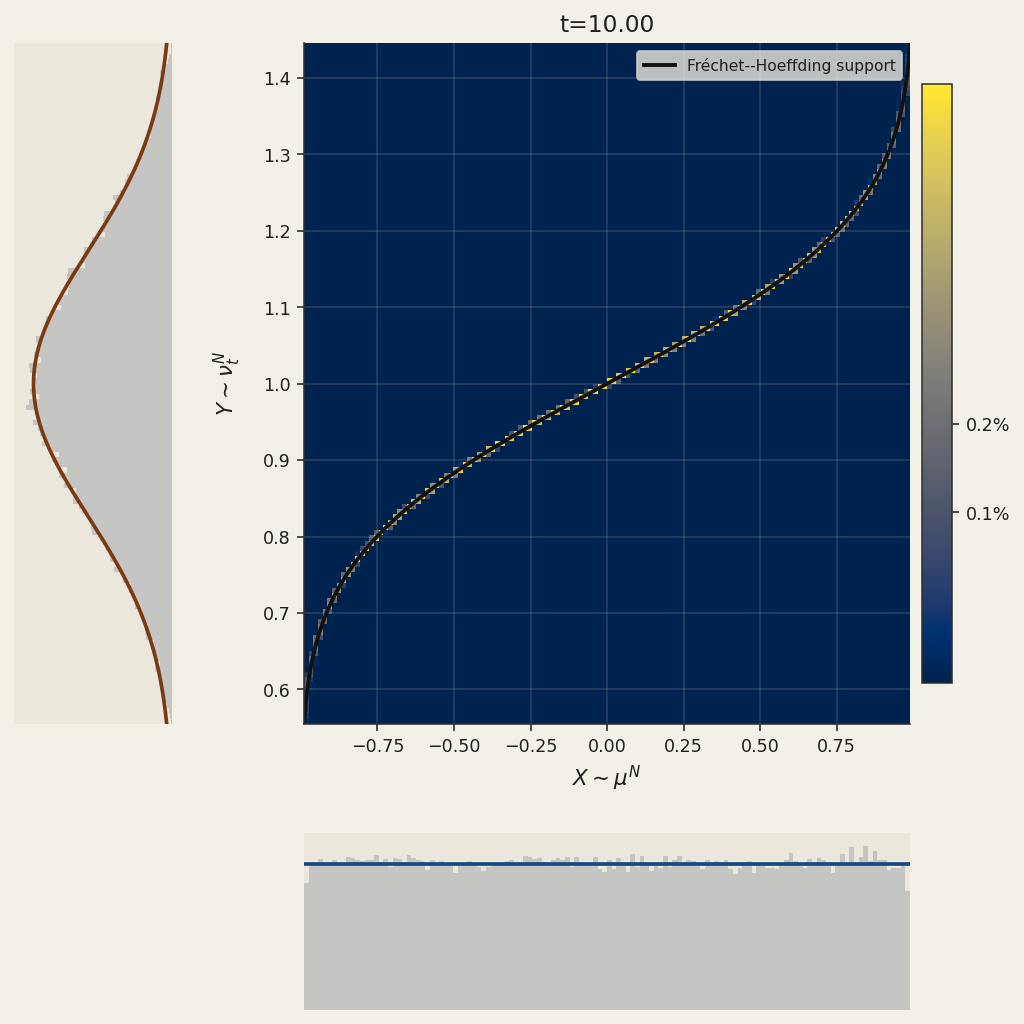}
       \caption{$t=10.00$}
   \end{subfigure}
   \caption{
   Evolution of the empirical coupling for the uniform--Gaussian optimal transport benchmark.
   }
   \label{fig:ot-ug-dynamics}
\end{figure}

The plots show a rapid concentration of the support of the empirical coupling. Starting from a diffuse cloud, the particles progressively align along the increasing transport map associated with the quadratic cost. This graph is precisely the support of the Fréchet--Hoeffding coupling in dimension one. The concentration becomes sharper as time increases, and the final configuration is concentrated close to the optimal transport graph, up to the residual entropic smoothing induced by the small parameter $\varepsilon$.

We also observe convergence of the empirical objective value toward the quadratic optimal transport value as $t$ goes to infinity. In the simulations below, $\varepsilon=10^{-6}$, so the entropic bias is negligible at the scale of the plot. The decay of the objective gap is consistent with an exponential convergence regime, which is the expected behavior for the underlying regularized gradient flow near equilibrium.
\begin{figure}[H]
   \centering
   \includegraphics[width=0.70\textwidth]{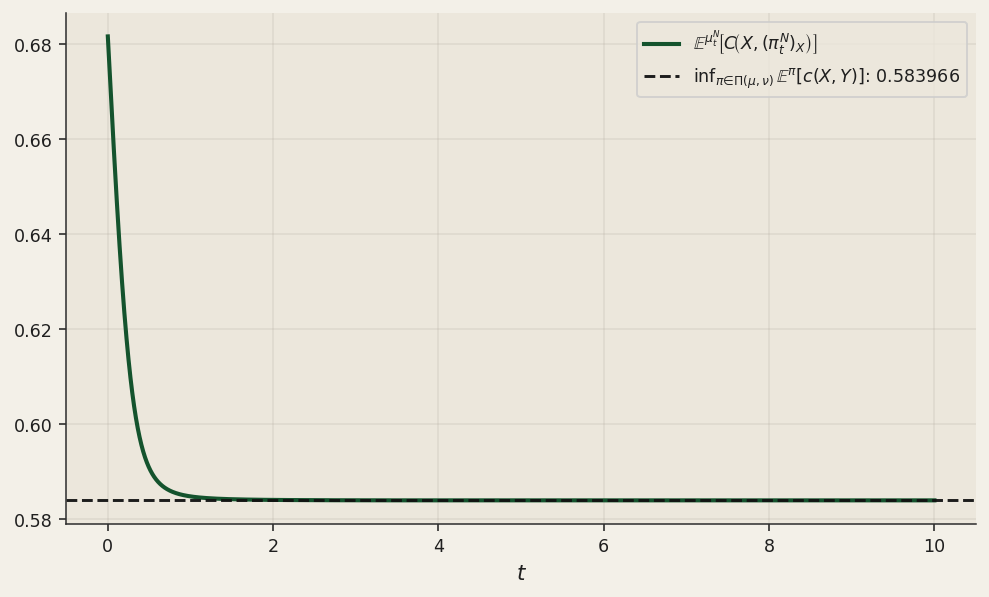}
   \caption{
   Convergence of the empirical objective value for the uniform--Gaussian optimal transport benchmark.
   }
   \label{fig:ot-ug-objective}
\end{figure}

\begin{Remark}
    Similar to \cite[Theorem 1.4]{conforti2023projected}, we have the exponential convergence for entropic optimal transport with quadratic cost. However, our dynamic
\begin{align*}
dY_t=-\left(2\mathbb E^{\pi_t}[X|Y_t]-2X +\varepsilon \nabla V(Y_t)\right)dt+\sqrt{2\varepsilon}\,dB_t
\end{align*}
where $\pi_t= \mathcal{L}(X,Y_t)$, is simpler to implement numerically as it only involves conditional expectation w.r.t. $Y$
\end{Remark}

\paragraph{Martingale Optimal Transport.}
This corresponds to the case
\[
c(x,\rho):=\int f(x,y)\,\rho(dy),
\qquad
\lambda_{\rm M}\gg 1 \text{ and } \varepsilon \ll 1.
\]
In one-dimensional Martingale Optimal Transport, under the Spence--Mirrlees condition, the optimal martingale transport is concentrated on the union of two graphs; see \cite{henrylabordere2017model} for their construction. This coupling is usually referred to as the left--monotone coupling, or as the mirror coupling in specific cases. Hence, when $\varepsilon\approx 0$, the empirical law is expected to converge toward it.

\begin{figure}[H]
   \centering
   \begin{subfigure}[t]{0.24\textwidth}
       \centering
       \includegraphics[width=\textwidth]{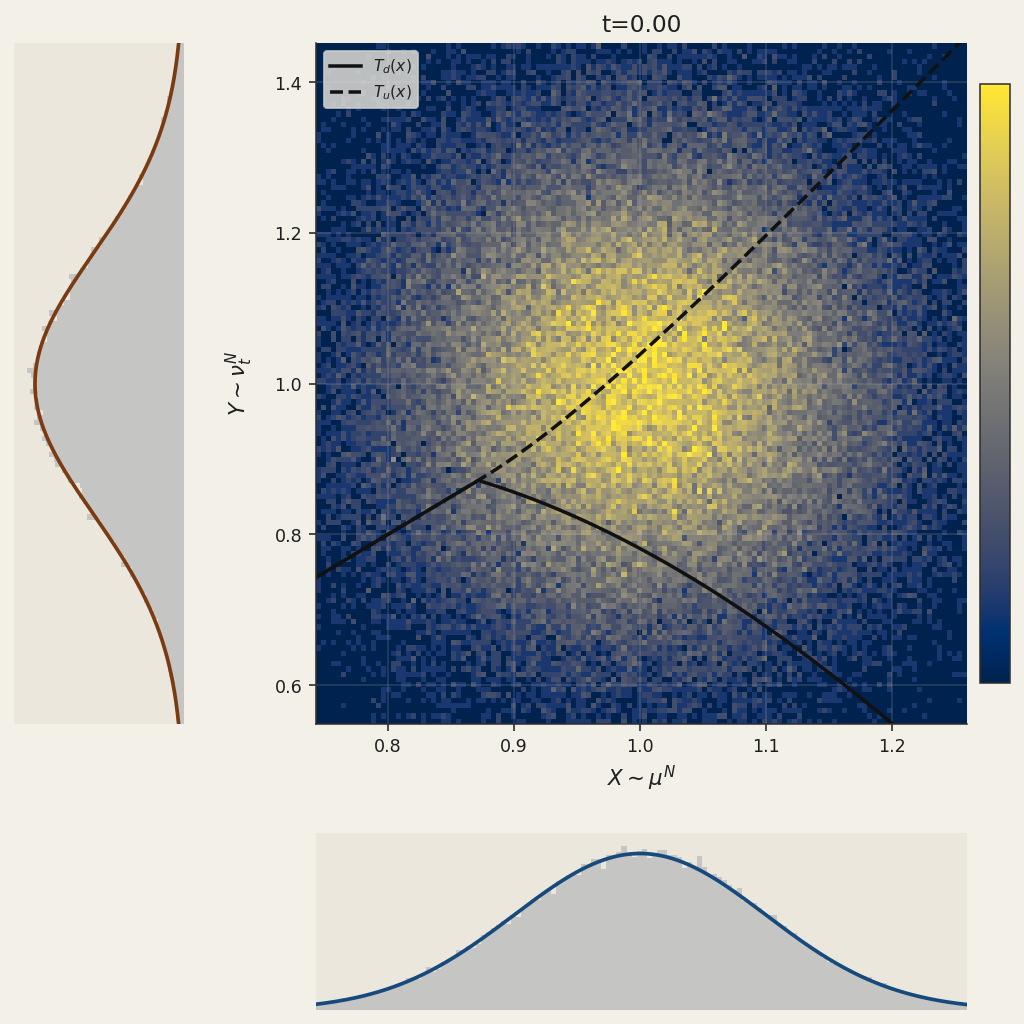}
       \caption{$t=0.00$}
   \end{subfigure}
   \hfill
   \begin{subfigure}[t]{0.24\textwidth}
       \centering
       \includegraphics[width=\textwidth]{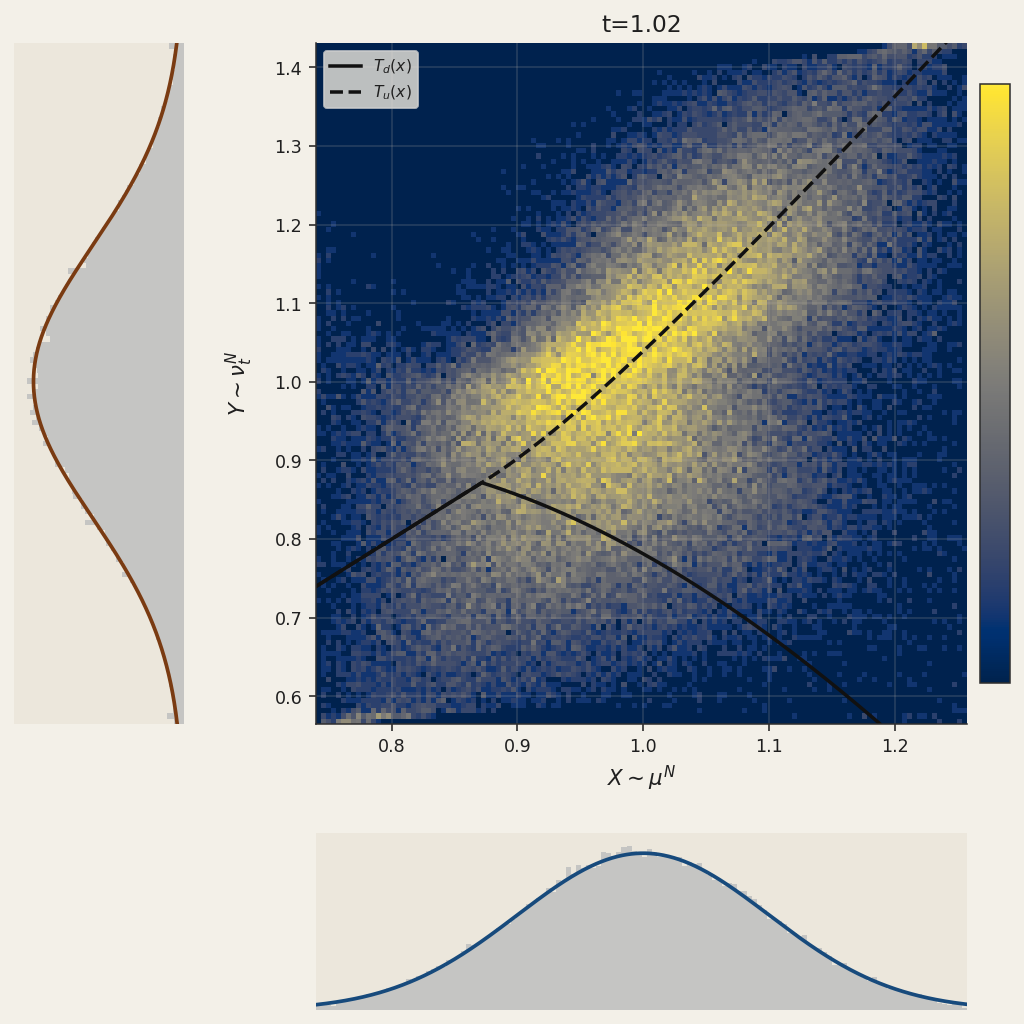}
       \caption{$t=1.02$}
   \end{subfigure}
   \hfill
   \begin{subfigure}[t]{0.24\textwidth}
       \centering
       \includegraphics[width=\textwidth]{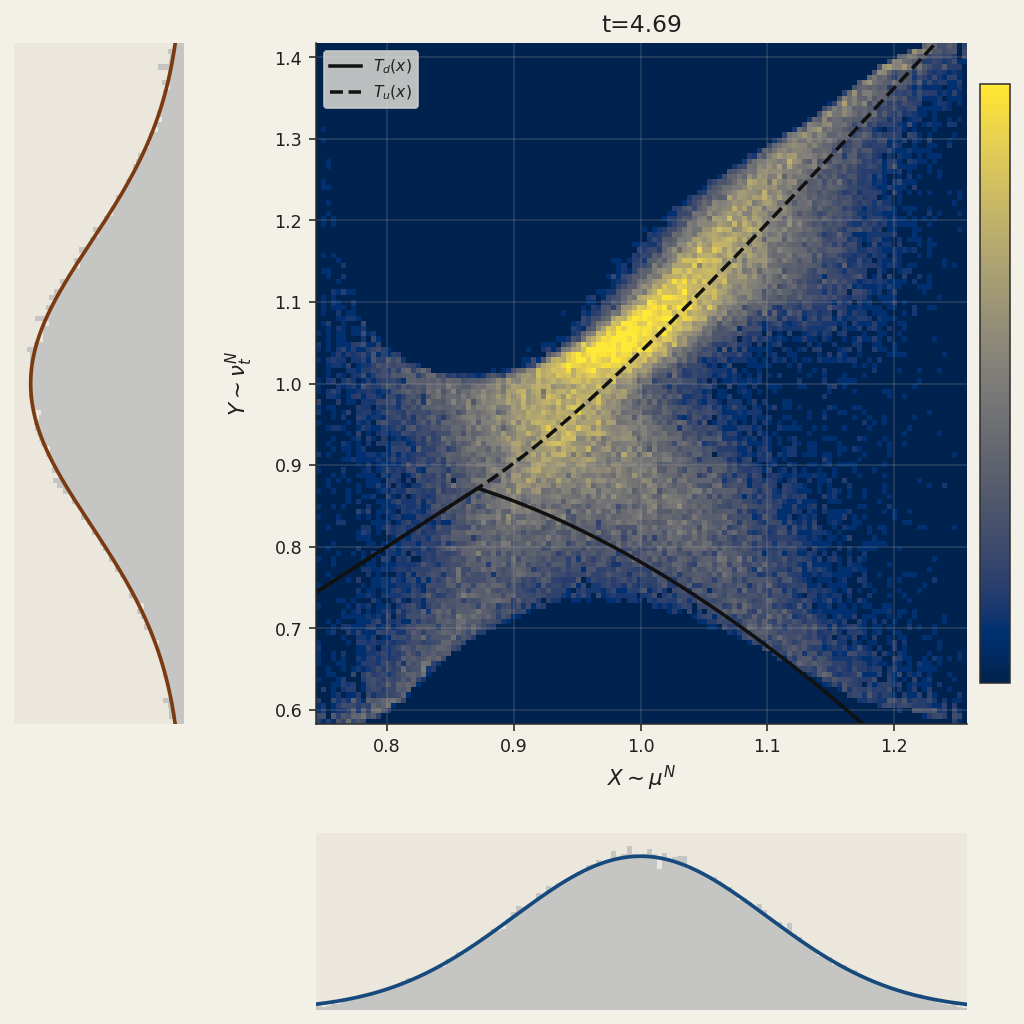}
       \caption{$t=4.69$}
   \end{subfigure}
   \hfill
   \begin{subfigure}[t]{0.24\textwidth}
       \centering
       \includegraphics[width=\textwidth]{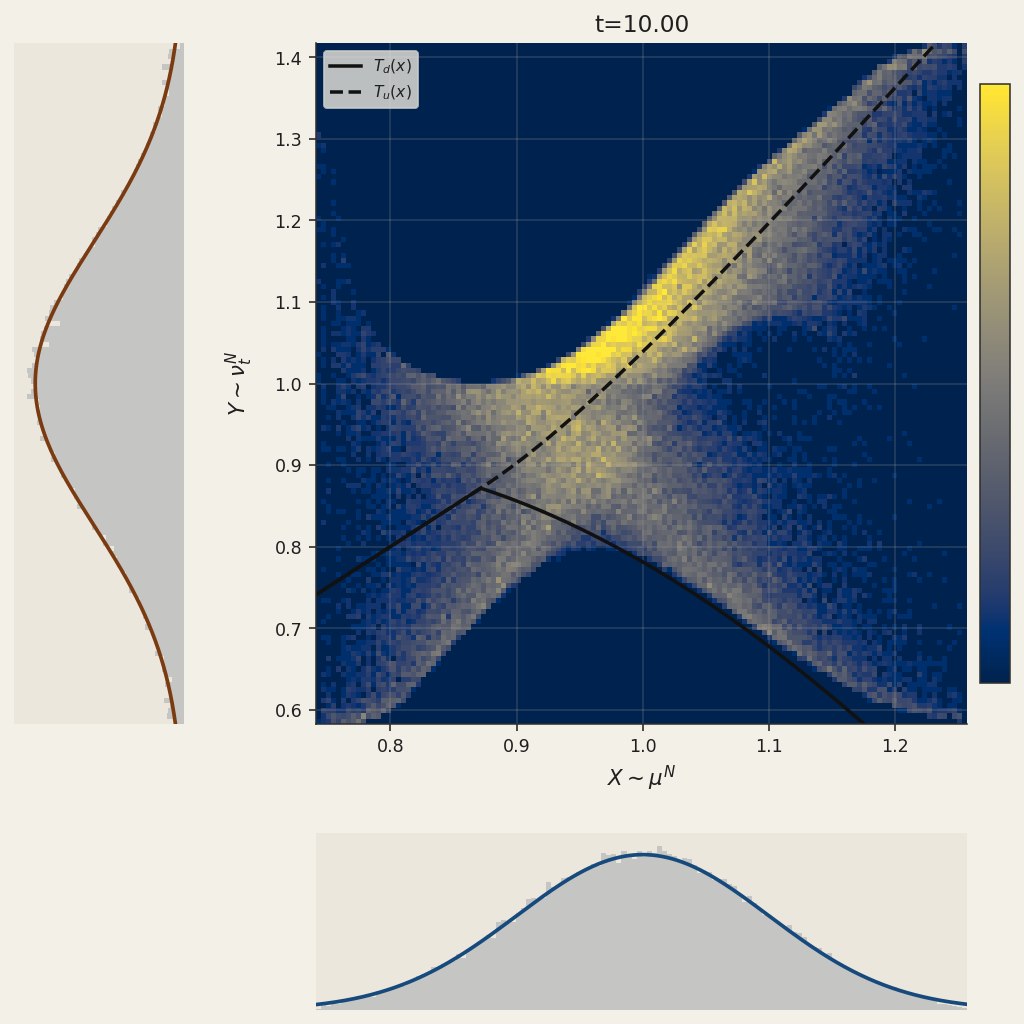}
       \caption{$t=10.00$}
   \end{subfigure}
   \caption{
   Evolution of the empirical coupling for the Gaussian--Gaussian martingale optimal transport benchmark.
   }
   \label{fig:mot-gg-dynamics}
\end{figure}

The histograms show the concentration of the empirical coupling toward the left--monotone martingale coupling, the expected optimizer for the chosen cost under the Spence--Mirrlees condition. The dynamics recover its two-graph structure while keeping the conditional barycenter close to the martingale constraint.

The empirical objective value also converges toward the MOT value. Since $\varepsilon=10^{-6}$, entropic smoothing is negligible at the scale of the plot. The small peaks come from reprojection steps onto the prescribed marginals, which correct the slight deviations caused by discretization and the finite martingale penalty $\lambda_{\rm M}$ during the Euler updates.

\begin{figure}[H]
   \centering
   \includegraphics[width=0.70\textwidth]{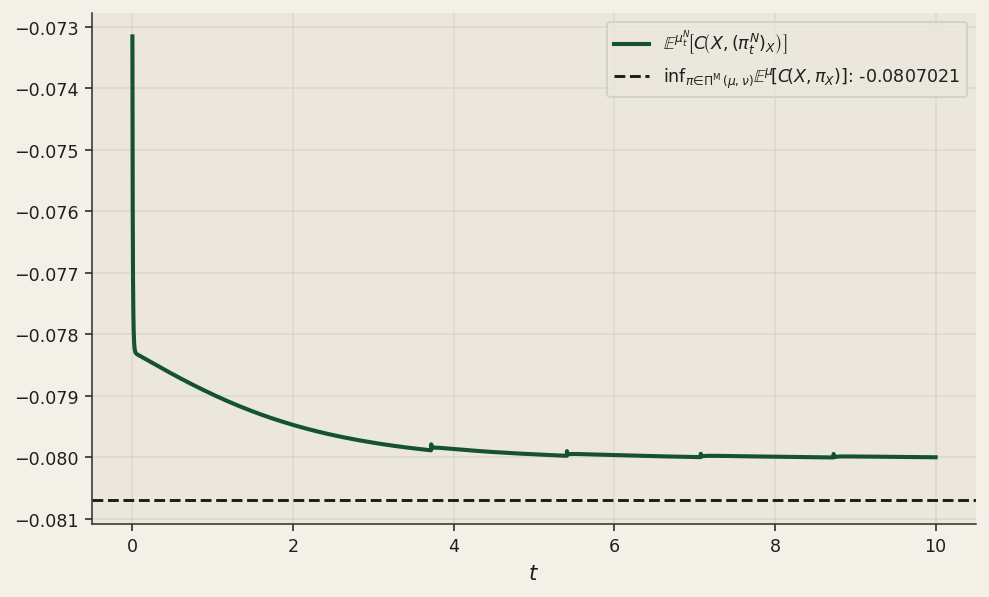}
   \caption{
   Convergence of the empirical objective value for the Gaussian--Gaussian martingale optimal transport benchmark.
   }
   \label{fig:mot-gg-objective}
\end{figure}

\newpage

\section{Well-posedness}\label{sec:wellposed}

We investigate in this section the well-posedness of \eqref{eq:SDE}. Without loss of generality, we only consider the case $\varepsilon =1$.
Let
$$
    a:\R_+\times\Xc\times\R^d\times\Pc(\R^d)\longrightarrow\R^d
$$
be a measurable coefficient and \( \pi_0\in\Pi(\mu,\nu) \).
We consider
{\usetagform{noparen}
\begin{align}
\tag{MV$(a,\pi_0)$}\label{eq:MV-a}
dY_t
&=
\hat b\bigl(t,X,Y_t,\Lc(X,Y_t)\bigr)\,dt
+\sqrt{2}\,dB_t,
\qquad
(X,Y_0)\sim\pi_0,
\\
\hat b(t,x,y,\pi)
&:=
\E^\pi\!\left[
    a(t,X,y,\pi_X)
    \mid Y=y
\right]
-
a(t,x,y,\pi_x)
-
\nabla_y V(y).
\nonumber
\end{align}
}We next introduce the corresponding Fokker--Planck formulation. A weak
continuous family \((\P_t)_{t\in[0,T]}\subset\Pc(\Xc\times\R^d)\) is said to solve
\eqreftag{eq:Fokker-planck-a} if \(\P_0=\pi_0\) and, for every test function
\(h\in C_b(\Xc\times\R^d;\R)\) such that \(h(x,\cdot)\in C_c^\infty(\R^d; \R)\) for all
\(x\in\Xc\), one has

{\usetagform{noparen}
\begin{equation}
\tag{FP$(a,\pi_0)$}\label{eq:Fokker-planck-a}
\E^{\P_t}[h]
=
\E^{\pi_0}[h]
+
\int_0^t
\E^{\P_s}\!\left[
    \nabla_y h(X,Y)\cdot \hat b(s,X,Y,\P_s)
    +
    \Delta_y h(X,Y)
\right]\,ds .
\end{equation}}

\noindent Since the \(X\)-variable does not diffuse and remains frozen at its initial
value, it is useful to introduce the conditional Fokker--Planck equation. If
\[
    \P_t(dx,dy)
    =
    \mu(dx)\,\P_{t,x}(dy),
\]
then, for \(\mu\)-a.e. \(x\in\Xc\), the conditional flow
\((\P_{t,x})_{t\in[0,T]}\) is expected to satisfy, for all
\(h\in C_c^\infty(\R^d ; \R)\),
{\usetagform{noparen}
\begin{equation}
\tag{CFP$(a,\pi_0)$}\label{eq:cond-Fokker-planck-a}
\E^{\P_{t,x}}[h]
=
\E^{\P_{0,x}}[h]
+
\int_0^t
\E^{\P_{s,x}}\!\left[
    \nabla h(Y)\cdot \hat b(s,x,Y,\P_s)
    +
    \Delta h(Y)
\right]\,ds,
\ \
\mu\text{-a.e. }x .
\end{equation}}Equivalently, since
$
(C_c^\infty(\R^d; \R),
    \|f\|_\infty+\|Df\|_\infty+\|D^2f\|_\infty )
$
is separable, one may require the existence of a measurable set
\(\Xc_0\subset\Xc\), with \(\mu(\Xc_0)=1\), such that
\eqreftag{eq:cond-Fokker-planck-a} holds for every \(x\in\Xc_0\), every
\(t\in[0,T]\), and every \(h\in C_c^\infty(\R^d; \R)\).

\medskip

\noindent We now give the corresponding notions of weak solution.

\begin{Definition}\label{Def:weak-sol:MV-a}
A weak solution of \eqreftag{eq:MV-a} consists of a filtered probability space
\((\Omega,\Fc,\F,\P)\) supporting a \(d\)-dimensional Brownian motion \(B\), a
random variable \((X,Y_0)\) with law \(\pi_0\), and an \(\F\)-adapted continuous
process \(Y=(Y_t)_{t\ge0}\), such that \(B\) is independent of \((X,Y_0)\) and the
following conditions hold.

\begin{enumerate}[label=(\roman*), ref=(\roman*)]
    \item \label{weak-sol:MV-a:int}
    There exists \(p>1\) such that, for every \(T>0\),
    \[
    \int_0^T
    \E\left[
    \left|
    \hat A_t
    -
    a\bigl(t,X,Y_t,\Lc(Y_t\mid X)\bigr)
    -
    \nabla_y V(Y_t)
    \right|^p
    \right]\,dt
    <\infty,
    \]
    where $
        \hat A_t
        :=
        \E [
            a (t,X,Y_t,\Lc(Y_t\mid X) )
            \mid Y_t
        ].
    $

    \item \label{weak-sol:MV-a:sde}
    The process \(Y\) satisfies, in the weak sense,
    \[
    dY_t
    =
    \left(
    \hat A_t
    -
    a\bigl(t,X,Y_t,\Lc(Y_t\mid X)\bigr)
    -
    \nabla_y V(Y_t)
    \right)dt
    +
    \sqrt{2}\,dB_t .
    \]
\end{enumerate}
\end{Definition}

\noindent We say that weak solutions of \eqreftag{eq:MV-a} are unique in law if, whenever two weak solutions are defined on possibly different probability spaces with the same initial law \(\pi_0\), the laws of \((X,Y)\) coincide on
$
    \Xc\times C(\R_+;\R^d).
$
Equivalently, since \(X\) is time-independent, one may identify \((X,Y)\) with
the path \(t\mapsto (X,Y_t)\) in \(C(\R_+;\Xc\times\R^d)\) whenever the latter
notation is meaningful.

\medskip

\begin{Definition}\label{Def:weak-sol:CFP-a}
A weak solution of \eqreftag{eq:cond-Fokker-planck-a} is a measurable map
\[
    \tilde \P:\R_+\times\Xc\longrightarrow \Pc(\R^d),
    \qquad
    (t,x)\longmapsto \tilde\P_{t,x},
\]
together with a Borel vector field
$
    \hat A:\R_+\times\R^d\longrightarrow\R^d,
$
such that the following conditions hold.

\begin{enumerate}[label=(\roman*), ref=(\roman*)]
    \item \label{weak-sol:CFP-a:cont}
    For \(\mu\)-a.e. \(x\in\Xc\), the map
    \(t\mapsto \tilde \P_{t,x}\) is weak continuous.

    \item \label{weak-sol:CFP-a:int}
    There exists \(p>1\) such that, for every \(T>0\),
    \[
    \int_0^T
    \int_\Xc
    \int_{\R^d}
    \left|
    \hat A(t,y)
    -
    a\bigl(t,x,y,\tilde\P_{t,x}\bigr)
    -
    \nabla_y V(y)
    \right|^p
    \tilde\P_{t,x}(dy)\,\mu(dx)\,dt
    <\infty.
    \]

    \item \label{weak-sol:CFP-a:fp}
    For every \(h\in C_c^\infty(\R^d; \R)\) and every \(t\ge0\), one has,
    for \(\mu\)-a.e. \(x\in\Xc\),
    \[
    \E^{\tilde\P_{t,x}}[h]
    =
    \E^{\tilde\P_{0,x}}[h]
    +
    \int_0^t
    \E^{\tilde\P_{s,x}}\!\left[
    \nabla h(Y)\cdot
    \left(
    \hat A(s,Y)
    -
    a\bigl(s,x,Y,\tilde\P_{s,x}\bigr)
    -
    \nabla_y V(Y)
    \right)
    +
    \Delta h(Y)
    \right]ds .
    \]

    \item \label{weak-sol:CFP-a:ce}
    If
    $
    \P_t(dx,dy)
    :=
    \mu(dx)\,\tilde\P_{t,x}(dy),
$
    then, for a.e. \(t\ge0\),
    \[
    \hat A(t,y)
    =
    \E^{\P_t}\!\left[
        a\bigl(t,X,y,(\P_t)_X\bigr)
        \mid Y=y
    \right],
    \qquad
    \P_t\text{-a.s.}
    \]
\end{enumerate}
\end{Definition}

\medskip

\begin{Definition}\label{Def:weak-sol:FP-a}
A weak solution of \eqreftag{eq:Fokker-planck-a} is a weak continuous flow
$
    (\P_t)_{t\ge0}\subset\Pc(\Xc\times\R^d),
$
$
    \P_0=\pi_0,
$
together with a Borel vector field
$
    \hat A:\R_+\times\R^d\longrightarrow\R^d,
$
such that the following conditions hold.

\begin{enumerate}[label=(\roman*), ref=(\roman*)]
    \item \label{weak-sol:FP-a:int}
    There exists \(p>1\) such that, for every \(T>0\),
    \[
    \int_0^T
    \int_{\Xc\times\R^d}
    \left|
    \hat A(t,y)
    -
    a\bigl(t,x,y,(\P_t)_x\bigr)
    -
    \nabla_y V(y)
    \right|^p
    \P_t(dx,dy)\,dt
    <\infty.
    \]

    \item \label{weak-sol:FP-a:fp}
    For every \(t\ge0\) and every
    \(h\in C_b(\Xc\times\R^d; \R)\) such that
    \(h(x,\cdot)\in C_c^\infty(\R^d; \R)\) for all \(x\in\Xc\),
    \begin{align*}
    \E^{\P_t}[h]
    =&
    \E^{\P_0}[h]
    +
    \int_0^t
    \E^{\P_s}\!\Big[
    \nabla_y h(X,Y)\cdot
    \big(
    \hat A(s,Y)
    -
    a\bigl(s,X,Y,(\P_s)_X\bigr)
    -
    \nabla_y V(Y)
    \big)\Big]ds
    \\
    &+
    \int_0^t
    \E^{\P_s}\!\Big[ \Delta_y h(X,Y)
    \Big]ds .
\end{align*}

    \item[(iii)] \label{weak-sol:FP-a:ce}
    For a.e. \(t\ge0\),
    \[
    \hat A(t,y)
    =
    \E^{\P_t}\!\left[
        a\bigl(t,X,y,(\P_t)_X\bigr)
        \mid Y=y
    \right],
    \qquad
    \P_t\text{-a.s.}
    \]
\end{enumerate}
\end{Definition}

\begin{Remark}
In the previous definitions, the exponent \(p>1\) is imposed to ensure the
equivalence between the probabilistic and Fokker--Planck formulations; see
Proposition~\ref{Prop:mimick} below. The individual formulations are meaningful
under weaker integrability assumptions.
\end{Remark}

\medskip

We will systematically identify conditional expectations with Borel-measurable
versions. This is justified by the following standard measurable selection
result; see Proposition~3.3 in \cite{conforti2023projected} and Proposition~5.1
in \cite{brunick2013mimicking}.

\begin{Lemma}\label{Prop:measurable selection}
Let \((\Omega,\Fc,\P)\) be a probability space supporting two processes
\((X_t)_{t\ge0}\) and \((Z_t)_{t\ge0}\), jointly measurable in \((t,\omega)\).
Assume that, for every \(T>0\),
$
    \E [
    \int_0^T |Z_s|\,ds
    ]<\infty.
$
Then there exists a Borel-measurable function
$
    \hat z:\R_+\times\R^d\longrightarrow\R^d
$
such that
\[
    \hat z(t,X_t)
    =
    \E[Z_t\mid X_t],
    \qquad
    \text{for a.e. }t\ge0,\ \P\text{-a.s.}
\]
\end{Lemma}

\medskip

Throughout this section, the measures \(\mu\) and \(\nu\) satisfy the following
standing assumptions.

\begin{Assumption}\label{ass:Laws existence}
Let \(\Xc\) be a complete metric space, \(\mu\in\Pc(\Xc)\), and $\nu \in \Pc(\R^d)$ with
\[
    \nu(dy)=e^{-V(y)}\,dy
\]
for some \(C^2\) function \(V:\R^d\to\R\). We assume the following conditions.

\begin{enumerate}[label=(\roman*), ref=(\roman*)]
    \item \label{Well-posed:SDE Ref}
    The following Langevin equation is well-posed in law:
    \begin{equation}\label{eqdef:refLangevin}
        dY_t
        =
        -\nabla V(Y_t)\,dt
        +
        \sqrt{2}\,dB_t,
        \qquad
        Y_0\sim\nu.
\end{equation}

    \item \label{ass:Laws existence:p-moment}
    There exists \(p>d+2\) such that
    $
        \nabla V\in L^p(\nu),
    $ and $
        \mu\in\Pc_{p}(\Xc)$.

    \item \label{ass:Laws existence:exponential}
    There exists $\alpha_\nu>0$ such that
    \[
        \int_{\R^d}e^{\alpha_\nu|y|^2}\,\nu(dy)<\infty.
    \]
\end{enumerate}
\end{Assumption}

\begin{Remark}
The lower Hessian bound in
Assumption~\ref{ass:marginals}(i) is sufficient for the well-posedness
of the reference Langevin equation in
Assumption~\ref{ass:Laws existence}(i). Indeed, it gives
\[
    \langle y,-\nabla V(y)\rangle\le C(1+|y|^2),
\]
which, together with
$V\in C^2$, ensures nonexplosion and uniqueness. Its normalized Gibbs
law $\nu(dy)=e^{-V(y)}dy$ is invariant.
The exponential moment in
Assumption~\ref{ass:Laws existence}\ref{ass:Laws existence:exponential} is the additional
integrability used in the weighted Pinsker estimates below.
\end{Remark}

\subsection{Preliminary results}
In order to prove the existence and uniqueness of weak solutions to
\eqreftag{eq:MV-a}, we first establish a mimicking result, which is a consequence of Theorem 1.3 in \cite{lacker2022superposition}.

\begin{Proposition} \label{Prop:mimick}
Let
\[
a : \R_+ \times \Xc \times \R^d \times \Pc(\R^d)\to \R^d.
\]
The following statements are equivalent:
\begin{enumerate}[label=(\roman*), ref=(\roman*)]
    \item \label{ass:existMV-a}
    Equation \eqreftag{eq:MV-a} admits a weak solution in the sense of Definition \ref{Def:weak-sol:MV-a}.
Furthermore, for $t\ge0$,
    $
    \Lc(X,Y_t)\in\Pi(\mu,\nu).
    $
    \item \label{ass:existCFP-a}
    Equation \eqreftag{eq:cond-Fokker-planck-a} admits a weak solution in the sense of Definition \ref{Def:weak-sol:CFP-a}.
    Furthermore for $t\ge0$,
    $
    \P_t(dx,dy)
    :=
    \tilde{\P}_{t,x}(dy)\,\mu(dx) \in \Pi(\mu,\nu)
    .$

    \item \label{ass:existFP-a}
    Equation \eqreftag{eq:Fokker-planck-a} admits a weak solution in the sense of Definition \ref{Def:weak-sol:FP-a}.
    for $t\ge0$,
    $
    \P_t\in\Pi(\mu,\nu).
    $
\end{enumerate}
\end{Proposition}

\begin{Remark}
{\rm
The additional conclusion that all solutions satisfy the coupling constraint follows from Lemma 3.2 of \cite{conforti2023projected}.
}
\end{Remark}
\begin{proof}
The equivalence between \ref{ass:existCFP-a} and \ref{ass:existFP-a} is
immediate by disintegration. We therefore only prove the equivalence between
\ref{ass:existMV-a} and \ref{ass:existCFP-a}. Fix $T>0$ and work on the interval
$[0,T]$.

\medskip

\noindent
Assume first that \eqreftag{eq:MV-a} admits a weak solution
$(\Omega,\Fc,\F,\P,B,X,Y_0,Y,\hat A)$ in
the sense of Definition~\ref{Def:weak-sol:MV-a}. By Lemma~\ref{Prop:measurable selection},
there exists a Borel map, still denoted by
$
\hat A : [0,T]\times \R^d \longrightarrow \R^d,
$
such that
\vspace{-2mm}
\[
\hat A_t=\hat A(t,Y_t)
\qquad\text{a.s. for a.e.\ }t\in[0,T].
\vspace{-2mm}
\]
We then define
$
\tilde{\P}_{t,x}:=\Lc(Y_t\mid X=x),
$ and $
\P_t(dx,dy):=
\tilde{\P}_{t,x}(dy)\,\mu(dx).
$
By the continuity of the trajectories of $Y$, the map $t\mapsto \tilde{\P}_{t,x}$ is weakly continuous for $\mu$-a.e.\ $x\in\Xc$, which proves \ref{weak-sol:CFP-a:cont} in Definition~\ref{Def:weak-sol:CFP-a}.
We further introduce the drift
\vspace{-1.5mm}
\[
\mathbf b(t,x,y)
:=
\hat A(t,y)
-
a\bigl(t,x,y,\tilde{\P}_{t,x}\bigr)
-
\nabla_y V(y).
\vspace{-1mm}
\]
The integrability condition~\ref{weak-sol:MV-a:int} of Definition~\ref{Def:weak-sol:MV-a} implies that the field
$\mathbf b$ satisfies the integrability property \ref{weak-sol:CFP-a:int} of Definition~\ref{Def:weak-sol:CFP-a}.
Moreover, since the SDE~\ref{weak-sol:MV-a:sde} of Definition~\ref{Def:weak-sol:MV-a} holds,
the process $Y$ satisfies
\[
dY_t
=
\mathbf b(t,X,Y_t)\,dt+\sqrt{2}\,dB_t.
\]
Applying It\^o's formula to $h(Y_t)$, for $h\in C_c^\infty(\R^d; \R)$, and then
taking the conditional expectation with respect to $X$, which is independent of
$B$, we obtain
\[
\E^{\tilde{\P}_{t,x}}[h]
=
\E^{\tilde{\P}_{0,x}}[h]
+
\int_0^t
\E^{\tilde{\P}_{s,x}}\!\Big[
\nabla h(Y)\cdot \mathbf b(s,x,Y)+\Delta h(Y)
\Big]\,
ds,
\qquad
\mu\text{-a.s. in }x.
\]
This proves that the conditional Fokker--Planck equation~\ref{weak-sol:CFP-a:fp} of
Definition~\ref{Def:weak-sol:CFP-a} holds.

\noindent Finally, since $\P_t=\Lc(X,Y_t)$, we deduce from
item~\ref{weak-sol:MV-a:int} of Definition~\ref{Def:weak-sol:MV-a} that, for
a.e.\ $t\in[0,T]$,
\[
\hat A(t,Y_t)
=
\E^\P\Big[
a\bigl(t,X,Y_t,(\P_t)_X\bigr)\mid Y_t
\Big]
\qquad
\P\text{-a.s.}
\]
This is precisely the conditional expectation property~\ref{weak-sol:CFP-a:ce} of
Definition~\ref{Def:weak-sol:CFP-a}. Therefore \eqreftag{eq:cond-Fokker-planck-a} admits a weak solution in the sense of Definition \ref{Def:weak-sol:CFP-a}.

\medskip

\noindent
Conversely, assume that \eqreftag{eq:cond-Fokker-planck-a} admits a weak solution $(\tilde{\P},\hat A)$ in the sense of Definition~\ref{Def:weak-sol:CFP-a}.
Define
\[
\P_t(dx,dy):=
\tilde{\P}_{t,x}(dy)\,\mu(dx),
\,\,
\text{and}
\,\,
\mathbf b(t,x,y)
:=
\hat A(t,y)
-
a\bigl(t,x,y,\tilde{\P}_{t,x}\bigr)
-
\nabla_y V(y).
\]
By the integrability condition~\ref{weak-sol:CFP-a:int} of Definition~\ref{Def:weak-sol:CFP-a}, the vector field $\mathbf b$ satisfies the required integrability condition~\ref{weak-sol:MV-a:int} of Definition~\ref{Def:weak-sol:MV-a}.
Moreover, the conditional Fokker--Planck equation~\ref{weak-sol:CFP-a:fp} of
Definition~\ref{Def:weak-sol:CFP-a} yields that, for every $h\in C_c^\infty(\R^d; \R)$ and
every $t\in[0,T]$,
\[
\E^{\tilde{\P}_{t,x}}[h]
=
\E^{\tilde{\P}_{0,x}}[h]
+
\int_0^t
\E^{\tilde{\P}_{s,x}}\!\Big[
\nabla h(Y)\cdot \mathbf b(s,x,Y)+\Delta h(Y)
\Big]\,
ds,
\qquad
\mu\text{-a.s. in }x.
\]
In addition, the continuity condition~\ref{weak-sol:CFP-a:cont} of Definition~\ref{Def:weak-sol:CFP-a} gives the weak continuity of $t\mapsto \tilde{\P}_{t,x}$ for $\mu$-a.e.\ $x$.

\noindent We may therefore apply Theorem 1.3 of \cite{lacker2022superposition} to the
conditional Fokker--Planck equation associated with the drift $\mathbf b$. This
yields a filtered probability space $(\Omega,\Fc,\F,\P)$ supporting a
$d$-dimensional Brownian motion $B$, random variables $(X,Y_0)\sim\pi_0$ with
$B$ independent of $(X,Y_0)$, and an $\F$-adapted process $Y$ such that
\[
dY_t
=
\mathbf b(t,X,Y_t)\,dt+\sqrt{2}\,dB_t,
\]
that is,
\[
dY_t
=
\Big(
\hat A(t,Y_t)
-
a\bigl(t,X,Y_t,\Lc(Y_t\mid X)\bigr)
-
\nabla_y V(Y_t)
\Big)\,dt
+
\sqrt{2}\,dB_t,
\]
and such that
\[
\Lc(Y_t \vert X) = \tilde{\P}_{t,X}, \,\, \text{and}\,\,
\Lc(X,Y_t)=\P_t
\qquad
\text{for all }t\in[0,T].
\]
In particular, the SDE~\ref{weak-sol:MV-a:sde} of
Definition~\ref{Def:weak-sol:MV-a} is satisfied.
Since $\Lc(X,Y_t)=\P_t$,
item~\ref{weak-sol:CFP-a:ce} of Definition~\ref{Def:weak-sol:CFP-a} implies
that
\[
\hat A(t,Y_t)
=
\E^\P\Big[
a\bigl(t,X,Y_t,\tilde{\P}_{t,X}\bigr)\mid Y_t
\Big]
\qquad
\P\text{-a.s. for a.e.\ }t,
\]
which is exactly the conditional expectation property~\ref{weak-sol:MV-a:int} of
Definition~\ref{Def:weak-sol:MV-a}. Finally,
item~\ref{weak-sol:CFP-a:int} of Definition~\ref{Def:weak-sol:CFP-a} yields
item~\ref{weak-sol:MV-a:int} of Definition~\ref{Def:weak-sol:MV-a}. Therefore
\ref{ass:existMV-a} holds.
\end{proof}

We now turn to the following estimates, which are based on
Corollary 6.3.2, Corollary 6.4.3, and Theorem 6.5.4 of \cite{bogachev2022fokker}.

\begin{Proposition} \label{prop:estim}
Let $p>d+2$, let $\P_0\in\Pc(\Xc\times\R^d)$, and let
\[
b:\R_+\times\Xc\times\R^d\longrightarrow\R^d
\]
be a Borel-measurable vector field. Assume that there exists a weak solution to
\[
dY_t
=
b(t,X,Y_t)\,dt+\sqrt{2} \,dB_t,
\qquad
(X,Y_0)\sim\P_0,
\qquad
(X, Y_0) \independent B,
\]
where $\independent$ denotes the indepenance.
For $T>0$ and $R>0$, define
\begin{equation} \label{eqdef:constant control}
C_1(R,T,b)
:=
\int_0^T
\left\{
\E\Big[
|b(t,X,Y_t)|^p\mathbf{1}_{\{|Y_t|\le R\}}
\Big]
+
\int_{\Xc} \int_{B(0,R)} |b(t,x,y)|^p\,dy\,\mu(dx)
\right\}
dt.
\end{equation}
Assume that $C_1(R,T,b)<\infty$ for every $R>0$. Then the following hold.
\begin{enumerate}[label=(\roman*), ref=(\roman*)]
    \item \label{estim density}
    For every $t\in(0,T]$, the law
    \[
    \P_t:=\Lc(X,Y_t)
    \]
admits a density with respect to $\mu \otimes \mathrm{L}$,
    still denoted by $\tilde{\P}_{t,x}(y)$.
    Moreover, for every $R>0$ and every
    $t\in(0,T]$, there exists a constant $K_{T,R}>0$ such that
    \begin{equation} \label{eq:density-bound-Pt}
    \big\|
    \tilde{\P}_t\,\mathbf{1}_{\{|y|\le R\}}
    \big\|_{\L^p(\mu\otimes \mathrm{L} ) }
    \le
    K_{T,R} \bigl(1+C_1(R,T,b)\bigr).
    \end{equation}

    \item \label{estim cond Exp}
    Assume furthermore that the second marginal of $\P_t$ is given by
    \[
    \Lc(Y_t)=\nu(dy)=e^{-V(y)} \,dy, \quad \forall \, t \in \R_+.
    \]
    For every bounded Borel function $\phi:\Xc\to\R$, define
    \[
    w_\phi(t,y):=\E[\phi(X)\mid Y_t=y].
    \]
    Equivalently, for $\nu$-a.e.\ $y$,
    \begin{equation} \label{eq:def-wphi}
    w_\phi(t,y)
    =
    \int_{\Xc} \phi(x)\,\tilde{\P}_{t,x}(y)\,\mu(dx)\,e^{V(y)}.
    \end{equation}
    Then, for every $R>0$ and every $\delta>0$, there exist constants
    $\alpha,\beta,K>0$, depending only on $C_1(R,T,b)$, $R$, and $\delta$, such
    that for all $y,y'\in B(0,R)$ and all $t,t'\in[\delta,T]$ satisfying
    $|t-t'|\le1$,
    \begin{equation} \label{eq:holder-bound-wphi}
    |w_\phi(t',y')-w_\phi(t,y)|
    \le
    K\|\phi\|_\infty
    \Big(
    |y-y'|^\alpha + |t-t'|^\beta
    \Big).
    \end{equation}
\end{enumerate}
\end{Proposition}

\begin{proof}
We begin with point \ref{estim density}.
By disintegration, for each $t\ge0$ there exists a measurable family
$(\tilde{\P}_{t,x})_{x\in\Xc} \subset\Pc(\R^d)$ such that
\[
\P_t(dx,dy)
=
\tilde{\P}_{t,x}(dy)\,\mu(dx).
\]
Moreover, by It\^o's formula applied to $h(Y_t)$, together with the independence
of $X$ and $B$, we obtain that, for $\mu$-a.e.\ $x\in\Xc$ and every
$h\in C_c^\infty(\R^d; \R)$,
\[
\int_{\R^d} h(y)\,\tilde{\P}_{t,x}(dy)
=
\int_{\R^d} h(y)\,\tilde{\P}_{0,x}(dy)
+
\int_0^t
\int_{\R^d}
\bigl(
\nabla h(y)\cdot b(s,x,y)+\Delta h(y)
\bigr)
\tilde{\P}_{s,x}(dy)\,ds.
\]
Thus, for $\mu$-a.e.\ $x$, the curve $t\mapsto\tilde{\P}_{t,x}$ is a
solution of the Fokker--Planck equation with drift $b(\cdot,x,\cdot)$.
Furthermore, by the assumed integrability \eqref{eqdef:constant control}, for every $R>0$,
\[
\int_0^T\int_{B(0,R)}
|b(t,x,y)|^p\,\tilde{\P}_{t,x}(dy)\,dt
<\infty,
\qquad
\text{for $\mu$-a.e.\ }x\in\Xc.
\]
Hence, by Theorem 6.5.4 in \cite{bogachev2022fokker}, for $\mu$-a.e.\ $x$ and
every $t\in(0,T]$, the measure $\tilde{\P}_{t,x}$ admits a density with respect to
Lebesgue measure\footnote{The set of probability measures admitting a density is Borel-measurable in $(\Pc_1(\R^d), \Wc_1)$ since it can be written as $\bigcap_{m\ge1}\ \bigcup_{n\ge1}\ \bigcap_{\substack{U \text{ is open} \\ \text{Diameter}(U)<2^{-n}}}
\{\eta:\eta(U)<2^{-m}\}$, and for every open set $U$, the map $\eta \mapsto \eta(U)$ is l.s.c.}.
Furthermore, also by Theorem 6.5.4 in \cite{bogachev2022fokker}, for every $R>0$ and every
compact interval $[s_1,s_2]\subset(0,T]$, there exists a constant $K_{R,T}>0$
such that
\[
\|
\tilde{\P}_{\cdot,x}
\|_{\L^\infty([s_1,s_2]\times B(0,R))}
\le
K_{R,T}
\left(
1+
\int_{s_1}^{s_2}
\int_{B(0,R)}
|b(t,x,y)|^p\,\tilde{\P}_{t,x}(y)\,dy\,dt
\right) \,\, \mu\text{-a.s.}
\]
Integrating with respect to $\mu(dx)$, we infer that, for every $t\in(0,T]$,
\[
\big\|
\tilde{\P}_t\,\mathbf{1}_{\{|y|\le R\}}
\big\|_{\L^p(\mu\otimes \mathrm{L})}
\le
K_{T,R}
\left(
1+
\int_0^T
\int_{\Xc}
\int_{B(0,R)}
|b(t,x,y)|^p\,\tilde{\P}_{t,x}(y)\,dy\,\mu(dx)\,dt
\right).
\]
Since
\[
\int_{\Xc}
\int_{B(0,R)}
|b(t,x,y)|^p\,\tilde{\P}_{t,x}(y)\,dy\,\mu(dx)
=
\E\Big[
|b(t,X,Y_t)|^p\mathbf{1}_{\{|Y_t|\le R\}}
\Big],
\]
this yields \eqref{eq:density-bound-Pt}.

\medskip

We now prove point \ref{estim cond Exp}. Let $\phi:\Xc\to\R$ be bounded and Borel-measurable. Since the conclusion of the theorem is invariant under adding a constant to $\phi$, we may assume that $1 \leq \phi$.
Let $w_\phi$ be given by \eqref{eq:def-wphi}; by the defining property of conditional expectation, it also satisfies $w_\phi \geq 1$.
For every bounded Borel function $g:\R^d\to\R$,
\[
\E\big[g(Y_t)w_\phi(t,Y_t)\big]
=
\E\big[g(Y_t)\phi(X)\big].
\]
In particular, for every $g\in C_c^\infty(\R^d; \R)$,
\[
\int_{\R^d} g(y)\,w_\phi(t,y)\,e^{-V(y)} \,dy
=
\E\big[g(Y_t)\phi(X)\big].
\]
Applying It\^o's formula to $h(t,Y_t)\phi(X)$, for a smooth compactly supported test function $h$, and using again that $X$ is
independent of $B$, we obtain
\begin{align*}
\frac{d}{dt} \E\big[h(t,Y_t)\phi(X)\big]
&=
\E\Big[
\phi(X) \big( \partial_t + \Delta \big) h(t,Y_t)
+
\phi(X)\nabla h(t,Y_t)\cdot b(t,X,Y_t)
\Big]
\\
&=
\int_{\R^d}
\big( \partial_t + \Delta \big)h(t,y)\,w_\phi(t,y)\,e^{-V(y)} \,dy
+
\int_{\R^d}
\nabla h(t,y)\cdot \gamma_\phi(t,y)\,dy,
\end{align*}
where
\[
\gamma_\phi(t,y)
:=
\int_{\Xc}
\phi(x)\,b(t,x,y)\,\tilde{\P}_t(x,y)\,\mu(dx).
\]
By the triangle inequality and since $w_\phi \geq1 $,
\[
\vert
\nabla h(t,y)\cdot \gamma_\phi(t,y)
\vert
\leq
\|\phi\|_\infty
\vert \nabla h(t,y) \vert \beta(t,y) w_{\phi}(t,y)e^{-V(y)},
\]
where
\[
\beta(t,y)
:=
\int_{\Xc}
|b(t,x,y)|\,\tilde{\P}_t(x,y)\,\mu(dx)\,e^{V(y)}.
\]
As $V\in C^2$, the functions $e^V$ and $e^{-V}$ are locally bounded.
Hence, for every $U_{R,T} \Subset [0,T]\times\R^d$, $w_{\phi}$
satisfies
\begin{align}\label{eq:conditionHolder}
\Big\vert
\int_{U_{R,T}} \big[ \partial_t h + \Delta h \big] w_{\phi} \,dy  dt
\Big\vert
\leq
C_{R,T}
\int_{U_{R,T}}
\vert \nabla h(t,y) \vert \beta(t,y) w_{\phi}(t,y)
dy dt.
\end{align}
Also, for every $R>0$,
\[
\int_0^T\int_{B(0,R)}
|\beta(t,y)|^p\,dy\,dt
<\infty.
\]
Together with \eqref{eq:conditionHolder} it verifies the local
integrability assumptions required by the interior parabolic regularity theory.
We may therefore apply the corresponding local Hölder estimates of Theorem 6.4.2 of \cite{bogachev2022fokker} and conclude that,
for every $R>0$ and every $\delta>0$, there exist constants $\alpha,\beta,K>0$,
depending only on $C_1(R,T,b)$, $R$, and $\delta$, such that for all
$y,y'\in B(0,R)$ and all $t,t'\in[\delta,T]$ with $|t-t'|\le1$,
\[
|w_\phi(t',y')-w_\phi(t,y)|
\le
K\|\phi\|_\infty
\Big(
|y-y'|^\alpha + |t-t'|^\beta
\Big).
\]
This proves \eqref{eq:holder-bound-wphi}.
\end{proof}

We now give the analogue of Lemma 3.5 in \cite{conforti2023projected}.
\begin{Lemma} \label{prop:entropy estimates}
Let
\[
b^1,b^2:\R_+\times\Xc\times\R^d\longrightarrow\R^d
\]
be Borel-measurable vector fields, and assume that $b^2$ is bounded.
For $i=1,2$, let $ \P^i := \Lc(X^i,Y^i)$ be a weak solution of
\[
dY_t^i
=
b^i(t,X^i,Y_t^i)\,dt
+
\sqrt{2} \,dB_t,
\]
satisfying that $(X^i,Y_0^i)$ is independent of the driving Brownian motion.
Let $\pi^1_0 :=\Lc(X^1,Y_0^1)$ and $\pi_0^2 :=\Lc(X^2,Y_0^2)$, and assume that
\[
H(\pi_0^1\mid \pi_0^2)<\infty
\,\,
\text{and}
\,\,
\int_0^T
\E^{\P^1}\!\Big[
|b^1(t,X^1,Y_t^1)-b^2(t,X^1,Y_t^1)|^2
\Big]\,
dt
<\infty.
\]
Then
\[
H(\P^1\mid\P^2)
=
H(\pi_0^1 \vert \pi_0^2 ) +
\frac{1}{4}
\int_0^T
\E^{\P^1} \!\Big[
\big|b^1(t,X^1,Y_t^1)-b^2(t,X^1,Y_t^1)\big|^2
\Big]\,
dt.
\]
\end{Lemma}

\begin{proof}
This is a consequence of the chain rule for relative entropy and Girsanov's theorem.
First, by the chain rule for relative entropy, conditioning on the initial value gives
\[
H(\P^1\mid\P^2)
=
H(\pi_0^1 \mid \pi_0^2 ) +
\int_{\Xc \times \R^d}
H(\P^1_{x,y} \mid \P^2_{x,y})
\,\pi_0^1 ( dx, dy ),
\]
where, by the independence of $B$ and $(X,Y_0)$, $\P^i_{x,y}$ denotes the law of a weak solution to
\[
dY_t^{i,y}
=
b^i(t,x,Y_t^{i,y})\,dt
+
\sqrt{2}\,dB_t,
\qquad
Y_0^{i,y}=y.
\]
Now, since
\[
\int_0^T
\E^{\P^1_{x,y}} \!\Big[
\big|b^1(t,x,Y_t^{1,y})-b^2(t,x,Y_t^{1,y})\big|^2
\Big]\,
dt
< \infty,
\]
Lemma~3.5 of \cite{conforti2023projected} yields
\[
H(\P^1_{x,y} \mid \P^2_{x,y})
=
\frac{1}{4}
\int_0^T
\E^{\P^1_{x,y}} \!\Big[
\big|b^1(t,x,Y_t^{1,y})-b^2(t,x,Y_t^{1,y})\big|^2
\Big]\,
dt.
\]
Finally, since
\begin{align*}
&\int_{\Xc \times \R^d}
\int_0^T
\E^{\P^1_{x,y}} \!\Big[
\big|b^1-b^2\big|^2(t,x,Y_t^{1,y})
\Big]\,
dt\,
\pi_0^1 ( dx, dy )
=
\int_0^T
\E^{\P^1} \!\Big[
\big|b^1- b^2\big|^2(t,X^1,Y_t^1)
\Big]\,
dt,
\end{align*}
we obtain the desired result.
\end{proof}

\noindent We finally state several convergence results that will be used repeatedly in the sequel.
The first one is a consequence of a standard gluing argument; see, for instance, Lemma 5.5 in
\cite{santambrogio20151}.

\begin{Lemma} \label{prop:gluing lemma}
Let $\Xc$ and $\Yc$ be complete metric spaces, and let
$(\gamma_n)_n\subset\Pc_p(\Xc)$ and $(\lambda_n)_n\subset\Pc_p(\Yc)$.
Assume that there exist $\gamma\in\Pc_p(\Xc)$ and $\lambda\in\Pc_p(\Yc)$ such that
\[
\Wc_p(\gamma_n,\gamma)\longrightarrow0,
\qquad
\Wc_p(\lambda_n,\lambda)\longrightarrow0.
\]
Then, for every $\pi\in\Pi(\gamma,\lambda)$, there exists a sequence
$(\pi_n)_n$ with $\pi_n\in\Pi(\gamma_n,\lambda_n)$ such that
\vspace{-2mm}
\[
\Wc_p(\pi_n,\pi)\longrightarrow0.
\vspace{-2mm}
\]
\end{Lemma}

\begin{Lemma} \label{prop:bootstrap}
Let $\Zc$ be a complete metric space, and let $(\pi_n)_n\subset\Pc(\Zc)$.
Assume that there exists a Borel measure $\lambda$ on $\Zc$ such that, for every
$n\ge1$, the measure $\pi_n$ is absolutely continuous with respect to $\lambda$,
with density
\[
f_n:=\frac{d \pi_n}{d \lambda}.
\]
Assume moreover that there exists $p>1$ such that, for every compact set
$K\subset\Zc$, there exists a constant $C_K>0$ satisfying
\vspace{-2mm}
\[
\|f_n\mathbf{1}_K\|_{\L^p(\lambda)} \le C_K,
\qquad
\text{for all }n\ge1.
\vspace{-2mm}
\]
If $\pi_n\to\pi$ weakly, then, for every bounded Borel-measurable function $g:\Zc\to\R$,
\vspace{-2mm}
\[
\int_{\Zc} g(z)\,\pi_n(dz)
\longrightarrow
\int_{\Zc} g(z)\,\pi(dz).
\vspace{-2mm}
\]

\end{Lemma}
\begin{proof}
Let $K\subset\Zc$ be compact. By assumption, $(f_n\mathbf{1}_K)_n$ is bounded in
$\L^p(K,\lambda)$. Since $1<p<\infty$, the space $\L^p(K,\lambda)$ is reflexive,
hence $(f_n\mathbf{1}_K)_n$ is relatively weakly compact.

\noindent Let $h_K\in\L^p(K,\lambda)$ be a weak cluster point of $(f_n\mathbf{1}_K)_n$.
Then, up to extraction,
\[
f_n\mathbf{1}_K \rightharpoonup h_K
\qquad
\text{weakly in }\L^p(K,\lambda).
\]
If $g:\Zc\to\R$ is continuous and supported in $K$, then $g\in \L^{p'}(K,\lambda)$
since $\lambda$ is Borel, and therefore
\[
\int_{\Zc} g(z)\,\pi_n(dz)
=
\int_K g(z)f_n(z)\,\lambda(dz)
\longrightarrow
\int_K g(z)h_K(z)\,\lambda(dz).
\]
On the other hand, since $\pi_n\to\pi$ weakly,
\vspace{-2mm}
\[
\int_{\Zc} g(z)\,\pi_n(dz)
\longrightarrow
\int_{\Zc} g(z)\,\pi(dz).
\vspace{-2mm}
\]
Hence
$
\int_{\Zc} g(z)\,\pi(dz)
=
\int_K g(z)h_K(z)\,\lambda(dz)
$
for every continuous function $g$ supported in $K$. Thus the restriction of $\pi$
to $K$ coincides with $h_K\lambda$. In particular, $\pi_{|K}\ll\lambda$, and the
density $h_K$ is uniquely determined $\lambda$-a.e. Therefore $(f_n\mathbf{1}_K)_n$
admits a unique weak cluster point in $\L^p(K,\lambda)$, which implies that
\vspace{-3mm}
\[
f_n\mathbf{1}_K \rightharpoonup f\mathbf{1}_K
\text{ weakly in }\L^p(K,\lambda),
\text{ where $f:=\frac{d\pi}{d\lambda}$.}
\vspace{-2mm}
\]

Let now $g:\Zc\to\R$ be bounded and Borel-measurable. Fix $\varepsilon>0$. Since
$\pi_n\to\pi$ weakly, the family $(\pi_n)_n\cup\{\pi\}$ is tight. Hence there
exists a compact set $K_\varepsilon\subset\Zc$ such that
\[
\sup_n \pi_n(\Zc\setminus K_\varepsilon)+\pi(\Zc\setminus K_\varepsilon)\le\varepsilon.
\]
Then
$$
\left|
\int_{\Zc} g\,d\pi_n-\int_{\Zc} g\,d\pi
\right|
\le
\left|
\int_{K_\varepsilon} g(z)\bigl(f_n(z)-f(z)\bigr)\,\lambda(dz)
\right|
+
2\|g\|_\infty\varepsilon.
$$
Since $\lambda$ is Radon, $\lambda(K_\varepsilon)<\infty$, hence
$g\mathbf{1}_{K_\varepsilon}\in \L^{p'}(\lambda)$. By the weak convergence of
$f_n\mathbf{1}_{K_\varepsilon}$ to $f\mathbf{1}_{K_\varepsilon}$ in
$\L^p(K_\varepsilon,\lambda)$, we get
\vspace{-2mm}
\[
\int_{K_\varepsilon} g(z)f_n(z)\,\lambda(dz)
\longrightarrow
\int_{K_\varepsilon} g(z)f(z)\,\lambda(dz).
\vspace{-2mm}
\]
Therefore
\vspace{-2mm}
\[
\limsup_{n\to\infty}
\left|
\int_{\Zc} g\,d\pi_n-\int_{\Zc} g\,d\pi
\right|
\le
2\|g\|_\infty\varepsilon.
\]
Letting $\varepsilon\to0$ concludes the proof.
\end{proof}

\begin{Proposition} \label{prop:bootstrap cond exp}
Let $\Xc$ and $\Yc$ be complete metric spaces, let $\gamma\in\Pc(\Xc)$ and
$\lambda\in\Pc(\Yc)$, and let $(\pi_n)_n\subset\Pi(\gamma,\lambda)$.
Assume that the following conditions hold:
\begin{enumerate}[label=(\roman*)]
    \item For every $n\ge1$, the measure $\pi_n$ is absolutely continuous with
    respect to $\gamma\otimes\lambda$, with density
    \(
    f_n:=\frac{d\pi_n}{d(\gamma\otimes\lambda)}.
    \)

    \item There exists $\alpha_1>1$ such that, for every compact set
    $K\subset\Xc\times\Yc$, there exists a constant $C_K>0$ satisfying
    \(
    \|f_n\mathbf{1}_K\|_{\L^{\alpha_1}(\gamma\otimes\lambda)} \le C_K,
    \
    \text{for all }n\ge1.
    \)

    \item We have $\pi_n\to\pi$ weakly for some $\pi\in\Pi(\gamma,\lambda)$.

    \item There exist $\alpha_2>1$ such that
    \[
    \E^{\pi_n}[g(X)\mid Y]
    \longrightarrow
    \E^\pi[g(X)\mid Y]
    \qquad
    \text{in } \ \L^{\alpha_2}(\lambda),
    \qquad
    \text{for every }g\in\L^\infty(\gamma).
    \]

    \item The function $w:\Xc\times\Yc\to\R$ is Borel-measurable and satisfies, with some $\alpha_3 >\alpha_2$
    \begin{equation}\label{eqdef:estimates w}
    \sup_n \E^{\pi_n}[|w(X,Y)|^{\alpha_3}]
    + \E^\pi[|w(X,Y)|^{\alpha_3}] < \infty.
    \end{equation}

    \item The exponents satisfy
    \(
    \frac{\alpha_2}{\alpha_3}+\frac{1}{\alpha_1}<1 \), equivalently, \(
    \alpha_2\,\alpha_1' < \alpha_3.\)
\end{enumerate}
Then it holds that
\[
\E^{\pi_n}[w(X,Y)\mid Y]
\longrightarrow
\E^\pi[w(X,Y)\mid Y]
\qquad
\text{in } \ \ \L^{\alpha_2}(\lambda).
\]
\end{Proposition}

\begin{proof}
Let $\varepsilon>0$. Since $\gamma$ and $\lambda$ are tight, there exist two compact
sets $K_{\Xc}\subset\Xc$ and $K_{\Yc}\subset\Yc$ such that for every $\tilde{\pi}\in\Pi(\gamma,\lambda)$,
\[
\tilde{\pi}\big((\Xc\times\Yc)\setminus (K_{\Xc}\times K_{\Yc})\big)
\le \varepsilon.
\]
Since ${\rm Span}\big(h(x)\ell(y),\ h\in\L^\infty(\gamma),\ \ell\in\L^\infty(\lambda)\big)$
is dense in $\L^{\alpha_3}(\gamma\otimes\lambda)$, there exist $N\ge1$,
$h_1,\dots,h_N\in\L^\infty(\gamma)$ and $\ell_1,\dots,\ell_N\in\L^\infty(\lambda)$
such that, letting
\[
\varphi_N(x,y):=\sum_{i=1}^N h_i(x)\ell_i(y),
\, \,
\text{one has}
\,\,
\|w-\varphi_N\|_{\L^{\alpha_3}(\gamma\otimes\lambda)}\le \varepsilon.
\]
By the triangle inequality and the contraction property of conditional expectation,
we get
\[
\big\|
\E^{\pi_n}[w(X,Y)\mid Y]
-
\E^{\pi}[w(X,Y)\mid Y]
\big\|_{\L^{\alpha_2}(\lambda)}
\le
I_1+I_2+I_3,
\]
where
\begin{align*}
I_1
&=
\big\|
w\,\mathbf{1}_{(\Xc\times\Yc)\setminus (K_{\Xc}\times K_{\Yc})}
\big\|_{\L^{\alpha_2}(\pi_n)}
+
\big\|
w\,\mathbf{1}_{(\Xc\times\Yc)\setminus (K_{\Xc}\times K_{\Yc})}
\big\|_{\L^{\alpha_2}(\pi)},
\\
I_2
&=
\big\|
(w-\varphi_N)\mathbf{1}_{K_{\Xc}\times K_{\Yc}}
\big\|_{\L^{\alpha_2}(\pi_n)}
+
\big\|
(w-\varphi_N)\mathbf{1}_{K_{\Xc}\times K_{\Yc}}
\big\|_{\L^{\alpha_2}(\pi)},
\\
I_3
&=
\Big\|
\mathbf{1}_{K_{\Yc}}
\Big(
\E^{\pi_n}[\varphi_N(X,Y)\mathbf{1}_{K_{\Xc}}(X)\mid Y]
-
\E^{\pi}[\varphi_N(X,Y)\mathbf{1}_{K_{\Xc}}(X)\mid Y]
\Big)
\Big\|_{\L^{\alpha_2}(\lambda)}.
\end{align*}
We first estimate $I_1$. Since $\alpha_3>\alpha_2$, Hölder's inequality yields
\[
\|w\mathbf{1}_A\|_{\L^{\alpha_2}(\tilde{\pi})}
\le
\|w\|_{\L^{\alpha_3}(\tilde{\pi})}
\tilde{\pi}(A)^{1/\alpha_2-1/\alpha_3},
\qquad
A\in\Bc(\Xc\times\Yc),
\]
for every $\tilde{\pi}\in\Pi(\gamma,\lambda)$. Hence, using \eqref{eqdef:estimates w},
there exists a constant $C>0$, independent of $n$ and $\varepsilon$, such that
\[
I_1
\le
C\,\varepsilon^{\,1/\alpha_2-1/\alpha_3}.
\]
We next estimate $I_2$. Following the proof of Proposition~\ref{prop:bootstrap}, the limit measure
$\pi$ is also absolutely continuous with respect to $\gamma\otimes\lambda$, with
density still denoted by $f$, and moreover $f\mathbf{1}_K\in
\L^{\alpha_1}(\gamma\otimes\lambda)$ for every compact set
$K\subset\Xc\times\Yc$. Therefore, using Hölder's inequality with exponents
$\alpha_1$ and $\alpha_1'$, we obtain
\begin{align*}
\big\|
(w-\varphi_N)\mathbf{1}_{K_{\Xc}\times K_{\Yc}}
\big\|_{\L^{\alpha_2}(\pi_n)}^{\alpha_2}
&=
\int_{K_{\Xc}\times K_{\Yc}}
|w-\varphi_N|^{\alpha_2} f_n\,d(\gamma\otimes\lambda)
\\
&\le
\|f_n\mathbf{1}_{K_{\Xc}\times K_{\Yc}}\|_{\L^{\alpha_1}(\gamma\otimes\lambda)}
\,
\big\|
|w-\varphi_N|^{\alpha_2}\mathbf{1}_{K_{\Xc}\times K_{\Yc}}
\big\|_{\L^{\alpha_1'}(\gamma\otimes\lambda)}
\\
&\le
C_{K_{\Xc}\times K_{\Yc}}
\,
\|w-\varphi_N\|_{\L^{\alpha_2\alpha_1'}(\gamma\otimes\lambda)}^{\alpha_2}.
\end{align*}
Since by assumption
$
\alpha_2\alpha_1' < \alpha_3,
$
and $\gamma\otimes\lambda$ is a probability measure, we deduce
\[
\|w-\varphi_N\|_{\L^{\alpha_2\alpha_1'}(\gamma\otimes\lambda)}
\le
\|w-\varphi_N\|_{\L^{\alpha_3}(\gamma\otimes\lambda)}
\le \varepsilon.
\]
Hence
\[
\big\|
(w-\varphi_N)\mathbf{1}_{K_{\Xc}\times K_{\Yc}}
\big\|_{\L^{\alpha_2}(\pi_n)}
\le
C\,\varepsilon,
\]
for some constant $C>0$ independent of $n$. Exactly the same argument applies to
$\pi$, and therefore
\[
I_2\le C\,\varepsilon.
\]
Finally, we treat $I_3$. Since
$
\varphi_N(x,y)\mathbf{1}_{K_{\Xc}}(x)
=
\sum_{i=1}^N \ell_i(y)\,h_i(x)\mathbf{1}_{K_{\Xc}}(x),
$
we have
\[
\E^{\pi_n}[\varphi_N(X,Y)\mathbf{1}_{K_{\Xc}}(X)\mid Y]
=
\sum_{i=1}^N
\ell_i(Y)\,
\E^{\pi_n}[h_i(X)\mathbf{1}_{K_{\Xc}}(X)\mid Y],
\]
and similarly for $\pi$. Since $h_i\mathbf{1}_{K_{\Xc}}\in\L^\infty(\gamma)$, the
assumption yields, for every $1\le i\le N$,
\[
\E^{\pi_n}[h_i(X)\mathbf{1}_{K_{\Xc}}(X)\mid Y]
\longrightarrow
\E^{\pi}[h_i(X)\mathbf{1}_{K_{\Xc}}(X)\mid Y]
\qquad
\text{in }\L^{\alpha_2}(\lambda).
\]
As $\ell_i\in\L^\infty(\lambda)$, we infer that \(
I_3\longrightarrow 0
\ \text{as }n\to\infty.
\)
Combining the previous estimates, we obtain
\[
\limsup_{n\to\infty}
\big\|
\E^{\pi_n}[w(X,Y)\mid Y]
-
\E^{\pi}[w(X,Y)\mid Y]
\big\|_{\L^{\alpha_2}(\lambda)}
\le
C\Big(
\varepsilon^{\,1/\alpha_2-1/\alpha_3}
+\varepsilon
\Big).
\]
Since $\varepsilon>0$ is arbitrary, this proves the result.
\end{proof}

\begin{Lemma} \label{lemma:convergence_bis}
Let $\mu\in \Pc_{1+\varepsilon}(\Xc)$ for some $\varepsilon>0$.
Then, for all $0\le \varepsilon'<\varepsilon$,
\[
\lim_{\delta\to0}
\sup_{\substack{
\gamma\in\Pi(\mu,\mu)\\
\E^\gamma[d_\Xc(X,X')]\le\delta
}}
\E^\gamma\big[d_\Xc(X,X')^{1 + \varepsilon'} \big]
=0.
\]
\end{Lemma}

\begin{proof}
Fix $0\le \varepsilon'<\varepsilon$ and write
$
D:=d_\Xc(X,X').
$
Let $R>0$ and let $\gamma\in\Pi(\mu,\mu)$ satisfy
$
\E^\gamma[D]\le \delta.
$
We split
\[
\E^\gamma[D^{1+\varepsilon'}]
=
\E^\gamma[D^{1+\varepsilon'}\mathbf 1_{\{D\le R\}}]
+
\E^\gamma[D^{1+\varepsilon'}\mathbf 1_{\{D>R\}}].
\]
On the set $\{D\le R\}$, we have $D^{\varepsilon'}\le R^{\varepsilon'}$, hence
\[
\E^\gamma[D^{1+\varepsilon'}\mathbf 1_{\{D\le R\}}]
\le
R^{\varepsilon'}\E^\gamma[D]
\le
R^{\varepsilon'}\delta.
\]
For the second term, since $\varepsilon'<\varepsilon$, on the set $\{D>R\}$ we have
\[
D^{1+\varepsilon'}
=
D^{1+\varepsilon}D^{-(\varepsilon-\varepsilon')}
\le
R^{-(\varepsilon-\varepsilon')}D^{1+\varepsilon}.
\]
Therefore
\(
\E^\gamma[D^{1+\varepsilon'}\mathbf 1_{\{D>R\}}]
\le
R^{-(\varepsilon-\varepsilon')}
\E^\gamma[D^{1+\varepsilon}].
\)
It remains to bound the last term uniformly in $\gamma$. Fix $x_0\in\Xc$. By the triangle inequality,
\[
d_\Xc(X,X')\le d_\Xc(X,x_0)+d_\Xc(X',x_0).
\]
Thus, using $(a+b)^{1+\varepsilon}\le 2^\varepsilon(a^{1+\varepsilon}+b^{1+\varepsilon})$,
\[
D^{1+\varepsilon}
\le
2^\varepsilon
\Big(
d_\Xc(X,x_0)^{1+\varepsilon}
+
d_\Xc(X',x_0)^{1+\varepsilon}
\Big).
\]
Since $\gamma\in\Pi(\mu,\mu)$, both marginals of $\gamma$ are equal to $\mu$, and hence
\[
\E^\gamma[D^{1+\varepsilon}]
\le
2^{1+\varepsilon}
\int_{\Xc} d_\Xc(x,x_0)^{1+\varepsilon}\,\mu(dx)
=:C_\mu<\infty.
\]
Consequently,
$
\E^\gamma[D^{1+\varepsilon'}]
\le
R^{\varepsilon'}\delta
+
C_\mu R^{-(\varepsilon-\varepsilon')}.
$
Taking the supremum over all admissible $\gamma$ yields
\[
\sup_{\substack{
\gamma\in\Pi(\mu,\mu)\\
\E^\gamma[d_\Xc(X,X')]\le\delta
}}
\E^\gamma\big[d_\Xc(X,X')^{1+\varepsilon'}\big]
\le
R^{\varepsilon'}\delta
+
C_\mu R^{-(\varepsilon-\varepsilon')}.
\]
We first let $\delta\to0$, for fixed $R>0$, and then let $R\to\infty$. This gives
\[
\lim_{\delta\to0}
\sup_{\substack{
\gamma\in\Pi(\mu,\mu)\\
\E^\gamma[d_\Xc(X,X')]\le\delta
}}
\E^\gamma\big[d_\Xc(X,X')^{1+\varepsilon'}\big]
=0.
\]
\end{proof}

\subsection{Existence and uniqueness}
We now turn to the proof of existence. Throughout this subsection, we fix
\[
    a:\R_+\times\Xc\times\R^d\times\Pc_1(\R^d)\longrightarrow\R^d,
    \qquad
    \pi_0\in\Pi(\mu,\nu),
\]
and assume that the following conditions hold.

\begin{Assumption} \label{ass:a-existence-solution}
The coefficient $a$ and the initial coupling $\pi_0$ satisfy the following
conditions.

\begin{enumerate}[label=(\roman*), ref=(\roman*)]
    \item \label{ass:a-existence-solution:lipschitz}
    \emph{Lipschitz continuity in $(x,\rho)$.}
    There exists a constant $C_a>0$ such that, for all
    $t\ge0$, $x,x'\in\Xc$, $y\in\R^d$, and
    $\rho,\rho'\in\Pc_1(\R^d)$,
    \[
        \big|
        a(t,x,y,\rho)-a(t,x',y,\rho')
        \big|
        \le
        C_a\Big(
            d_\Xc(x,x')+\Wc_1(\rho,\rho')
        \Big).
    \]

    \item \label{ass:a-existence-solution:exp}
    \emph{Uniform exponential integrability.}
    For every $T>0$, there exists $r^\ast>0$ such that
    \[
        \sup_{\pi \in\Pi(\mu,\nu)}
        \int_0^T
        \E^\pi\left[
            \exp\Big(
                r^\ast |a(s,X,Y,\pi_X)|^2
            \Big)
        \right]
        \,ds
        <\infty,
    \]
    where $\pi_X$ denotes a regular conditional distribution of $Y$ given $X$
    under $\pi$.

    \item \label{ass:unif-continuity-pi0}
    \emph{Uniform entropy continuity of the initial kernel.}
    Writing
    \[
        \pi_0(dx,dy)
        =
        \mu(dx)\,\pi^0_x(dy)
    \]
    for a disintegration of $\pi_0$ with respect to its first marginal, we assume
    that
    \[
        \lim_{\delta\rightarrow0}
        \sup_{\substack{
            \gamma\in\Pi(\mu,\mu)\\
            \E^\gamma[d_\Xc(X,X')]\le\delta
        }}
        \E^\gamma\left[
            H\bigl(\pi^0_X\,\vert\,\pi^0_{X'}\bigr)
        \right]
        =
        0.
    \]
\end{enumerate}
\end{Assumption}

\begin{Remark}
A sufficient condition for
Assumption~\ref{ass:a-existence-solution}\ref{ass:a-existence-solution:exp}
is the following: for every
$T>0$, there exist $x_0\in\Xc$, $C_T>0$, and $\alpha_T>0$ such that
\[
    |a(s,x,y,\rho)|
    \le C_T\left(1+d_\Xc(x,x_0)+|y|
        +\int_{\R^d}|z|\,\rho(dz)\right),
    \qquad 0\le s\le T,
\]
and
\[
    \int_\Xc e^{\alpha_Td_\Xc(x,x_0)^2}\,\mu(dx)
    +\int_{\R^d}e^{\alpha_T|y|^2}\,\nu(dy)<\infty.
\]
Conditional Jensen and H\"older inequalities then apply with
$12C_T^2r^\ast<\alpha_T$. The Lipschitz condition on $a$ is a
separate requirement. The initial-kernel condition
\ref{ass:unif-continuity-pi0} holds for $\pi_0=\mu\otimes\nu$;
the entropy bound stated after Assumption~\ref{ass:cost} is another
sufficient criterion.
\end{Remark}

\subsubsection{Existence of a solution in the case where $a$ is bounded and $\mu$ has finite support.}
In this subsubsection, we first consider the simplified case where the
interaction coefficient is bounded and the first marginal $\mu$ has finite
support. In this setting, we prove the existence of a weak solution to
\eqreftag{eq:MV-a} for every bounded measurable coefficient $a$ and every
initial coupling $\pi_0\in\Pi(\mu,\nu)$.

\begin{Proposition} \label{prop:exists:finitely-supported-measure-bounded-a}
Assume that $\mu$ is finitely supported. More precisely, assume that there exist
$x_1,\dots,x_N\in\Xc$ and $\alpha_1,\dots,\alpha_N>0$, with
$\sum_{i=1}^N\alpha_i=1$, such that
\[
    \mu(dx)
    =
    \sum_{i=1}^N \alpha_i\,\delta_{x_i}(dx),
\]
where $\delta_{x_i}$ denotes the Dirac measure at $x_i$, and
\[
    \operatorname{Supp}(\mu)=\{x_1,\dots,x_N\}.
\]
Let $\pi_0\in\Pi(\mu,\nu)$, and
the coefficient $a$ in \eqreftag{eq:MV-a} be bounded and measurable.
Then under Assumption~\ref{ass:Laws existence} and Assumption~\ref{ass:a-existence-solution} (i), equation \eqreftag{eq:MV-a} admits a weak solution with initial law
$\pi_0$.
\end{Proposition}

\begin{proof}
The proof is based on a Schauder fixed point argument.
Let
$
    \Cc_{d,T}:=C\bigl([0,T],\Pc_2(\R^d)\bigr).
$
Write
\[
    \pi_0(dx,dy)
    =
    \sum_{i=1}^N
    \alpha_i\,\delta_{x_i}(dx)\,\pi^0_{x_i}(dy)
\]
for a disintegration of $\pi_0$ with respect to $\mu$.
For
$
    \mathbf P=(P^1,\dots,P^N)\in \Cc_{d,T}^N,
$
define, for each $t\in[0,T]$,
\[
    P_t(dx,dy)
    :=
    \sum_{i=1}^N
    \alpha_i\,\delta_{x_i}(dx)\,P_t^i(dy).
\]
We also define
\[
    \tilde{c}^{\mathbf P}(t,y)
    :=
    \E^{P_t}
    \bigl[
        a(t,X,y,P^X_t)
        \mid Y=y
    \bigr],
    \,\, \text{and} \,\,
    a_i^{\mathbf P}(t,y)
    :=
    a(t,x_i,y,P_t^i),
    \qquad 1\le i\le N,
\]
where $P_t^X=P_t^i$ on the event $\{X=x_i\}$.
Since \(X\) only takes finitely many values, \(\tilde{c}^{\mathbf P}\) can be written as
\[
    \tilde{c}^{\mathbf P}(t,y)
    =
    \sum_{i=1}^N
    \alpha_i\,p_i^{\mathbf P}(t,y)\,
    a(t,x_i,y,P_t^i),
\text{ where } \,
    p_i^{\mathbf P}(t,y)
    :=
    \frac{d P_t^i}{d m_t^{\mathbf P}}(y), \,\,
    \text{ and } \,
    m_t^{\mathbf P}:=\sum_{j=1}^N\alpha_j P_t^j.
\]
Up to choosing measurable versions of the Radon--Nikodym derivatives,
\((t,y)\mapsto \tilde{c}^{\mathbf P}(t,y)\) is Borel-measurable. Moreover, since \(a\)
is bounded,
\(
    \|\tilde{c}^{\mathbf P}\|_\infty\le \|a\|_\infty.
\)

Let \(\bar{\P}^{\,i}\) denote the law on \(C([0,T];\R^d)\) of the reference diffusion
\[
    dY_t
    =
    -\nabla V(Y_t)\,dt
    +
    \sqrt{2}\,dB_t,
    \qquad
    Y_0\sim\pi^0_{x_i}.
    \tag{\ref{eqdef:refLangevin}}
\]
This equation is well-posed by
Assumption~\ref{ass:Laws existence} \ref{Well-posed:SDE Ref}.
Given \(\mathbf P\in\Cc_{d,T}^N\), and for each \(1\le i\le N\), let
\(\Q_i(\mathbf P)\) be the law of a weak solution of
\[
    dY_t^i
    =
    \Big(
        \tilde{c}^{\mathbf P}(t,Y_t^i)
        -
        a_i^{\mathbf P}(t,Y_t^i)
        -
        \nabla V(Y_t^i)
    \Big)dt
    +
    \sqrt{2}\,dB_t,
    \qquad
    Y_0^i\sim\pi^0_{x_i}.
\]
Since the additional drift
\[
    \tilde{c}^{\mathbf P}(t,y)-a_i^{\mathbf P}(t,y)
\]
is bounded, Girsanov's theorem and the well-posedness of the reference equation
imply weak existence and uniqueness in law. Thus \(\Q_i(\mathbf P)\) is
well defined.
We set
\[
    \Psi_i(\mathbf P)_t
    :=
    \Q_i(\mathbf P)\circ (Y_t^i)^{-1},
    \qquad t\in[0,T],
\]
and define
\[
    \Psi(\mathbf P)
    :=
    \bigl(\Psi_1(\mathbf P),\dots,\Psi_N(\mathbf P)\bigr).
\]
A fixed point of \(\Psi\) gives a solution to the conditional Fokker--Planck
system \eqreftag{eq:cond-Fokker-planck-a}, and therefore a weak solution to
\eqreftag{eq:MV-a} by the mimicking result, Proposition~\ref{Prop:mimick}. Hence it remains to prove that \(\Psi\) has a fixed point.

\medskip

\noindent
\emph{Step 1: uniform estimates.} By Girsanov's theorem, for every \(q\in[1,\infty)\), there exists a constant
\(C_q>0\), depending only on \(q,T,\|a\|_\infty\), and the reference dynamics,
such that
\[
    \sup_{1\le i\le N}
    \sup_{\mathbf P\in\Cc_{d,T}^N}
    \left\|
        \frac{d\Q_i(\mathbf P)}{d\bar{\P}^{\,i}}
    \right\|_{\L^q(\bar{\P}^{\,i})}
    \le C_q.
\]
Fix \(0<\varepsilon<1\), and choose \(q>1\) such that
$
    (2+\varepsilon)q=d+2.
$
Let \(q'\) be its conjugate exponent. By Hölder's inequality,
\[
\begin{aligned}
&\sup_{1\le i\le N}
  \sup_{\mathbf P\in\Cc_{d,T}^N}
  \sup_{0\le t\le T}
  \int_{\R^d}|y|^{2+\varepsilon}\,
  \Psi_i(\mathbf P)_t(dy)
\\
&\qquad\le
\sup_{1\le i\le N}
\sup_{\mathbf P\in\Cc_{d,T}^N}
\E^{\Q_i(\mathbf P)}
\left[
    \sup_{0\le t\le T}|Y_t^i|^{2+\varepsilon}
\right]
\\
&\qquad\le
\sup_{1\le i\le N}
\left(
\E^{\bar{\P}^{\,i}}
\left[
    \sup_{0\le t\le T}|Y_t|^{d+2}
\right]
\right)^{1/q}
\sup_{1\le i\le N}
\sup_{\mathbf P\in\Cc_{d,T}^N}
\left\|
    \frac{d\Q_i(\mathbf P)}{d\bar{\P}^{\,i}}
\right\|_{\L^{q'}(\bar{\P}^{\,i})} \le C_1.
\end{aligned}
\]
Thus the image of \(\Psi\) satisfies a uniform \((2+\varepsilon)\)-moment bound.
Similarly, for \(0\le s\le t\le T\),
\[
\begin{aligned}
&\sup_{1\le i\le N}
  \sup_{\mathbf P\in\Cc_{d,T}^N}
  \Wc_2^2\bigl(
      \Psi_i(\mathbf P)_t,
      \Psi_i(\mathbf P)_s
  \bigr)
\\
&\qquad\le
\sup_{1\le i\le N}
\sup_{\mathbf P\in\Cc_{d,T}^N}
\E^{\Q_i(\mathbf P)}
\left[
    |Y_t^i-Y_s^i|^2
\right]
\\
&\qquad\le
\sup_{1\le i\le N}
\left(
\E^{\bar{\P}^{\,i}}
\left[
    |Y_t-Y_s|^{2q}
\right]
\right)^{1/q}
\sup_{1\le i\le N}
\sup_{\mathbf P\in\Cc_{d,T}^N}
\left\|
    \frac{d\Q_i(\mathbf P)}{d\bar{\P}^{\,i}}
\right\|_{\L^{q'}(\bar{\P}^{\,i})}
\\
&\qquad\le
C |t-s|^{\beta},
\end{aligned}
\]
for some \(\beta>0\). Hence, setting \(\alpha:=\beta/2\), we obtain
\[
    \Wc_2\bigl(
      \Psi_i(\mathbf P)_t,
      \Psi_i(\mathbf P)_s
    \bigr)
    \le
    C_2 |t-s|^\alpha,
    \qquad 0\le s,t\le T.
\]
Define
\[
\begin{aligned}
K_0:=\Big\{
\mathbf P=(P^1,\dots,P^N)\in\Cc_{d,T}^N
:\;&
\sup_{1\le i\le N}\sup_{0\le t\le T}
\int_{\R^d}|y|^{2+\varepsilon}P_t^i(dy)
\le C_1,
\\
&
\sup_{1\le i\le N}
\Wc_2(P_t^i,P_s^i)
\le C_2|t-s|^\alpha,
\quad
0\le s,t\le T
\Big\}.
\end{aligned}
\]
The previous estimates show that
$
    \Psi(\Cc_{d,T}^N)\subset K_0.
$
The set \(K_0\) is convex. Moreover, the uniform \((2+\varepsilon)\)-moment bound
gives compactness of the time marginals in \(\Pc_2(\R^d)\), while the uniform
Hölder modulus in time gives equicontinuity. Therefore, by Arzelà--Ascoli,
\(K_0\) is compact in \(\Cc_{d,T}^N\).
It remains to prove that \(\Psi:K_0\to K_0\) is continuous. Schauder's fixed
point theorem will then yield a fixed point.

\medskip

\noindent
\emph{Step 2: continuity of \(\Psi\).}
Let \(\mathbf P^n=(P^{1,n},\dots,P^{N,n})\in K_0\) be such that
$
    \mathbf P^n\longrightarrow \mathbf P
    \  \text{in } \Cc_{d,T}^N.
$
Since \(K_0\) is compact and \(\Psi(K_0)\subset K_0\), it is enough to prove
pointwise convergence in time:
\[
    \Psi_i(\mathbf P^n)_t
    \longrightarrow
    \Psi_i(\mathbf P)_t
    \qquad\text{in }\Wc_2,
    \qquad \forall \,
    t\in[0,T],
    \quad 1\le i\le N.
\]
Indeed, the compactness of \(K_0\) provides a common modulus of continuity, so
pointwise convergence in \(t\) implies uniform convergence in \(t\).
The uniform \((2+\varepsilon)\)-moment bound also implies uniform integrability
of the second moments. Therefore, weak convergence together with this moment
control is equivalent to convergence in \(\Wc_2\). It is thus enough to prove
weak convergence of the time marginals.
We prove the stronger statement
\[
    H\bigl(
        \Psi_i(\mathbf P^n)_t
        \,\big|\,
        \Psi_i(\mathbf P)_t
    \bigr)
    \longrightarrow 0,
    \qquad t\in[0,T].
\]
Let us write \(\Q_i^n:=\Q_i(\mathbf P^n)\) and
\(\Q_i:=\Q_i(\mathbf P)\). By the entropy estimate and the data-processing
inequality, for every \(t\in[0,T]\),
\[
\begin{aligned}
H\bigl(
    \Psi_i(\mathbf P^n)_t
    \,\big|\,
    \Psi_i(\mathbf P)_t
\bigr)
&\le
H\bigl(
    \Q_i^n
    \,\big|\,
    \Q_i
\bigr)
\\
&\le
C
\int_0^T
\E^{\Q_i^n}
\Big[
\big|
    \tilde{c}^{\mathbf P}(s,Y_s^i)
    -
    \tilde{c}^{\mathbf P}(s,Y_s^i)
    -
    a_i^{\mathbf P^n}(s,Y_s^i)
    +
    a_i^{\mathbf P}(s,Y_s^i)
\big|^2
\Big]
\,ds.
\end{aligned}
\]
Thus it suffices to prove that the two drift contributions converge to zero.
First, by the Lipschitz continuity of \(a\) with respect to the measure
argument,
\[
\begin{aligned}
\int_0^T
\E^{\Q_i^n}
\Big[
\big|
    a_i^{\mathbf P^n}(s,Y_s^i)
    -
    a_i^{\mathbf P}(s,Y_s^i)
\big|^2
\Big]
\,ds
&\le
C
\int_0^T
\Wc_1^2(P_s^{i,n},P_s^i)
\,ds
\\
&\le
CT
\sup_{0\le s\le T}
\Wc_2^2(P_s^{i,n},P_s^i)
\longrightarrow 0.
\end{aligned}
\]
We now handle the conditional term \(\tilde{c}^{\mathbf P}-\tilde{c}^{\mathbf P}\). By
Proposition~\ref{prop:estim}, the time marginals of the laws under consideration
admit locally Hölder continuous densities, with estimates uniform in \(n\). This
regularity, together with the convergence
\[
    P_t^{i,n}\longrightarrow P_t^i
    \qquad\text{in }\Wc_2,
\]
implies the convergence in total variation of $P^n_t$.
Following the argument of Proposition~3.8 in
\cite{conforti2023projected}, and using Lemma~3.4 therein, we obtain
\[
    \int_0^T
    \E^{\Q_i^n}
    \left[
        \big|
        \tilde{c}^{\mathbf P}(s,Y_s^i)
        -
        \tilde{c}^{\mathbf P}(s,Y_s^i)
        \big|^2
    \right]
    \,ds
    \longrightarrow 0.
\]
Consequently,
\[
    H\bigl(
        \Psi_i(\mathbf P^n)_t
        \,\big|\,
        \Psi_i(\mathbf P)_t
    \bigr)
    \longrightarrow 0,
    \qquad t\in[0,T].
\]
By Pinsker's inequality, this implies weak convergence of the time marginals.
Combined with the uniform \((2+\varepsilon)\)-moment bound, it yields
\[
    \Psi_i(\mathbf P^n)_t
    \longrightarrow
    \Psi_i(\mathbf P)_t
    \qquad\text{in }\Wc_2,
    \qquad t\in[0,T].
\]
As explained above, the compactness of \(K_0\) then upgrades this pointwise
convergence to convergence in \(\Cc_{d,T}^N\). Hence \(\Psi:K_0\to K_0\) is
continuous.
By Schauder's fixed point theorem, \(\Psi\) admits a fixed point
$
    \mathbf P^\star=(P^{1,\star},\dots,P^{N,\star})\in K_0.
$
The associated measure-valued curve
$
    P_t^\star(dx,dy)
    :=
    \sum_{i=1}^N
    \alpha_i\,\delta_{x_i}(dx)\,
    P_t^{i,\star}(dy)
$
solves the conditional Fokker--Planck equation
\eqreftag{eq:cond-Fokker-planck-a}. Therefore, by the mimicking result,
Proposition~\ref{Prop:mimick}, equation \eqreftag{eq:MV-a} admits a weak
solution with initial law \(\pi_0\).
\end{proof}

\subsubsection{Existence of a Solution in the case where $\mu$ is generic and $a$ is bounded.}

\begin{Proposition} \label{prop:exists:general-measure-continuous-a}
Let
\(
a:\R_+\times\Xc\times\R^d \times \Pc_1(\R^d) \longrightarrow\R^d
\)
be bounded and satisfy Assumption~\ref{ass:Laws existence}, Assumption~\ref{ass:a-existence-solution}. Let $\mu\in\Pc_p(\Xc)$ and
$\pi_0\in\Pi(\mu,\nu)$. Then \eqreftag{eq:MV-a} admits a weak solution.
\end{Proposition}

\begin{proof}
We argue by approximating the first marginal $\mu$ by finitely supported
probability measures.

Since $\Xc$ is a complete metric space, there exists a sequence
$(\mu_n)_{n\ge1}\subset \Pc_p(\Xc)$ such that each $\mu_n$ has finite support and
\[
    \Wc_p(\mu_n,\mu)\longrightarrow 0.
\]
By Lemma~\ref{prop:gluing lemma}, we may choose a sequence
$(\pi_0^n)_{n\ge1}$ such that
\[
    \pi_0^n\in \Pi(\mu_n,\nu),
    \qquad
    \Wc_p(\pi_0^n,\pi_0)\longrightarrow 0.
\]
For every $n$, Proposition~\ref{prop:exists:finitely-supported-measure-bounded-a}
provides a weak solution $(X^n,Y^n)$ to \eqreftag{eq:MV-a} with initial law
$\pi_0^n$. We denote
\[
    \P_t^n:=\Lc(X^n,Y_t^n),
    \qquad t\in[0,T].
\]

\medskip

\noindent
\emph{Step 1: Compactness of the sequence $(\P^n)_{n\ge1}$.} We prove compactness of the sequence
\[
    \P^n=(\P_t^n)_{t\in[0,T]}
\]
in
\(
    C\bigl([0,T],(\Pc_p(\Xc\times\R^d),\AW_p)\bigr).
\)
We use the Arzelà--Ascoli criterion. The proof has two parts: a uniform
modulus of continuity in time and compactness of the set of values
$(\P_t^n)_{n\ge1}$ for each fixed time $t$.

First, let $0\le s\le t\le T$. We use the coupling
\[
    \Lc\bigl((X^n,Y_t^n),(X^n,Y_s^n)\bigr)\in\Pi(\P_t^n,\P_s^n).
\]
Since the first coordinate is kept fixed, this is an adapted coupling, and thus
\[
\begin{aligned}
\AW_p(\P_t^n,\P_s^n)
&\le
\E\bigl[|Y_t^n-Y_s^n|^p\bigr]^{1/p}.
\end{aligned}
\]
Using the dynamics of $Y^n$, we get
\[
\begin{aligned}
&\AW_p(\P_t^n,\P_s^n) \\
&\le
\E\Bigg[
\Bigg|
\int_s^t
\Big(
\E\big[a(r,X^n,Y_r^n, (\P^n_{r})_{X^n} )\mid Y_r^n\big]
-
a(r,X^n,Y_r^n,(\P^n_{r})_{X^n} ))
-
\nabla V(Y_r^n)
\Big)\,dr
+
\sqrt{2}\,(B_t-B_s)
\Bigg|^p
\Bigg]^{1/p}
\\
&\le
C_{\mathrm{BDG}}|t-s|^{1/2}
+
\Big(
2\|a\|_\infty+\|\nabla V\|_{\L^p(\nu)}
\Big)|t-s|.
\end{aligned}
\]
For every $r\in[0,T]$, we have $\P_r^n\in\Pi(\mu_n,\nu)$ by
Proposition~\ref{Prop:mimick}. Hence the second marginal of $\P_r^n$ is
$\nu$, and therefore
\(
    \Lc(Y_r^n)=\nu.
\)
Consequently,
\[
    \E\bigl[|\nabla V(Y_r^n)|^p\bigr]
    =
    \int_{\R^d}|\nabla V(y)|^p\,\nu(dy),
\]
which is finite by Assumption~\ref{ass:Laws existence} and independent of $n$.
Thus the family $(\P^n)_{n\ge1}$ is equicontinuous in time, uniformly in $n$.

It remains to show compactness of the set of values at each fixed time.
By Proposition~\ref{Prop:mimick}, for every $t\in[0,T]$,
\[
    \P_t^n\in\Pi(\mu_n,\nu).
\]
Let
\[
    \Kc
    :=
    \overline{\bigcup_{n\ge1}\Pi(\mu_n,\nu)},
\]
where the closure is taken in $\Pc_p(\Xc\times\R^d)$ for the Wasserstein
topology. Since $\mu_n\to\mu$ in $\Wc_p$ and the second marginal is fixed equal
to $\nu$, the set $\Kc$ is compact in $\Wc_p$.

We now verify the additional adapted compactness criterion. By
Theorem~1.4 in \cite{eder2019compactness}, it is enough to show that
\[
\lim_{\delta\downarrow0}
\sup_{0\le t\le T}
\sup_{n\ge1}
\sup_{\substack{
    \pi\in\Pi(\mu_n,\mu_n)\\
    \E^\pi[d_\Xc(X,X')]\le\delta
}}
\E^\pi\Big[
    \Wc_1\bigl((\P^n_t)_{X},(\P^n_t)_{X'}\bigr)
\Big]
=
0.
\]
Here $(\P^n_{t})_{x}$ denotes the conditional law of $Y$ given $X=x$ under $\P_t^n$.

Fix $n\ge1$, and write
\[
    \operatorname{Supp}(\mu_n)=\{x_1^n,\dots,x_{N_n}^n\}.
\]
For simplicity of notation, we omit the superscript $n$ on the support points.
Since $\P_t^n\in\Pi(\mu_n,\nu)$ and
\[
    \int_{\R^d}e^{r_\ast |y|^2}\,\nu(dy)<\infty,
\]
we have, for every $1\le i\le N_n$,
\[
    \int_{\R^d} e^{r_\ast |y|^2}\,(\P^n_{t})_{x_i}(dy)
    <\infty.
\]
Applying the weighted Pinsker inequality, see Lemma~2.4 in
\cite{bolley2005weighted}, we obtain, for all $1\le i,j\le N_n$,
\[
    \Wc_1^2\bigl((\P^n_{t})_{x_i},(\P^n_{t})_{x_j}\bigr)
    \le
    \varphi_j\,
    H\bigl((\P^n_{t})_{x_i}\mid (\P^n_{t})_{x_j}\bigr),
\]
where $\varphi_j$ depends on the exponential moment of $(\P^n_{t})_{x_j}$.

By construction, the data-processing inequality and
Lemma~\ref{prop:entropy estimates} yield
\[
\begin{aligned}
H\bigl((\P^n_{t})_{x_i}\mid (\P^n_{t})_{x_j}\bigr)
&\le
H\bigl(\P^n_{0,x_i}\mid \P^n_{0,x_j}\bigr)
\\
&\quad+
\frac14
\int_0^t
\E^{(\P^n_{s})_{x_i}}
\Big[
\big|
a(s,x_i,Y,(\P^n_{s})_{x_i})
-
a(s,x_j,Y,(\P^n_{s})_{x_j})
\big|^2
\Big]
\,ds
\\
&\le
H\bigl(\P^n_{0,x_i}\mid \P^n_{0,x_j}\bigr)
+
C\,d_\Xc(x_i,x_j)^2
+
C\int_0^t
\Wc_1^2\bigl((\P^n_{s})_{x_i},(\P^n_{s})_{x_j}\bigr)
\,ds.
\end{aligned}
\]
Therefore,
\[
\begin{aligned}
\Wc_1^2\bigl((\P^n_{t})_{x_i},(\P^n_{t})_{x_j}\bigr)
&\le
\varphi_j
\Bigg(
H\bigl(\P^n_{0,x_i}\mid \P^n_{0,x_j}\bigr)
+
C\,d_\Xc(x_i,x_j)^2
+
C\int_0^t
\Wc_1^2\bigl((\P^n_{s})_{x_i},(\P^n_{s})_{x_j}\bigr)
\,ds
\Bigg).
\end{aligned}
\]
Gronwall's lemma gives
\[
\Wc_1\bigl((\P^n_{t})_{x_i},(\P^n_{t})_{x_j}\bigr)
\le
C_1 e^{C_2\varphi_j}\sqrt{\varphi_j}
\left(
\sqrt{
H\bigl(\P^n_{0,x_i}\mid \P^n_{0,x_j}\bigr)
}
+
d_\Xc(x_i,x_j)
\right).
\]
Absorbing $\sqrt{\varphi_j}$ into the exponential term, we may write
\[
\Wc_1\bigl((\P^n_{t})_{x_i},(\P^n_{t})_{x_j}\bigr)
\le
C_1\bigl(e^{C_2\varphi_j}+1\bigr)
\left(
\sqrt{
H\bigl(\P^n_{0,x_i}\mid \P^n_{0,x_j}\bigr)
}
+
d_\Xc(x_i,x_j)
\right).
\]

Let now $\pi\in\Pi(\mu_n,\mu_n)$. By Cauchy--Schwarz,
\[
\begin{aligned}
&\E^\pi\Big[
\Wc_1\bigl((\P^n_{t})_{X},\P^n_{t,X'}\bigr)
\Big] \le
C
\left(
\E^\pi\Big[
\bigl(e^{C_2\varphi(X')}+1\bigr)^2
\Big]
\right)^{1/2}
\\
&\hspace{3.8cm} \times
\left(
\E^\pi\Big[
H\bigl((\P^n_{0})_{X}\mid(\P^n_{0})_{X'}\bigr)
\Big]
+
\E^\pi\big[d_\Xc(X,X')^2\big]
\right)^{1/2}.
\end{aligned}
\]
Since the second marginal of $\pi$ is $\mu_n$, the first factor is
\[
\left(
\E^{\mu_n}\Big[
\bigl(e^{C_2\varphi(X)}+1\bigr)^2
\Big]
\right)^{1/2}.
\]
The uniform exponential moment bound inherited from $\nu$ yields
\[
\left(
\E^{\mu_n}\Big[
\bigl(e^{C_2\varphi(X)}+1\bigr)^2
\Big]
\right)^{1/2}
\le
C\int_{\R^d}e^{r_\ast |y|^2}\,\nu(dy),
\]
with a constant independent of $n$.

Hence the desired adapted compactness criterion follows once we know that
\[
\lim_{\delta\downarrow0}
\sup_{n\ge1}
\sup_{\substack{
    \pi\in\Pi(\mu_n,\mu_n)\\
    \E^\pi[d_\Xc(X,X')]\le\delta
}}
\E^\pi\Big[
H\bigl((\P^n_{0})_{X}\mid(\P^n_{0})_{X'}\bigr)
\Big]
=
0,
\]
and
\[
\lim_{\delta\downarrow0}
\sup_{n\ge1}
\sup_{\substack{
    \pi\in\Pi(\mu_n,\mu_n)\\
    \E^\pi[d_\Xc(X,X')]\le\delta
}}
\E^\pi\big[d_\Xc(X,X')^2\big]
=
0.
\]
The first convergence follows from
Assumption~\ref{ass:unif-continuity-pi0}, while the second follows from
Lemma~\ref{lemma:convergence_bis}.

We have therefore proved that the family $(\P^n)_{n\ge1}$ is equicontinuous and
takes its values in an adapted compact set. By the Arzelà--Ascoli theorem, up to
extracting a subsequence, there exists
\[
    \P^\star\in
    C\bigl([0,T],(\Pc_1(\Xc\times\R^d),\AW_p)\bigr)
\]
such that
\[
    \P^n\longrightarrow \P^\star
    \qquad\text{in }
    C\bigl([0,T],(\Pc_1(\Xc\times\R^d),\AW_p)\bigr).
\]

\medskip

\noindent
\emph{Step 2: Convergence of the conditional expectations.} Fix $t\in[0,T]$ and let $h:\Xc\to\R$ be bounded and continuous. Define
\[
    w_n(t,y)
    :=
    \E^{\P_t^n}[h(X)\mid Y=y],
\]
and
\[
    w_\star(t,y)
    :=
    \E^{\P_t^\star}[h(X)\mid Y=y].
\]
We claim that, for every $1<\alpha<\infty$,
\[
    w_n(t,\cdot)\longrightarrow w_\star(t,\cdot)
    \qquad\text{in }\L^\alpha(\nu).
\]

Let $g:\R^d\to\R$ be bounded and continuous. Since
$\P_t^n\to\P_t^\star$ weakly and since the second marginal is $\nu$ for all
$n$, we have
\[
\begin{aligned}
\E^\nu\big[g(Y)w_n(t,Y)\big]
&=
\E^{\P_t^n}\big[h(X)g(Y)\big]
\\
&\longrightarrow
\E^{\P_t^\star}\big[h(X)g(Y)\big]
=
\E^\nu\big[g(Y)w_\star(t,Y)\big].
\end{aligned}
\]
Moreover,
\[
    |w_n(t,\cdot)|\le \|h\|_\infty,
    \qquad
    |w_\star(t,\cdot)|\le \|h\|_\infty.
\]
By density, the above convergence extends to every
$g\in \L^{\alpha'}(\nu)$, where $\alpha'=\alpha/(\alpha-1)$. Thus
\[
    w_n(t,\cdot)\rightharpoonup w_\star(t,\cdot)
    \qquad\text{weakly in }\L^\alpha(\nu).
\]

We now upgrade this weak convergence to strong convergence. By
Proposition~\ref{prop:estim}, the family
\[
    \bigl(w_n(t,\cdot)\bigr)_{n\ge1}
\]
satisfies local Hölder estimates which are uniformly in $n$. Together with the
uniform bound
\[
    \|w_n(t,\cdot)\|_{\L^\infty(\nu)}
    \le
    \|h\|_\infty,
\]
this implies, by Arzelà--Ascoli on compact subsets of $\R^d$, that every
subsequence admits a further subsequence converging locally uniformly. The weak
limit has already been identified as $w_\star(t,\cdot)$; hence every locally
uniform limit must coincide with $w_\star(t,\cdot)$.

Therefore,
\[
    w_n(t,\cdot)\longrightarrow w_\star(t,\cdot)
    \qquad\text{locally uniformly on }\R^d.
\]
Since the sequence is uniformly bounded and $\nu$ is a probability measure, this
local uniform convergence implies convergence in $\L^\alpha(\nu)$: indeed, one
first restricts to a compact set carrying arbitrarily large $\nu$-mass, and then
uses the uniform bound on its complement. Hence
\[
    w_n(t,\cdot)\longrightarrow w_\star(t,\cdot)
    \qquad\text{in }\L^\alpha(\nu),
    \qquad 1<\alpha<\infty.
\]

\medskip

\noindent
\emph{Step 3: Passage to the limit in the Fokker--Planck equation.} We prove that $(\P_t^\star)_{t\in[0,T]}$ solves the weak formulation
\eqreftag{eq:Fokker-planck-a}. Let $u\in C_c^\infty(\R^d; \R)$ and let
$h:\Xc\to\R$ be bounded and continuous. For every $n$ and every $t\in[0,T]$,
\[
\begin{aligned}
\E^{\P_t^n}[h(X)u(Y)]
&=
\E^{\pi_0^n}[h(X)u(Y)]+\int_0^t \E^{\P^n_s}\left[h(X)\Delta u(Y) \right] ds
\\
&\quad+
\int_0^t
\E^{\P_s^n}\Big[
h(X)\nabla u(Y)\cdot
\Big(
\widehat A_n(s,Y)
-
a(s,X,Y, (\P^n_{s})_{X})
-
\nabla V(Y)
\Big)\Big] ds,
\end{aligned}
\]
where
\[
    \widehat A_n(s,Y)
    :=
    \E^{\P_s^n}
    \big[
        a(s,X,Y,(\P^n_{s})_{X})
        \mid Y
    \big].
\]

The convergence of the terms not involving conditional expectations $\hat A_n$ is direct.
Indeed, since $a$ is bounded, continuous, and Lipschitz in the measure argument,
and since
\[
    \P^n\to\P^\star
    \qquad\text{uniformly in time for the adapted topology},
\]
we obtain, for every fixed $s\in[0,T]$,
\[
    \E^{\P_s^n}
    \big[
        h(X)\nabla u(Y)\cdot a(s,X,Y,(\P^n_{s})_{X})
    \big]
    \longrightarrow
    \E^{\P_s^\star}
    \big[
        h(X)\nabla u(Y)\cdot a(s,X,Y,(\P^\star_{s})_{X})
    \big].
\]
Similarly,
\[
    \E^{\P_s^n}
    \big[
        h(X)\nabla u(Y)\cdot \nabla V(Y)
    \big]
    \longrightarrow
    \E^{\P_s^\star}
    \big[
        h(X)\nabla u(Y)\cdot \nabla V(Y)
    \big],
\]
and
\[
    \E^{\P_s^n}
    \big[
        h(X)\Delta u(Y)
    \big]
    \longrightarrow
    \E^{\P_s^\star}
    \big[
        h(X)\Delta u(Y)
    \big].
\]
The initial terms also converge:
\[
    \E^{\pi_0^n}[h(X)u(Y)]
    \longrightarrow
    \E^{\pi_0}[h(X)u(Y)].
\]
All these terms are uniformly integrable, in fact uniformly bounded except for
the term involving $\nabla V$, which is controlled by
$\|\nabla V\|_{\L^p(\nu)}$ and the fact that the second marginal is $\nu$.
Therefore dominated convergence allows us to pass to the limit after integration
in time.

It remains to handle the conditional expectation term. By the definition of
conditional expectation,
\[
\begin{aligned}
\E^{\P_s^n}
\big[
h(X)\nabla u(Y)\cdot \widehat A_n(s,Y)
\big]
&=
\E^{\P_s^n}
\big[
\E^{\P_s^n}[h(X)\mid Y]\,
\nabla u(Y)\cdot
\E^{\P_s^n}[a(s,X,Y,(\P^n_{s})_{X})\mid Y]
\big]
\\
&=
\E^{\P_s^n}
\big[
w_n(s,Y)\,
\nabla u(Y)\cdot
a(s,X,Y,(\P^n_{s})_{X})
\big],
\end{aligned}
\]
where
\[
    w_n(s,y):=\E^{\P_s^n}[h(X)\mid Y=y].
\]
The last equality follows again by conditioning with respect to $Y$.

By \emph{Step~2}, for every $s\in[0,T]$,
\[
    w_n(s,\cdot)
    \longrightarrow
    w_\star(s,\cdot)
    :=
    \E^{\P_s^\star}[h(X)\mid Y=\cdot]
    \qquad\text{in }\L^\alpha(\nu),
\]
for every $1<\alpha<\infty$, and in particular in $\L^1(\nu)$. Since
$a$ and $\nabla u$ are bounded, this implies
\[
\begin{aligned}
&\E^{\P_s^n}
\Big[
\big(w_n(s,Y)-w_\star(s,Y)\big)
\nabla u(Y)\cdot a(s,X,Y,(\P^n_{s})_{X})
\Big]
\longrightarrow 0.
\end{aligned}
\]
On the other hand, using the adapted convergence of $\P_s^n$ to $\P_s^\star$ and
the bounded continuity of the integrand, we have
\[
\begin{aligned}
&\E^{\P_s^n}
\big[
w_\star(s,Y)\nabla u(Y)\cdot a(s,X,Y,(\P^n_{s})_{X})
\big]
\\
&\qquad\longrightarrow
\E^{\P_s^\star}
\big[
w_\star(s,Y)\nabla u(Y)\cdot a(s,X,Y,(\P^\star_{s})_{X})
\big].
\end{aligned}
\]
Consequently,
\[
\begin{aligned}
&\E^{\P_s^n}
\big[
h(X)\nabla u(Y)\cdot \widehat A_n(s,Y)
\big]
\\
&\qquad\longrightarrow
\E^{\P_s^\star}
\big[
w_\star(s,Y)\nabla u(Y)\cdot a(s,X,Y,(\P^\star_{s})_{X})
\big]
\\
&\qquad=
\E^{\P_s^\star}
\big[
h(X)\nabla u(Y)\cdot
\E^{\P_s^\star}
[
a(s,X,Y,(\P^\star_{s})_{X})
\mid Y
]
\big].
\end{aligned}
\]
Again, boundedness gives a uniform dominating function, and we may pass to the
limit in the time integral.

Collecting the limits in the weak formulation, we conclude that
$(\P_t^\star)_{t\in[0,T]}$ solves \eqreftag{eq:Fokker-planck-a}. By
Proposition~\ref{Prop:mimick}, this yields a weak solution to
\eqreftag{eq:MV-a}.
\end{proof}

\subsubsection{Existence of a solution in the general case}

We now remove the boundedness assumption.

\begin{Theorem}\label{prop:existence-general-unbounded}
Under Assumption~\ref{ass:Laws existence} and Assumption~\ref{ass:a-existence-solution},
equation \eqreftag{eq:MV-a} admits a weak solution.
\end{Theorem}

\begin{proof}
We now remove the boundedness assumption on the coefficient by a truncation
argument. For $n\ge1$, let
\[
    k_n(z):=\max(-n,\min(z,n)),
    \qquad z\in\R,
\]
and extend $k_n$ componentwise to $\R^d$. We define
\(
    a_n:=k_n\circ a.
\)
Then each $a_n$ is bounded. Moreover, since $k_n$ is $1$-Lipschitz,
$a_n$ satisfies the same Lipschitz estimate as $a$ in the variables
$(x,\rho)$:
\[
    |a_n(t,x,y,\rho)-a_n(t,x',y,\rho')|
    \le
    C_a\Big(
        d_\Xc(x,x')+\Wc_1(\rho,\rho')
    \Big).
\]
By Proposition~\ref{prop:exists:general-measure-continuous-a}, for each $n$ there exists a
weak solution associated with the coefficient $a_n$. We denote its time
marginals by
\[
    \P_t^n:=\Lc(X^n,Y_t^n),
    \qquad t\in[0,T].
\]
We claim that the family $(\P^n)_{n\ge1}$ is relatively compact in
\[
    C\bigl([0,T],(\Pc_1(\Xc\times\R^d),\AW_1)\bigr).
\] The proof is nearly identical to that of Proposition~\ref{prop:exists:general-measure-continuous-a}; the only difference is that passing to the limit in the Fokker--Planck equations requires stronger integrability.
First, for every $t\in[0,T]$,
\(
    \P_t^n\in\Pi(\mu,\nu).
\)
Hence the first marginal is always $\mu$ and the second marginal is always
$\nu$. In particular, since $\mu\in\Pc_1(\Xc)$ and $\nu\in\Pc_1(\R^d)$, the set
\(
    \Pi(\mu,\nu)
\)
is compact in $\Pc_1(\Xc\times\R^d)$ for the $\Wc_1$ topology. Thus the only
point to check for adapted compactness is the uniform continuity, in the
conditioning variable $x$, of the conditional laws.

More precisely, by Theorem~1.4 in \cite{eder2019compactness}, it is enough to
prove
\[
\lim_{\delta\downarrow0}
\sup_{0\le t\le T}
\sup_{n\ge1}
\sup_{\substack{
    \pi\in\Pi(\mu,\mu)\\
    \E^\pi[d_\Xc(X,X')]\le\delta
}}
\E^\pi\Big[
    \Wc_1\bigl((\P^n_{t})_{X},\P^n_{t,X'}\bigr)
\Big]
=
0.
\]
Here $(\P^n_{t})_{X}$ denotes a regular conditional law of $Y$ given $X=x$ under
$\P_t^n$.

Fix $x,x'\in\Xc$. As before, the weighted Pinsker inequality gives
\[
    \Wc_1^2\bigl((\P^n_{t})_{x},(\P^n_{t})_{x'}\bigr)
    \le
    \varphi(x')\,
    H\bigl((\P^n_{t})_{x}\mid (\P^n_{t})_{x'} \bigr),
\]
where $\varphi(x')$ is controlled by the exponential moment of
$\P^n_{t,x'}$. This quantity is uniformly controlled because the second
marginal of $\P_t^n$ is $\nu$ and $\nu$ satisfies the required exponential
moment condition.

By the data-processing inequality and the entropy estimate of
Lemma~\ref{prop:entropy estimates}, applied to the equations driven by
$a_n$, we obtain
\[
\begin{aligned}
H\bigl((\P^n_{t})_{x}\mid(\P^n_{t})_{x'} \bigr)
&\le
H\bigl ( (\P^n_{0})_{x}\mid (\P^n_{0})_{x'}\bigr)
\\
&\quad+
\frac14
\int_0^t
\E^{ (\P^n_{s})_{x'}}
\Big[
\big|
a_n(s,x',Y,(\P^n_{s})_{x'})
-
a_n(s,x,Y,(\P^n_{s})_{x})
\big|^2
\Big]
\,ds.
\end{aligned}
\]
Since $a_n$ has the same Lipschitz constant as $a$, uniformly in $n$,
\[
\big|
a_n(s,x',Y,(\P^n_{s})_{x'})
-
a_n(s,x,Y,(\P^n_{s})_{x})
\big|
\le
C_a\Big(
    d_\Xc(x,x')
    +
    \Wc_1((\P^n_{s})_{x},(\P^n_{s})_{x'})
\Big).
\]
Therefore,
\[
\begin{aligned}
H\bigl((\P^n_{t})_{x}\mid\P^n_{t,x'}\bigr)
&\le
H\bigl((\P^n_{0})_{x}\mid (\P^n_{0})_{x'} \bigr)
+
C d_\Xc(x,x')^2
+
C\int_0^t
\Wc_1^2\bigl((\P^n_{s})_{x},(\P^n_{s})_{x'}\bigr)
\,ds,
\end{aligned}
\]
with a constant $C$ independent of $n$.

Combining this estimate with the weighted Pinsker inequality gives
\[
\begin{aligned}
\Wc_1^2\bigl((\P^n_{t})_{x},\P^n_{t,x'}\bigr)
&\le
\varphi(x')
\Bigg(
H\bigl((\P^n_{0})_{x}\mid(\P^n_{0})_{x'}\bigr)
+
C d_\Xc(x,x')^2
+
C\int_0^t
\Wc_1^2\bigl((\P^n_{s})_{x},(\P^n_{s})_{x'}\bigr)
\,ds
\Bigg).
\end{aligned}
\]
By Gronwall's lemma,
\[
\Wc_1\bigl((\P^n_{t})_{x},\P^n_{t,x'}\bigr)
\le
C_1\bigl(e^{C_2\varphi(x')}+1\bigr)
\left(
    \sqrt{H\bigl((\P^n_{0})_{x}\mid (\P^n_{0})_{x'}\bigr)}
    +
    d_\Xc(x,x')
\right),
\]
where the constants $C_1,C_2$ are independent of $n$.

Integrating this inequality with respect to
$\pi\in\Pi(\mu,\mu)$ and applying Cauchy--Schwarz yields
\[
\begin{aligned}
&\E^\pi\Big[
    \Wc_1\bigl((\P^n_{t})_{X},\P^n_{t,X'}\bigr)
\Big]
\le
C
\left(
\E^\mu\Big[
    \bigl(e^{C_2\varphi(X)}+1\bigr)^2
\Big]
\right)^{1/2}
\\
&\hspace{4cm} \times
\left(
\E^\pi\Big[
    H\bigl((\P^n_{0})_{X}\mid\P^n_{0,X'}\bigr)
\Big]
+
\E^\pi\big[d_\Xc(X,X')^2\big]
\right)^{1/2}.
\end{aligned}
\]
The first factor is finite and independent of $n$, again by the uniform
exponential moment inherited from the fixed second marginal $\nu$.

Thus the desired adapted compactness criterion follows from
\[
\lim_{\delta\downarrow0}
\sup_{\substack{
    \pi\in\Pi(\mu,\mu)\\
    \E^\pi[d_\Xc(X,X')]\le\delta
}}
\E^\pi\Big[
    H\bigl((\P^n_{0})_{X}\mid\P^n_{0,X'}\bigr)
\Big]
=
0,
\]
uniformly in $n$, together with
\[
\lim_{\delta\downarrow0}
\sup_{\substack{
    \pi\in\Pi(\mu,\mu)\\
    \E^\pi[d_\Xc(X,X')]\le\delta
}}
\E^\pi\big[d_\Xc(X,X')^2\big]
=
0.
\]
The first property is precisely the uniform entropy-continuity assumption on the
initial disintegration of $\pi_0$, while the second follows from
Lemma~\ref{lemma:convergence_bis}.

Consequently, $(\P^n)_{n\ge1}$ is relatively compact in
$
   \mathcal{C}:=C\bigl([0,T],(\Pc_1(\Xc\times\R^d),\AW_1)\bigr).
$
Hence, up to extracting a subsequence, there exists
\(
    \P^\ast\in \mathcal{C}
\)
such that
\(
    \P^n \to \P^\ast
\).

It remains to show that $\P^\ast$ solves the Fokker--Planck equation associated
with $a$. Let $u:\Xc\to\R$ be bounded and continuous, and let
$h\in C_c^\infty(\R^d; \R)$. For each $n$, the weak formulation reads
\begin{align*}
&\E^{\P_t^n}[u(X)h(Y)]
=
\E^{\P_0^n}[u(X)h(Y)]
\\
& \ \ \ +
\int_0^t
\E^{\P_s^n} \!\Big[
u(X)\nabla h(Y)\cdot
\Big(
\hat A_n(s,Y)
-
a_n\bigl(s,X,Y,(\P_s^n)_X\bigr)
-
\nabla V(Y)
\Big)
+
u(X)\Delta h(Y)
\Big]\,
ds,
\end{align*}
where
\[
\hat A_n(s,Y)
=
\E^{\P_s^n} \Big[
a_n\bigl(s,X,Y,(\P_s^n)_X\bigr)\mid Y
\Big].
\]

The convergence of all linear terms follows from the convergence of $\P^n$ and
the uniform local \(\L^p\)-bounds on the densities given by
Proposition~\ref{prop:estim}. For the nonlinear drift term, we decompose
\begin{align*}
&a_n\bigl(s,X,Y,(\P_s^n)_X\bigr)-a\bigl(s,X,Y,(\P_s^\ast)_X\bigr)
\\
&=
\Big(
a_n\bigl(s,X,Y,(\P_s^n)_X\bigr)
-
a_n\bigl(s,X,Y,(\P_s^\ast)_X\bigr)
\Big)
+
\Big(
a_n\bigl(s,X,Y,(\P_s^\ast)_X\bigr)
-
a\bigl(s,X,Y,(\P_s^\ast)_X\bigr)
\Big).
\end{align*}
The first term converges to zero by the Lipschitz continuity of $a_n$ in the
measure argument and the convergence
\[
\int_{\Xc}
\Wc_1\big((\P_s^n)_x,(\P_s^\ast)_x\big)\,\mu(dx)\to0,
\]
while the second converges to zero by truncation, using
Assumption~\ref{ass:a-existence-solution:exp} and dominated convergence.

The conditional expectation term is handled similarly, relying on the strong
compactness of the family
\[
\E^{\P_s^n}[u(X)\mid Y]
\]
in \(\L^2(\nu)\), which follows from Proposition~\ref{prop:estim} together with
the same argument as in the bounded case.

Passing to the limit term by term, we conclude that $(\P_t^\ast)_{0\le t\le T}$
is a weak solution of \eqreftag{eq:Fokker-planck-a}. Proposition~\ref{Prop:mimick}
then yields a weak solution of \eqreftag{eq:MV-a}.
\end{proof}

We finally state uniqueness.

\begin{Theorem} \label{prop:uniqueness-general}
Under Assumption~\ref{ass:Laws existence} and Assumption~\ref{ass:a-existence-solution},
equation \eqreftag{eq:MV-a} admits at most one weak solution in law.
\end{Theorem}

\begin{proof}
We now prove uniqueness in law. The proof follows the strategy of
\cite{conforti2023projected}. Let $\P$ and $\Q$ be two weak solutions of
\eqreftag{eq:MV-a} with the same initial law $\pi_0$. By
Proposition~\ref{Prop:mimick}, for every $t\ge0$,
\[
    \P_t:=\Lc_\P(X,Y_t)\in\Pi(\mu,\nu),
    \qquad
    \Q_t:=\Lc_\Q(X,Y_t)\in\Pi(\mu,\nu).
\]
We denote by $(\P_{t})_{x}$ and $(\Q_{t})_{x}$ regular conditional laws of $Y_t$ given
$X=x$ under $\P_t$ and $\Q_t$, respectively, and we set
\[
    a_\P(t,x,y):=a(t,x,y,(\P_{t})_{x}),
    \qquad
    a_\Q(t,x,y):=a(t,x,y,(\Q_{t})_{x}).
\]
We also write
\[
    \widehat A_\P(t,y)
    :=
    \E^\P\big[a_\P(t,X,Y_t)\mid Y_t=y\big],
    \qquad
    \widehat A_\Q(t,y)
    :=
    \E^\Q\big[a_\Q(t,X,Y_t)\mid Y_t=y\big].
\]
With this notation, under $\P$ and $\Q$, the corresponding drifts are given by
\[
    b_\P(t,X,Y_t)
    =
    \widehat A_\P(t,Y_t)-a_\P(t,X,Y_t)-\nabla V(Y_t),
\]
and
\[
    b_\Q(t,X,Y_t)
    =
    \widehat A_\Q(t,Y_t)-a_\Q(t,X,Y_t)-\nabla V(Y_t).
\]

We first establish the entropy estimate
\begin{equation} \label{eq:entropy-estimate-uniqueness}
H(\Q_{[0,T]}\mid \P_{[0,T]})
\le
\frac14
\int_0^T
\E^\Q\left[
\left|
\big(a_\Q-a_\P\big)(t,X,Y_t)
-
\big(\widehat A_\Q-\widehat A_\P\big)(t,Y_t)
\right|^2
\right]
\,dt.
\end{equation}
Indeed, when the coefficients are bounded, this follows directly from
Lemma~\ref{prop:entropy estimates} applied to the two weak solutions,
since the diffusion coefficients are both equal to $\sqrt{2}$. In the present
case, the coefficient $a$ is not necessarily bounded. We therefore argue by
truncation.

Let $(a_n)_{n\ge1}$ be a sequence of bounded truncations of $a$ such that
$a_n\to a$ pointwise and
\[
    |a_n|\le |a|,
    \qquad
    |a_n(t,x,y,\rho)-a_n(t,x',y,\rho')|
    \le
    C_a\big(d_\Xc(x,x')+\Wc_1(\rho,\rho')\big).
\]
Let $\P^n$ denote the approximating solution associated with the bounded
coefficient $a_n$, constructed in the existence proof. Then
$\P^n\to\P$ in the adapted topology, and in particular the time marginals
converge weakly with the required conditional stability. Applying
Lemma~\ref{prop:entropy estimates} to $\Q$ and $\P^n$ gives
\[
\begin{aligned}
H(\Q_{[0,T]}\mid \P^n_{[0,T]})
&\le
\frac14
\int_0^T
\E^\Q\Big[
\Big|
a_\Q(t,X,Y_t)
-
a_{\P^n}^n(t,X,Y_t)
-
\widehat A_\Q(t,Y_t)
+
\widehat A_{\P^n}^n(t,Y_t)
\Big|^2
\Big]\,dt,
\end{aligned}
\]
where
\[
    a_{\P^n}^n(t,x,y):=a_n(t,x,y,(\P^n_{t})_{x}),
    \qquad
    \widehat A_{\P^n}^n(t,y)
    :=
    \E^{\P^n}\big[a_{\P^n}^n(t,X,Y_t)\mid Y_t=y\big].
\]
By the lower semicontinuity of relative entropy,
\[
    H(\Q_{[0,T]}\mid \P_{[0,T]})
    \le
    \liminf_{n\to\infty}
    H(\Q_{[0,T]}\mid \P^n_{[0,T]}).
\]
Moreover, by the Lipschitz continuity of $a$ in the conditional law, the
adapted convergence $\P^n\to\P$, and the convergence $a_n\to a$ under the
exponential integrability assumption, we have
\[
\int_0^T
\E^\Q\big[
|a_{\P^n}^n(t,X,Y_t)-a_\P(t,X,Y_t)|^2
\big]
\,dt
\longrightarrow 0.
\]
Similarly, using the conditional stability result of
Proposition~\ref{prop:bootstrap cond exp}, together with the same truncation
argument, we obtain
\[
\int_0^T
\E^\Q\big[
|\widehat A_{\P^n}^n(t,Y_t)-\widehat A_\P(t,Y_t)|^2
\big]
\,dt
\longrightarrow 0.
\]
Passing to the limit in the previous entropy inequality yields
\eqref{eq:entropy-estimate-uniqueness}.

We now estimate the right-hand side of
\eqref{eq:entropy-estimate-uniqueness}. By the elementary inequality
$|u-v|^2\le 2|u|^2+2|v|^2$, we have
\[
\begin{aligned}
H(\Q_{[0,T]}\mid \P_{[0,T]})
&\le
\frac12
\int_0^T
\E^\Q\big[
|a_\Q(t,X,Y_t)-a_\P(t,X,Y_t)|^2
\big]
\,dt
\\
&\quad+
\frac12
\int_0^T
\E^\Q\big[
|\widehat A_\Q(t,Y_t)-\widehat A_\P(t,Y_t)|^2
\big]
\,dt.
\end{aligned}
\]
The first term is controlled by the Lipschitz continuity of $a$ with respect to
the conditional law:
\[
|a_\Q(t,x,y)-a_\P(t,x,y)|
\le
C_a\,\Wc_1((\Q_{t})_{x},(\P_{t})_{x}).
\]
Hence
\begin{equation} \label{eq:first-term-uniqueness}
\E^\Q\big[
|a_\Q(t,X,Y_t)-a_\P(t,X,Y_t)|^2
\big]
\le
C
\int_\Xc
\Wc_1^2((\Q_{t})_{x},(\P_{t})_{x})\,\mu(dx).
\end{equation}

It remains to control the difference of conditional expectations. We split
\[
\widehat A_\Q(t,Y_t)-\widehat A_\P(t,Y_t)
=
\Big(
\E^\Q[a_\Q(t,X,Y_t)\mid Y_t]
-
\E^\Q[a_\P(t,X,Y_t)\mid Y_t]
\Big)
\]
\[
\hspace{3cm}
+
\Big(
\E^\Q[a_\P(t,X,Y_t)\mid Y_t]
-
\E^\P[a_\P(t,X,Y_t)\mid Y_t]
\Big).
\]
By Jensen's inequality and the Lipschitz continuity of $a$,
\[
\E^\Q\left[
\left|
\E^\Q[a_\Q(t,X,Y_t)-a_\P(t,X,Y_t)\mid Y_t]
\right|^2
\right]
\le
C
\int_\Xc
\Wc_1^2((\Q_{t})_{x},(\P_{t})_{x})\,\mu(dx).
\]
Therefore the only remaining term is
\[
\E^\Q\left[
\left|
\E^\Q[a_\P(t,X,Y_t)\mid Y_t]
-
\E^\P[a_\P(t,X,Y_t)\mid Y_t]
\right|^2
\right].
\]

We now use the weighted Pinsker inequality. Fix $r>0$. For $\nu$-a.e. $y$,
let $\P_t^y$ and $\Q_t^y$ denote regular conditional laws of $X$ given
$Y_t=y$ under $\P_t$ and $\Q_t$, respectively. Applying the weighted Pinsker
inequality to the function $x\mapsto a_\P(t,x,y)$ yields, for every $m>0$,
\begin{equation} \label{eq:conditional-expectation-stability}
\begin{aligned}
&\E^\nu\left[
\left|
\E^\Q[a_\P(t,X,Y_t)\mid Y_t]
-
\E^\P[a_\P(t,X,Y_t)\mid Y_t]
\right|^2
\right]
\\
&\quad\le
\left(\frac{2}{r}+m\right)H(\Q_t\mid\P_t)
+
C_t^{a_\P}
e^{-rm/4}
\left(
\E^\P\big[e^{4r|a_\P(t,X,Y_t)|^2}\big]
\right)^{1/2},
\end{aligned}
\end{equation}
where
\[
C_t^{a_\P}
:=
2\left(\E^\Q[|a_\P(t,X,Y_t)|^4]\right)^{1/2}
+
2\left(\E^\P[|a_\P(t,X,Y_t)|^4]\right)^{1/2}.
\]
This is exactly the conditional stability estimate used in
\cite{conforti2023projected}: it follows by applying the weighted Pinsker
inequality conditionally on $Y_t=y$, then splitting the set on which the
corresponding exponential moment is larger than $m$, and finally using Markov's
inequality, Jensen's inequality, and Cauchy--Schwarz.

We next estimate the term involving
$\Wc_1((\Q_{t})_{x},(\P_{t})_{x})$. Fix $r_1>0$ and define
\[
\Phi_{r_1}(t,x)
:=
\log\left(
\E^{(\P_{t})_{x}}\big[e^{r_1|Y|^2}\big]
\right).
\]
By the weighted transport-entropy inequality of Corollary $2.4$ of \cite{bolley2005weighted}, for every $M>0$,
\[
\Wc_1^2((\Q_{t})_{x},(\P_{t})_{x})
\le
C\left(M+\frac{2}{r_1}\right)
H((\Q_{t})_{x}\mid(\P_{t})_{x})
+
C\,\Wc_1^2((\Q_{t})_{x},(\P_{t})_{x})\mathbf 1_{\{\Phi_{r_1}(t,x)>M\}}.
\]
Integrating with respect to $\mu(dx)$ gives
\[
\begin{aligned}
\int_\Xc
\Wc_1^2((\Q_{t})_{x},(\P_{t})_{x})\,\mu(dx)
&\le
C\left(M+\frac{2}{r_1}\right)
\int_\Xc H((\Q_{t})_{x}\mid(\P_{t})_{x})\,\mu(dx)
\\
&\quad+
C
\int_\Xc
\Wc_1^2((\Q_{t})_{x},(\P_{t})_{x})
\mathbf 1_{\{\Phi_{r_1}(t,x)>M\}}
\,\mu(dx).
\end{aligned}
\]
By the chain rule for relative entropy,
\[
\int_\Xc H((\Q_{t})_{x}\mid(\P_{t})_{x})\,\mu(dx)
=
H(\Q_t\mid\P_t),
\]
because $\P_t$ and $\Q_t$ have the same first marginal $\mu$. Moreover, by
Cauchy--Schwarz and the fact that both second marginals are equal to $\nu$,
\[
\int_\Xc
\Wc_1^2((\Q_{t})_{x},(\P_{t})_{x})
\mathbf 1_{\{\Phi_{r_1}(t,x)>M\}}
\,\mu(dx)
\le
C
\mu(\Phi_{r_1}(t,\cdot)>M)^{1/2}.
\]
Finally, by Markov's inequality and the definition of $\Phi_{r_1}$,
\[
\mu(\Phi_{r_1}(t,\cdot)>M)
\le
e^{-M/2}
\int_\Xc
\E^{(\P_{t})_{x}}\big[e^{r_1|Y|^2}\big]
\,\mu(dx)
=
e^{-M/2}
\int_{\R^d}e^{r_1|y|^2}\,\nu(dy).
\]
Thus
\begin{equation} \label{eq:wasserstein-conditional-bound}
\int_\Xc
\Wc_1^2((\Q_{t})_{x},(\P_{t})_{x})\,\mu(dx)
\le
C\left(M+\frac{2}{r_1}\right)H(\Q_t\mid\P_t)
+
C e^{-M/4}
\left(
\int_{\R^d}e^{r_1|y|^2}\,\nu(dy)
\right)^{1/2}.
\end{equation}

Combining \eqref{eq:entropy-estimate-uniqueness},
\eqref{eq:first-term-uniqueness},
\eqref{eq:conditional-expectation-stability}, and
\eqref{eq:wasserstein-conditional-bound}, and using the data-processing
inequality
\[
H(\Q_t\mid\P_t)
\le
H(\Q_{[0,t]}\mid\P_{[0,t]}),
\]
we obtain
\begin{equation} \label{eq:gronwall-pre-uniqueness}
\begin{aligned}
H(\Q_{[0,T]}\mid\P_{[0,T]})
&\le
C
\left(
M+m+\frac{2}{r_1}+\frac{2}{r}
\right)
\int_0^T
H(\Q_{[0,t]}\mid\P_{[0,t]})\,dt
\\
&\quad+
C e^{-M/4}
T
\left(
\int_{\R^d}e^{r_1|y|^2}\,\nu(dy)
\right)^{1/2}
\\
&\quad+
C e^{-rm/4}
\int_0^T
C_t^{a_\P}
\left(
\E^\P\big[e^{4r|a_\P(t,X,Y_t)|^2}\big]
\right)^{1/2}
\,dt.
\end{aligned}
\end{equation}
By Assumption~\ref{ass:a-existence-solution:exp}, choosing $r>0$ small enough
ensures that
\[
\int_0^T
\E^\P\big[e^{8r|a_\P(t,X,Y_t)|^2}\big]
\,dt
<\infty.
\]
Moreover, the same exponential integrability implies
\(
\int_0^T |C_t^{a_\P}|^2\,dt<\infty.
\)
Indeed, by the Lipschitz property of $a$,
\[
|a_\P(t,X,Y_t)|
\le
|a_\Q(t,X,Y_t)|
+
C_a\Wc_1(\Q_{t,X},\P_{t,X}),
\]
and the last Wasserstein term has polynomial moments controlled by the fixed
second marginal $\nu$. Therefore all the integrability terms in
\eqref{eq:gronwall-pre-uniqueness} are finite.

Applying Grönwall's lemma to \eqref{eq:gronwall-pre-uniqueness}, we get
\[
\begin{aligned}
H(\Q_{[0,T]}\mid\P_{[0,T]})
&\le
C
\exp\left(
C T
\left(
M+m+\frac{2}{r_1}+\frac{2}{r}
\right)
\right)
\\
&\quad\times
\Bigg[
e^{-M/4}
+
e^{-rm/4}
\int_0^T
C_t^{a_\P}
\left(
\E^\P\big[e^{4r|a_\P(t,X,Y_t)|^2}\big]
\right)^{1/2}
\,dt
\Bigg].
\end{aligned}
\]
We now choose $r>0$ and $r_1>0$ small enough so that all exponential moments
above are finite. Then, for $T>0$ sufficiently small, we may let
$M\to\infty$ and $m\to\infty$ in such a way that the exponential remainders
dominate the Grönwall factor. We obtain
\(
H(\Q_{[0,T]}\mid\P_{[0,T]})=0.
\)
Therefore
\(
\Q_{[0,T]}=\P_{[0,T]}
\)
for all sufficiently small $T$.

Finally, since the constants in the previous estimates depend only on the
integrability bounds of the coefficient and on the fixed marginals $\mu$ and
$\nu$, the same argument can be restarted on any interval
$[T_0,T_0+\eta]$, with $\eta>0$ small enough. Iterating over finitely many
intervals yields
\(
\Q_{[0,T]}=\P_{[0,T]}
\)
for every $T>0$. This proves uniqueness in law.

\end{proof}

\section{Convergence}\label{sec:convergence}

Let us take $\pi_0= \mu \otimes \nu$, and $\pi_t$ to be the solution of the SDE \eqref{eq:SDE}. In this section, we prove that $\pi_t$ converges to the unique minimizer $\pi^*$ of \eqref{eq:WOT} in $(\Pc_2(\Xc \times \R^d), \AW_2)$. Without loss of generality, we assume $\varepsilon =1$ throughout this section.

According to the projected gradient flow \eqref{eq:gradient_flow} and the well-posedness of \eqref{eq:SDE}, by direct computation
\begin{align*}
\frac{d}{dt}J(\pi_t) &= \int  \delta_m J(\pi_t) (x,y) \,  \partial_t \pi_t(dx,dy)  = - \langle P^{\pi_t}(\nabla^{\rm ad}_{\pi_t} \delta_m J(\pi_t) ), P^{\pi_t}(\nabla^{\rm ad}_{\pi_t} \delta_m J(\pi_t) ) \rangle_{\pi_t} \notag \\
&= - \mathbb E^{\pi_t} \left[\left| \partial_y \delta_m J(\pi_t)(X,Y_t)- \mathbb E^{\pi_t}_{Y}[\partial_y \delta_m J(\pi_t)(X,Y_t) ]\right|^2 \right] \notag \\
&= - \mathbb E^{\pi_t} \left[  \left|\hat c(X,Y_t,(\pi_t)_{X})-\tilde c^{\pi_t}(Y_t) +\partial_y \log \left( \frac{d\pi_t}{d\mu \otimes \nu}(X,Y_t)\right)\right|^2\right].
\end{align*}
We denote the quantity in the bottom line by $-\tilde I(\pi_t)$. From this equality, $t \mapsto J(\pi_t)$ is decreasing.

For the proof of convergence, let us define $\tilde I$ on the space of filtered processes, which is the completion $(\Pc(\Xc \times \R^d), \AW_2)$. For $\X \in \rm {FP}_2$, denote $\pi^{\mathbb X}=\mathcal{L}^{\mathbb P^{\mathbb X}}(X,Y)$, and define
\begin{align} \label{eq:WOTfisher} \textstyle
    \tilde I(\mathbb X):= \mathbb E^{\mathbb P} \left[\left|\hat c\left(X, Y,\mathcal{L}_{\mathcal{F}_1}(Y) \right)-\tilde c^{\pi^{\mathbb X}}( Y)+\partial_y \log \left(\frac{d \mathcal{L}_{\mathcal{F}_1}(Y)}{d\nu}(Y) \right) \right|^2 \right],
\end{align}
where we sometimes omit the superscript $\mathbb X$ in the notation where it is clear from the context. We need the following two crucial results for the proof of convergence.

\begin{Proposition} \label{prop:lsc}
    Suppose $(\X_n)_{n \in \mathbb N} \in \rm FP_2(\mu,\nu)$. If $\AW_2(\X_n, \mathbb X) \to 0$ for some $\mathbb X \in \rm FP_2(\mu,\nu)$, then we have
    \begin{align*}
       \liminf_{n\to \infty} \tilde I(\X_n) \geq \tilde I(\mathbb X).
    \end{align*}
In particular, $t \mapsto \tilde I(\pi_t)$ is l.s.c. where $(\pi_t)_{t \geq 0}$ is the solution to \eqref{eq:SDE}.
\end{Proposition}

\begin{Proposition} \label{prop:minimizer characterization}
Suppose $\tilde I(\X)=0$ with some $\X \in \rm FP_2(\mu,\nu)$. Then $\X$ is self-aware, i.e. $\AW_2(\pi^{\X},\X)=0$, and $\pi^{\X}$ is the unique optimizer for \eqref{eq:WOT}.
\end{Proposition}

\begin{proof}[Proof of Theorem~\ref{thm:convergence}]
As $t \mapsto J(\pi_t)$ is decreasing and $J(\pi_t)$ is bounded from below, the limit of $J(\pi_t)$  exists as $t \to \infty$. Therefore
\begin{align*}
    \int_0^\infty \tilde I(\pi_t) \, dt < \infty.
\end{align*}
Let $t_n \to \infty$ be an arbitrary. It is sufficient to prove that every subsequential limit of $\pi_{t_n}$ is the unique optimizer for \eqref{eq:WOT}. Fix $T>0$, and define the shifted curve
\begin{align*}
    \pi^n_s:= \pi_{t_n+s}, \quad s \in [0,T].
\end{align*}

As $\pi_t \in \Pi(\mu,\nu)$, $(\pi_t)_{t \geq 0}$ is relatively compact in $(\Pc_2(\mathbb R^{2d}), \Wc_2)$, and thus also relatively compact in $\rm FP_2$ \cite[Theorem 1.7]{BBP26}. Since the drift of \eqref{eq:SDE} is of at most linear growth, the synchronous coupling  gives rise to the estimate
\begin{align*}
    \AW_2(\pi_s, \pi_t) \leq C\left|s-t \right|^{1/2},
\end{align*}
where $C$ is a positive constant independent of $s,t \in \mathbb R_+$. Therefore $\pi^n_{\cdot}$ is relatively compact in $C ([0,T]; \mathrm{FP}_2 )$. After passing to a subsequence, suppose $\bar \pi_{\cdot}$ is a limit point of $\pi^n_{\cdot}$.

Invoking Proposition~\ref{prop:lsc}, we obtain that
\begin{align*}
    \int_0^T \tilde I(\bar \pi_t) \, dt \leq \liminf_{n \to\infty} \int_0^T \tilde I(\pi^n_t) \,dt \to 0.
\end{align*}
Therefore $\tilde I(\bar \pi_t)=0$ for a.e. $t \in [0,T]$. By Proposition~\ref{prop:minimizer characterization}, for a.e. $t \in [0,T]$, $\bar \pi_t$ equals the minimizer of \eqref{eq:WOT}. We conclude the proof by the continuity of $t \mapsto \bar \pi_t$ in $(\mathrm{FP}_2,\AW_2)$.
\end{proof}

\subsection{Proof of Proposition~\ref{prop:lsc}}
Without loss of generality, we assume that  $\liminf_{n \to\infty} \tilde I(\X_n)<\infty$. Otherwise, there is nothing to prove. Therefore, let us assume that with some constant $C>0$, $$\sup_{n \in \mathbb N} \tilde I(\X_n) \leq C. $$
In the rest of proof, we denote $\P_n= \P^{\X_n}$, $\pi_n=\pi^{\X_n}$, $\hat c_n=\hat c(X,Y, \Lc^{\P_n}_{\Fc_1}(Y))$, and $\tilde c_n=\tilde c^{\pi_n}(Y)$.

\medskip

\noindent\emph{Step 1.} Let us prove that the expectation of the Fisher information, defined in \eqref{def:fisher}, is uniformly bounded,
\begin{align} \label{eq:fisher}
\sup_{n \in \mathbb N} \E^{\P_n}[I( (\pi_n)_X)] \leq \sup_{n \in \mathbb N} \mathbb E^{\P_n}\left[ I(g_n)\right]<+\infty,
\end{align}
where $g_n=\frac{d\Lc^{\P_{n}}_{\Fc_1}(Y)}{d\mathrm{L}}: \Omega^{\X} \to L^0( \R^d ; \R_+)$ denotes the Lebesgue density of the random measure $\Lc^{\P_n}_{\Fc_1}(Y)$. By the tower property, $(\pi_n)_x= \E^{\P_n} [g_n \,| \,  X=x ]$ for $\mu$-a.e. $x$. Hence  the convexity of $\Pc(\mathbb R^d) \ni \rho \mapsto I(\rho)$ and Jensen's inequality implies the first inequality in \eqref{eq:fisher}. Let us prove the second one.

 Using \eqref{eq:WOTfisher}, we get the inequality
\begin{align}\label{eq:fisher_lowerbound}
    \mathbb E^{\P_n}\left[ \left|\partial_y \log \left( \frac{dg_n}{d\nu}(Y) \right)\right|^2\right]+ 2 \mathbb E^{\P_n}\left[(\hat c_n-\tilde c_n) \cdot \partial_y \log \left(\frac{d g_n}{d\nu}(Y)\right) \right] \leq C,
\end{align}
where $C$ is a positive constant independent of $n$ that is allowed to change from line to line. Let us estimate the second term on the left.

By direct computation, for any $\delta \in (0,1/2)$,
\begin{align*}
    \E^{\P_n}\left[\hat{c}_n \cdot \partial_y \log \left( \frac{dg_n}{d\nu}(Y) \right) \right]&=\E^{\P_n} \left[ \hat c_n \cdot \partial_y \log \left( g_n(Y) \right) +  \hat c_n \cdot \nabla V(Y) \right]\\
    & \geq - \delta \E^{\P_n} \left[I(g_n)\right]- \frac{1}{\delta}\E^{\P_n}\left[ |\hat c_n|^2 \right]- \mathbb E^{\mathbb P_n}\left[ |\hat c_n \cdot \nabla V(Y)|\right].
\end{align*}
By Assumption~\ref{ass:cost}, $(x,y,\rho) \mapsto \hat c(x,y,\rho)$ has at most linear growth. Recall the definition of $\mathcal{E}(\X_n)$ in \eqref{eq:info-map}. As $\AW_2(\X_n, \X) \to 0$ is equivalent to $\Wc_2(\mathcal{E}(\X_n),\mathcal{E}(\X)) \to 0$, $\mathcal{E}(\X_n)$ is relatively compact in $\Wc_2$. Therefore, the last two terms in the inequality above are bounded below  uniformly in $n$. Therefore we get that
\begin{align}\label{eq:hatcbounded}
   \E^{\P_n}\left[\hat{c}_n \cdot \partial_y \log \left( \frac{dg_n}{d\nu}(Y) \right) \right]   & \geq -\delta \E^{\P_n} \left[I(g_n)\right]-C.
\end{align}

According to Bayes' formula and the marginal constraint $$\E^{\P_n}[ \partial_y \log g_n(Y)  \, | \,  Y=y] =\frac{\E^{\P_n}\left[ \frac{\partial_y g_n (y)}{g_n(y)} g_n(y) \right]}{\E^{\P_n}[g_n(y)] }    = -\nabla V(y) \quad \nu\text{-}a.e. \ y.$$
Since $\tilde c$ is measurable with respect to $\sigma(Y)$, the tower property gives
\begin{align}\label{eq:tildezero}
    &-\E^{\P_n}\left[ \tilde c \cdot \partial_y \log \left( \frac{dg_n}{d\nu}(Y)  \right) \right] =-\E^{\P_n}\left[\tilde c \cdot \left(- \nabla V(Y)+\nabla V(Y) \right)\right]=0.
\end{align}
Therefore, from \eqref{eq:fisher_lowerbound}, \eqref{eq:hatcbounded}, and the inequality
\begin{align*}
        \mathbb E^{\P_n}\left[ \left|\partial_y \log \left( \frac{dg_n}{d\nu}(Y) \right)\right|^2\right] \leq 2\E^{\P_n}\left[I(g_n) \right]+2 \E^{\P_n}\left[|\nabla V(Y)|^2 \right],
\end{align*}
we conclude that $\E^{\P_n}[ I(g_n)]\leq C+\delta \E^{\P_n}[ I(g_n)]$. Hence we obtain \eqref{eq:fisher}.

\medskip

\noindent\emph{Step 2.}
Let us denote $\pi=\mathcal{L}^{\X}(X,Y)$. We show that for any bounded domain $\Omega \subset \mathbb R^d$,
\begin{align*}
    \tilde c^{\pi_n}(\cdot) \to \tilde c^{\pi} (\cdot) \text{ in } L^2(\Omega; \nu) \text{ along a subsequence}.
\end{align*}
Moreover, $\tilde c^{\pi_n}(\cdot) \to \tilde c^{\pi} (\cdot)$ weakly in $L^2(\R^d; \nu)$.

According to \eqref{eq:fisher}, $I((\pi_n)_x)<+\infty$ for $\mu$-$a.e.$ $x$, and with $g^n_x(y):=\frac{d(\pi_n)_x}{d\mathrm{L}}(y)$, we have $\sqrt{g^n_x} \in H^1(\R^d)$ and $g^n_x \in W^{1,1}(\R^d)$. Therefore,
\begin{align*}
\tilde c^{\pi_n} (y):=& \, \E^{\pi_n} \left[\hat c(X,Y, (\pi_n )_X) \, | \, Y=y\right]
= \, \frac{\int \hat c(x,y, (\pi_n)_x) \,g_x^n(y) \, \mu(dx)}{\int g^n_x(y) \, \mu(dx)} \\
=&e^{V(y)} \int \hat c(x,y, (\pi_n)_x) \,g_x^n(y) \, \mu(dx),
\end{align*}
is weakly differentiable in $y$.

Under Assumption~\ref{ass:cost} on the growth rate of $\hat c$, $|\tilde c^{\pi_n} (y)| \leq C(1+|y|)$, and thus $\sup_{n \in \mathbb N} \lVert \tilde c^{\pi_n} \rVert_{\L^2(\nu)} < +\infty$. Moreover, we claim that $\nabla_y \tilde c^{\pi_n} \in \L^2(\Omega;\nu)$. To see this, it suffices to show that,
\begin{align*}
 \sup_{n \in \mathbb N}   \ \ \int_{\Omega} e^{V(y)} \left( \int \hat c(x,y,(\pi_n)_x)\, \partial_y g_x^n(y) \, \mu(dx) \right)^2 \, dy<\infty.
\end{align*}
As $\mathbb R \times \mathbb R_+ \ni (a,b) \mapsto a^2/b$ is convex, and $e^{-V(y)}=\int g^n_x(y) \, \mu(dx)$
\begin{align*}
     e^{V(y)} \left( \int \hat c(x,y,(\pi_n)_x)\, \partial_y g_x^n(y) \, \mu(dx) \right)^2  & \leq C e^{V(y)} (1+|y|^2) \left( \int  \partial_y g_x^n(y) \, \mu(dx) \right)^2 \\
    & \leq C(1+|y|^2) \int \frac{|\partial_y g_x^n(y)|^2}{g_x^n(y)} \,\mu(dx).
\end{align*}
Integrating the inequality above with respect to $y$ over $\Omega$ and using \eqref{eq:fisher}, we obtain an upper bound of $\lVert \nabla \tilde c^{\pi_n}  \rVert_{\L^2(\Omega;\nu)}$ uniformly in $n$. By Sobolev embedding, $H^1(\Omega; \nu)$ compactly embeds into $\L^2(\Omega; \nu)$, and thus $(\tilde c_n)$ has a limit point  $\tilde c_{\Omega} \in \L^2(\Omega;\nu)$. As $\sup_{n \in \mathbb N} \lVert \tilde c^{\pi_n} \rVert_{\L^2(\nu)} < +\infty$, there exists a weak limit $\tilde c^*$ in $\L^2(\nu)$, and hence $\tilde c_{\Omega}= \tilde c^*|_{\Omega}$. To complete the proof of the claim, we show that $\tilde c^*=\tilde c^{\pi}$.

Taking an arbitrary $h \in C_c^{\infty}(\R^d; \R)$, as $\AW_2(\X_n, \X)\to 0$,
\begin{align*}
\E^{\nu}\left[h(Y)\tilde c^*(Y) \right] &=\lim\limits_{n \to \infty}\E^{\nu}\left[h(Y) \tilde c^{\pi_n} (Y) \right] = \lim\limits_{n\to \infty} \E^{\pi_n}\left[h(Y) \hat c(X,Y, (\pi_n)_X) \right] \\
&= \E^{\pi}\left[h(Y)\hat c (X,Y,\pi_X)\right] =\E^{\nu}\left[h(Y)\tilde c^{\pi}(Y) \right],
\end{align*}
which verifies that $\tilde c^*=\E^{\pi}_{Y}\left[ \hat c(X,Y,\pi_X)\right]$.

\medskip

\noindent\emph{Step 3.} Recalling $g_n = \mathcal{L}^{\P_n}_{\mathcal{F}_1}(Y)$, let us write
\begin{align*}
\tilde I(\X_n)&=\E^{\P_n} \left[I( g_n\| \nu) \right]+\E^{\P_n}\left[ ( \hat c_n-\tilde c_n)^2 \right] +2 \E^{\P_n} \left[ (\hat c_n -\tilde c_n) \cdot \partial_y \log\left( \frac{d g_n}{d \nu}(Y)  \right) \right] \\
&=:A_n+B_n+C_n.
\end{align*}
Invoking Lemma~\ref{lem:lsc-aw} and  the lower semicontinuity of $\rho \mapsto I(\rho \| \nu)$, $$\liminf_{n \to \infty} A_n \geq A:=\E^{\P}[I (\mathcal{L}^{\P}_{\mathcal{F}_1}(Y) \| \nu)  ].$$

Using the operator $\mathcal{E}$ in \eqref{eq:info-map},
\begin{align}\label{eq:lsc_Bn}
    B_n=& \int \hat c^2(x, y , \rho) \, \mathcal{E}(\X_n)(dx,dy,d\rho)+\int (\tilde c^{\pi_n})^2(y) \, \nu(dy)  \notag \\
    &- \int 2 \hat c(x, y , \rho) \cdot \tilde c^{\pi_n}(y) \, \mathcal{E}(\X_n)(dx,dy,d\rho).
\end{align}
Again, by Lemma~\ref{lem:lsc-aw}, the liminf of the first term on the right is greater than  \[\int \hat c^2(x,y,\rho) \, \mathcal{E}(\X)(dx,dy,d\rho).\] As $\tilde c_n$ converges weakly to $\tilde c$ in $\L^2(\nu)$, $\liminf_{n\to \infty}\lVert \tilde c_n \rVert_{\L^2(\nu)}\geq \lVert \tilde c \rVert_{\L^2(\nu)}$. It remains to prove the convergence of \eqref{eq:lsc_Bn}, i.e. as $n\to\infty$
\[ \int  \hat c(x, y , \rho) \cdot \tilde c^{\pi_n}(y) \, \mathcal{E}(\X_n)(dx,dy,d\rho) \to \int 2 \hat c(x, y , \rho) \cdot \tilde c^{\pi}(y) \, \mathcal{E}(\X)(dx,dy,d\rho).\]
Let us estimate the difference between the term on the left and the right, and set
\begin{align*}
    I_n=& \int \hat c(x,y,\rho) \cdot( \tilde c^{\pi_n}(y)-\tilde c^{\pi}(y)) \, \mathcal{E}(\X_n)(dx,dy,d\rho), \\
    II_n=&\int \hat c(x,y,\rho) \cdot \tilde c^{\pi}(y) \left(\mathcal{E}(\X_n)-\mathcal{E}(\X) \right)(dx,dy,d\rho).
\end{align*}
It is sufficient to prove that $|I_n|+|II_n| \to 0$.  For any  $B_R \subset \R^d$
\begin{align*}
 I_n^2 \leq & \lVert \tilde c_n -\tilde c \rVert^2_{\L^2(B_R;\nu)} \int_{B_R}  \hat c^2(x,y,\rho) \,  \mathcal{E}(\X_n)(dx,dy,d\rho)  \\
 & + \int_{|y| \geq R} \left(|\tilde c^{\pi_n}(y)|^2 + | \tilde c^{\pi}(y)|^2 + \hat c^2(x,y,\rho) \right) \mathcal{E}(\X_n)(dx,dy,d\rho).
\end{align*}
According to \emph{Step 2}, the first term goes to zero as $n \to \infty$ for any $R>0$. Due to the linear growth condition on $\tilde c_n, \tilde c$, the second term vanishes as $R \to 0$ uniformly in $n$. Therefore $I_n \to 0$ as $n \to \infty$.

For any small $\varepsilon$, choose a bounded continuous function $h$ such that $\lVert h - \tilde c \rVert_{\L^2(\nu^{\delta})} \leq \varepsilon$. Then we have
\begin{align*}
    \int \hat c(x,y,\rho) \cdot h(y) \left(\mathcal{E}(\X_n)-\mathcal{E}(\X) \right)(dx,dy,d\rho) \to& \  0, \\
   \left|\int \hat c(x,y,\rho) \cdot (h(y)-\tilde c) \left(\mathcal{E}(\X_n)-\mathcal{E}(\X) \right) (dx,dy,d\rho)  \right| \leq & \ C\varepsilon.
\end{align*}
Letting $\varepsilon \to 0$, $II_n\to 0$, and hence
\begin{align*}
    \liminf_{n \to \infty} B_n \geq B:=\E^{\P} \left[(\hat c-\tilde c)^2 \right].
\end{align*}

In the end, due to \eqref{eq:tildezero}, it suffices to estimate $\lim\limits_{n \to\infty} \E^{\P_n}\left[\hat c_n \cdot \partial_y \log(\rho_n )\right]$. Denote by $\mathcal{R}(\X_n)(dx,d\rho)$ the pushforward measure of $\mathcal{E}(\X_n)$
under $(x,y,\rho) \mapsto (x, \rho)$. As
\begin{align*}
    \int       \mathcal{R}(\X_n)(dx,d\rho)  \int \left|\hat c(x,y,\rho) \cdot \partial_y \log (\rho(y))\rho(y)\right| \,dy &=\E^{\P_n}\left[|\hat c_n \cdot \partial_y \log(\rho_n )|\right]<\infty, \\
    \int        \mathcal{R}(\X_n)(dx,d\rho)  \int \left|\nabla_y \cdot \hat c(x,y,\rho) \rho(y) \right| \,dy&= \E^{\P_n}\left[|\nabla_y \cdot \hat c_n |\right] < \infty,
\end{align*}
for $\mathcal{R}(\X_n)$-a.e. $(x,\rho)$, it holds that,
\begin{align*}
     \int \left|\hat c(x,y,\rho) \cdot \partial_y \log (\rho(y))\rho(y)\right| \,dy &< \infty, \\
      \int \left|\nabla_y \cdot \hat c(x,y,\rho) \rho(y) \right| \,dy &< \infty.
\end{align*}
Therefore we are allowed to integrate by parts, and hence get
\begin{align*}
\E^{\P_n}\left[\hat c_n \cdot \partial_y \log(\rho_n )\right] =& \int       \mathcal{R}(\X_n)(dx,d\rho)  \int \hat c(x,y,\rho) \cdot \partial_y \log (\rho(y))\rho(y) \,dy  \\
=&-\int        \mathcal{R}(\X_n)(dx,d\rho)  \int \nabla_y \cdot \hat c(x,y,\rho) \rho(y) \,dy.
\end{align*}
Note that $(x,\rho) \mapsto \int \nabla_y \cdot \hat c(x,y,\rho) \rho(y) \, dy$ is continuous. Therefore, by the weak convergence of $\mathcal{R}(\X_n) \to \mathcal{R}(\X)$, we conclude that $\lim\limits_{n\to\infty} \E^{\P_n}\left[\hat c_n \cdot \partial_y \log(\rho_n )\right]=\E^{\P}\left[\hat c \cdot \partial_y \log(\rho )\right]$, and hence
\begin{align*}
    \liminf_{n \to \infty} C_n \geq C:=2 \E^{\P} \left[ (\hat c -\tilde c) \cdot \partial_y \log\left( \frac{d \rho}{d \nu}(Y)  \right) \right],
\end{align*}
which completes the proof. \qed

\subsection{Proof of Proposition~\ref{prop:minimizer characterization}}

Let us start with an auxiliary lemma, which could be skipped at the first reading.

\begin{Lemma}\label{lem:positivedensity}
Let $D\subset \mathbb R^d$ be connected, and
\(
    u\in H^1_{\mathrm{loc}}(D)
\) with $u \geq 0$. Suppose that
\[
    \nabla u=b\,u
    \quad\text{a.e. on }D,
\]
where
\(
b\in L^\infty_{\mathrm{loc}}(D;\mathbb R^d).
\)
Then either $u=0$ a.e., or
$u>0$ \, a.e. on $D$.

\end{Lemma}

\begin{proof}
Fix a compact ball $B\subset D$. For $\delta>0$, define
\[
    v_\delta:=\log(u+\delta).
\]
Since $u\in H^1_{\mathrm{loc}}(D)$ and $r\mapsto \log(r+\delta)$ is Lipschitz on
$[0,\infty)$, we have $v_\delta\in H^1(B)$. Moreover,
\[
    \nabla v_\delta
    =
    \frac{\nabla u}{u+\delta}
    =
    b\,\frac{u}{u+\delta}.
\]
Hence $|\nabla v_\delta|\le |b| \  \text{a.e. on }B$.
Since $b\in L^\infty(B)$, it follows that $v_\delta\in W^{1,\infty}(B)$, and hence
\begin{align}\label{eq:oscupper}
    \operatorname*{ess\,osc}_{B} v_\delta
    \le
    C_B \|b\|_{L^\infty(B)},
\end{align}
where $C_B$ depends only on the ball $B$. Importantly, this bound is independent
of $\delta$.

We now show that, on $B$, the function $u$ cannot vanish on a set of positive
measure and be positive on another set of positive measure. Otherwise, $|\{u=0\}\cap B|>0$ and there exists $r>0$ such that $|\{u>r\}\cap B|>0$. On $\{u=0\}\cap B$, we have $v_\delta=\log \delta$, while on $\{u>r\}\cap B$,
$v_\delta\ge \log(r+\delta)$. Hence
\[
    \operatorname*{ess\,osc}_{B} v_\delta
    \ge
    \log(r+\delta)-\log\delta.
\]
Letting $\delta\downarrow0$, the right-hand side diverges to $+\infty$, which
contradicts \eqref{eq:oscupper}.
\[
    \operatorname*{ess\,osc}_{B} v_\delta
    \le
    C_B \|b\|_{L^\infty(B)}.
\]
Therefore, for every compact ball $B\subset D$, either
$u=0\ \text{a.e.}$, or $u>0\ \text{a.e. on }B$. Since $D$ is connected, this local dichotomy implies the global dichotomy, which completes the proof.
\end{proof}

\begin{proof}[Proof of Proposition~\ref{prop:minimizer characterization}]

\noindent \emph{Step 1.} Let us consider the relaxed version of \eqref{eq:WOT}. For any $\X \in \rm FP(\mu,\nu)$, define
\begin{align*}
    \bar J(\X):= \E^{\P^{\X}}\left[c(X, \Lc(Y| \Fc_1)) + H(\Lc(Y|\Fc_1) \| \nu) \right],
\end{align*}
and then according to \cite{BPRS25}, as $\rho \mapsto c(x,\rho)$ is convex, there exists a unique minimizer to
\begin{align*}
    \inf_{\X \in \rm FP(\mu,\nu)} \bar J(\X),
\end{align*}
which coincides with the unique minimizer to \eqref{eq:WOT}.

Let us denote $\eta=\mathcal{R}(\X) \in \Pc(\R^d \times \Pc(\R^d))$, $ f_{\rho}(y):= \frac{d \rho}{d \nu}(y)$, and
\begin{align}\label{eq:linearderivative}
 \Psi(x,y,\rho):= =\delta_m c(x,y,\rho)+\log f_{\rho}(y).
\end{align}
For the fixed $\X$, with $\pi=\mathcal{L}^{\P^{\X}}(X,Y)$, $\tilde c^{\pi}(\cdot)$  is a function of $y$. From $\tilde I(\X)=0$, it follows that for $\eta$-a.e. $(x,\rho)$,
\begin{equation}\label{eq:rhoequation}
\log f_{\rho}(y)=\tilde c^{\pi}(y)-\hat c(x,y,\rho), \quad \rho\text{-}a.e. \ y.
\end{equation}

\medskip
\noindent\emph{Step 2.} Let us prove that there exist weakly differentiable $g:\R^d \to \R$ and measurable $a:\R^d \times \Pc(\R^d) \to \R$ such that for $\eta$-a.e.$(x,\rho)$
\begin{align}\label{eq:ydependence}
    \Psi(x,y,\rho)=g(y)+a(x,\rho), \quad \forall \, y \in \R^d.
\end{align}

Denote $u_{\rho}= \sqrt{f_{\rho}}$. From \emph{Step 1} of Proposition~\ref{prop:lsc}, for $\eta$-a.e. $(x,\rho)$, it holds that $I(\rho) < \infty$, which implies that $u_{\rho} \in H^1(\R^d)$. According to \eqref{eq:rhoequation}, on the set $\{u_{\rho}>0\}$,
\[
    \nabla u_{\rho}(y)
    =
    \frac12
    \bigl(
       \tilde c^{\pi}(y)-\hat c(x,y,\rho)
    \bigr)
    u_{x,\rho}(y)
\]
Moreover, since $u_{\rho}\in H^1(\R^d)$, according to \cite[Chapter 9, p. 314, item 4]{BR11} we have
\[
    \nabla u_{\rho}(y)=0
    \quad
    \text{a.e. on }\{u_{\rho}=0\}.
\]
Therefore
\[
    \nabla u_{x,\rho}(y)
    =
    \frac12
    \bigl(
        \tilde c^{\pi}(y)-\hat c(x,y,\rho)
    \bigr)
    u_{\rho}(y)
    \quad
    \text{a.e. on }\mathbb R^d.
\]

Since $\tilde c^{\pi}(\cdot) -\hat c(x,\cdot,\rho) \in L^\infty_{\mathrm{loc}}(\mathbb R^d)$, by Lemma~\ref{lem:positivedensity}, we have either $u_{\rho}=0$ a.e. or $u_{\rho} > 0$ a.e. on $\R^d$.
The first alternative is impossible because $\rho$ is a probability measure. Therefore $\rho$ is equivalent to the Lebesgue measure.

Together with \eqref{eq:rhoequation}, we get that
for $\eta$-a.e. $(x,\rho)$
\[
\nabla_y \Psi(x,y,\rho)=\tilde c^{\pi}(y), \quad \mathrm{L}\text{-a.e. } y.
\]
Lebesgue-a.e. in $y$, for $\eta$-a.e. $(x,\rho)$.

Let $G\subset \mathbb R^d\times\mathcal P(\mathbb R^d)$ be the full
$\eta$-measure set on which the above identity holds. Choose one element
$(x_0,\rho_0)\in G$ and define
\[
    h(y):=\Psi(x_0,y,\rho_0).
\]
Then for every $(x,\rho) \in G$,
\[
    \nabla_y\bigl(\Psi(x,y,\rho)-h(y)\bigr)=0
    \quad \text{a.e.} \text{ in } \R^d.
\]
Since $\mathbb R^d$ is connected, every weakly differentiable function with zero
weak gradient is a constant. Hence, for every $(x,\rho)\in G$, there exists a
constant $a(x,\rho)$ such that
\[
    \Psi(x,y,\rho)-h(y)=a(x,\rho), \quad \forall y \in \R^d,
\]
which proves \eqref{eq:ydependence}.

\medskip

\noindent \emph{Step 3.} Let us prove that $\eqref{eq:ydependence}$ implies the optimality of $\X$. Let \(\tilde \X \in \rm FP(\mu,\nu)\) be any competitor, and define $\tilde \eta =R (\tilde \X )$.
Disintegrate \(\eta\) and \(\tilde\eta\) with respect to the first marginal \(\mu\),
\[
\eta(dx,d\rho)=\mu(dx)\,K_x(d\rho), \qquad
\tilde\eta(dx,d\tilde\rho)=\mu(dx)\,\tilde K_x(d\tilde\rho).
\]
Then define a coupling
\[
\Gamma(dx,d\rho,d\tilde\rho):=\mu(dx)\,K_x(d\rho)\,\tilde K_x(d\tilde\rho),
\]
for which \((x,\rho)\)-marginal is \(\eta\) and \((x,\tilde\rho)\)-marginal is \(\tilde\eta\). Therefore,
\[
\bar J(\tilde\X)-\bar J(\X)
=
\int \big(F(x,\tilde\rho)-F(x,\rho)\big)\,\Gamma(dx,d\rho,d\tilde\rho),
\]
where $F(x,\rho):=c(x,\rho)+H(\rho \| \nu)$.

By the convexity of $\rho \mapsto F(x,\rho)$, recalling $\Psi$ defined in \eqref{eq:linearderivative}
\[
F(x,\tilde\rho)-F(x,\rho)
\ge
\int \Psi(x,y,\rho)\,(\tilde\rho-\rho)(dy).
\]
It follows that
\[
\bar J(\tilde\X)-\bar J(\X)
\ge
\int \left[\int \Psi(x,\rho,y)\,(\tilde\rho-\rho)(dy)\right]
\,\Gamma(dx,d\rho,d\tilde\rho).
\]
According to \eqref{eq:ydependence}, the right-hand side of the above inequality equals
\begin{align*}
   & \int \left[\int h(y)\,\tilde\rho(dy)\right]\Gamma(dx,d\rho,d\tilde\rho)
-
\int \left[\int h(y)\,\rho(dy)\right]\Gamma(dx,d\rho,d\tilde\rho)=0,
\end{align*}
where we have used that $\int \rho \, \Gamma(dx,d\rho,d\tilde \rho)=\int \tilde \rho \, \Gamma(dx,d\rho,d\tilde \rho)=\nu$. Therefore $\bar J(\tilde \X) \geq \bar J(\X)$ for any $\tilde \X \in \rm FP(\mu,\nu)$, which completes the proof.
\end{proof}

\bibliographystyle{siam}
\bibliography{Biblio}

\end{document}